\documentclass[10pt, twoside]{amsart}
\usepackage{amsfonts, amsmath, amssymb}
\usepackage{tikz}
\usepackage{tqft}

\usepackage[all,2cell,knot]{xy} \UseTwocells
\input{xy}
\xyoption{all}
\xyoption{arc} 
\def\mapright#1{\smash{\mathop{\longrightarrow}\limits^{#1}}}

\newtheorem{proposition}{Proposition}[section]
\newtheorem{lemma}[proposition]{Lemma}
\newtheorem{corollary}[proposition]{Corollary}
\newtheorem{theorem}[proposition]{Theorem}
\theoremstyle{definition}
\newtheorem{definition}[proposition]{Definition}

\theoremstyle{remark}
\newtheorem{remark}[proposition]{Remark}

\newcommand{\thlabel}[1]{\label{th:#1}}
\newcommand{\thref}[1]{Theorem~\ref{th:#1}}
\newcommand{\selabel}[1]{\label{se:#1}}
\newcommand{\seref}[1]{Section~\ref{se:#1}}
\newcommand{\lelabel}[1]{\label{le:#1}}
\newcommand{\leref}[1]{Lemma~\ref{le:#1}}
\newcommand{\prlabel}[1]{\label{pr:#1}}
\newcommand{\prref}[1]{Proposition~\ref{pr:#1}}
\newcommand{\colabel}[1]{\label{co:#1}}
\newcommand{\coref}[1]{Corollary~\ref{co:#1}}

\newcommand{\delabel}[1]{\label{de:#1}}
\newcommand{\deref}[1]{Definition~\ref{de:#1}}

\newcommand{\eqlabel}[1]{\label{eq:#1}}
\newcommand{\equref}[1]{(\ref{eq:#1})}

\def\ra{\rightarrow}

\def\Id{{\rm Id}}
\def\mfM{{\mf M}}
\def\mfN{\mf {N}}
\def\mfP{\mf {P}}

\def\ot{\otimes}
\def\va{\varepsilon}

\def\un{\underline}

\def\mf{\mathfrak}

\def\l{\lambda}

\def\r{\rho}

\def\va{\varepsilon}

\def\ra{\rightarrow}
\def\a{\alpha}
\def\b{\beta}

\def\d{\delta}

\def\ov{\overline}
\def\cal{\mathcal}
\def\un{\underline}

\newcommand{\Cc}{\cal C}
\newcommand{\Dc}{\cal D}

\def\equal#1{\smash{\mathop{=}\limits^{#1}}}

 \newcount\grcalca
 \newcount\grcalcb
 \newcount\grcalcc
 \newcount\grcalcd
 \newcount\grcalce
 \newcount\grcalcf
 \newcount\grcalcg
 \newcount\grcalch
 \newcount\grcalci
 \newcount\grrow
 \newcount\grcolumn
 \newcount\grwidth
 \newcount\factor
 \newcount\hfactor
 \newcount\qfactor
 \newcount\tfactor
 \newcount\sfactor
 \newcount\dfactor
 \hfactor = \factor
 \divide    \hfactor by 2
 \qfactor = \hfactor
 \divide    \qfactor by 2
 \tfactor = \factor
 \divide    \tfactor by 3
 \sfactor = \factor
 \divide    \sfactor by 6
 \dfactor = \factor
 \divide    \dfactor by 12

 \newcommand{\gbeg}[2]{
   \unitlength=1pt
   \grrow = #2
   \grcolumn = 0
   \grcalca = #1
   \grcalcb = #2
   \multiply \grcalca by \factor
   \grwidth = \grcalca
   \multiply \grcalcb by \factor
   \begin{minipage}{\grcalca pt}
   \begin{picture}(\grcalca,\grcalcb)
   \advance \grcalcb by -\factor
   \put(0, \grcalcb){\line(1,0){\grwidth}} }

 \newcommand{\gend}{
   \put(0, \factor){\line(1,0){\grwidth}}
   \end{picture}
   {\vskip2.5ex}
   \end{minipage} }

 \newcommand{\gnl}{
   \advance \grrow by -1
   \grcolumn = 0}

 \newcommand{\gvac}[1]{       
   \advance \grcolumn by #1} 
 
 \newcommand{\gcl}[1]{
   \grcalca = \grcolumn
   \multiply \grcalca by \factor
   \advance \grcalca by \hfactor
   \grcalcb = \grrow
   \multiply \grcalcb by \factor
   \grcalcc = #1
   \multiply \grcalcc by \factor
   \put(\grcalca,\grcalcb) {\line(0,-1){\grcalcc}} 
   \advance \grcolumn by 1}

 \newcommand{\gcn}[4]{
   \grcalca = \grcolumn
   \multiply \grcalca by \factor
   \grcalci = #3
   \multiply \grcalci by \hfactor
   \advance \grcalca by \grcalci
   \grcalcb = \grcolumn
   \multiply \grcalcb by \factor 
   \grcalci = #3
   \advance \grcalci by #4
   \multiply \grcalci by \qfactor
   \advance \grcalcb by \grcalci
   \grcalcc = \grcolumn
   \multiply \grcalcc by \factor
   \grcalci = #4
   \multiply \grcalci by \hfactor
   \advance \grcalcc by \grcalci
   \grcalcd = \grrow
   \multiply \grcalcd by \factor 
   \grcalce = \grrow
   \multiply \grcalce by \factor 
   \grcalci = #2
   \multiply \grcalci by \tfactor
   \advance \grcalce by -\grcalci
   \grcalcf = \grrow
   \multiply \grcalcf by \factor 
   \grcalci = #2
   \multiply \grcalci by \hfactor
   \advance \grcalcf by -\grcalci
   \grcalcg = \grrow
   \multiply \grcalcg by \factor 
   \grcalci = #2
   \multiply \grcalci by \tfactor
   \multiply \grcalci by 2
   \advance \grcalcg by -\grcalci
   \grcalch = \grrow
   \advance \grcalch by -#2
   \multiply \grcalch by \factor 
   \qbezier(\grcalca,\grcalcd)(\grcalca,\grcalce)(\grcalcb,\grcalcf) 
   \qbezier(\grcalcb,\grcalcf)(\grcalcc,\grcalcg)(\grcalcc,\grcalch) 
   \advance \grcolumn by #1}

 \newcommand{\gnot}[1]{
   \grcalca = \grcolumn
   \multiply \grcalca by \factor
   \advance \grcalca by \hfactor
   \grcalcb = \grrow
   \multiply \grcalcb by \factor
   \advance \grcalcb by -\hfactor
   \put(\grcalca,\grcalcb) {\makebox(0,0){$\scriptstyle #1$}} }

 \newcommand{\got}[2]{
   \grcalca = \grcolumn
   \multiply \grcalca by \factor
   \grcalcc = #1
   \multiply \grcalcc by \hfactor
   \advance \grcalca by \grcalcc
   \grcalcb = \grrow
   \multiply \grcalcb by \factor
   \advance \grcalcb by -\tfactor
   \advance \grcalcb by -\tfactor
   \put(\grcalca,\grcalcb){\makebox(0,0)[b]{$#2$}}
   \advance \grcolumn by #1}

 \newcommand{\gob}[2]{
   \grcalca = \grcolumn
   \multiply \grcalca by \factor
   \grcalcc = #1
   \multiply \grcalcc by \hfactor
   \advance \grcalca by \grcalcc
   \put(\grcalca,0){\makebox(0,0)[b]{$#2$}}
   \advance \grcolumn by #1}

 \newcommand{\gmu}{  
   \grcalca = \grcolumn
   \advance \grcalca by 1
   \multiply \grcalca by \factor
   \grcalcb = \grrow
   \multiply \grcalcb by \factor
   \grcalcc = \factor
   \advance \grcalcc by \hfactor
   \put(\grcalca,\grcalcb){\oval(\factor,\grcalcc)[b]}
   \advance \grcalcb by -\hfactor
   \advance \grcalcb by -\qfactor
   \put(\grcalca,\grcalcb) {\line(0,-1){\qfactor}} 
   \advance \grcolumn by 2}

 \newcommand{\gcmu}{   
   \grcalca = \grcolumn
   \advance \grcalca by 1
   \multiply \grcalca by \factor
   \grcalcb = \grrow
   \advance \grcalcb by -1
   \multiply \grcalcb by \factor
   \grcalcc = \factor
   \advance \grcalcc by \hfactor
   \put(\grcalca,\grcalcb){\oval(\factor,\grcalcc)[t]}
   \advance \grcalcb by \factor
   \put(\grcalca,\grcalcb) {\line(0,-1){\qfactor}} 
   \advance \grcolumn by 2}

 \newcommand{\glm}{
   \grcalca = \grcolumn
   \multiply \grcalca by \factor
   \advance \grcalca by \hfactor
   \grcalcb = \grcalca
   \advance \grcalcb by \factor
   \grcalcc = \grrow
   \multiply \grcalcc by \factor
   \grcalcd = \grcalcc
   \advance \grcalcd by -\tfactor
   \grcalce = \grcalcd
   \advance \grcalce by -\tfactor
   \put(\grcalca, \grcalcc){\line(0,-1){\tfactor}}
   \put(\grcalca, \grcalcd){\line(1,0){\factor}}
   \put(\grcalca, \grcalcd){\line(3,-1){\factor}}
   \put(\grcalcb, \grcalcc){\line(0,-1){\factor}}
   \advance \grcolumn by 2}

 \newcommand{\grm}{
   \grcalcb = \grcolumn
   \multiply \grcalcb by \factor
   \advance \grcalcb by \hfactor
   \grcalca = \grcalcb
   \advance \grcalca by \factor
   \grcalcc = \grrow
   \multiply \grcalcc by \factor
   \grcalcd = \grcalcc
   \advance \grcalcd by -\tfactor
   \grcalce = \grcalcd
   \advance \grcalce by -\tfactor
   \put(\grcalca, \grcalcc){\line(0,-1){\tfactor}}
   \put(\grcalca, \grcalcd){\line(-1,0){\factor}}
   \put(\grcalca, \grcalcd){\line(-3,-1){\factor}}
   \put(\grcalcb, \grcalcc){\line(0,-1){\factor}}
   \advance \grcolumn by 2}

 \newcommand{\glcm}{
   \grcalca = \grcolumn
   \multiply \grcalca by \factor
   \advance \grcalca by \hfactor
   \grcalcb = \grcalca
   \advance \grcalcb by \factor
   \grcalcc = \grrow
   \advance \grcalcc by -1
   \multiply \grcalcc by \factor
   \grcalcd = \grcalcc
   \advance \grcalcd by \tfactor
   \grcalce = \grcalcd
   \advance \grcalce by \tfactor
   \put(\grcalca, \grcalcc){\line(0,1){\tfactor}}
   \put(\grcalca, \grcalcd){\line(1,0){\factor}}
   \put(\grcalca, \grcalcd){\line(3,1){\factor}}
   \put(\grcalcb, \grcalcc){\line(0,1){\factor}}
   \advance \grcolumn by 2}

 \newcommand{\grcm}{
   \grcalcb = \grcolumn
   \multiply \grcalcb by \factor
   \advance \grcalcb by \hfactor
   \grcalca = \grcalcb
   \advance \grcalca by \factor
   \grcalcc = \grrow
   \advance \grcalcc by -1
   \multiply \grcalcc by \factor
   \grcalcd = \grcalcc
   \advance \grcalcd by \tfactor
   \grcalce = \grcalcd
   \advance \grcalce by \tfactor
   \put(\grcalca, \grcalcc){\line(0,1){\tfactor}}
   \put(\grcalca, \grcalcd){\line(-1,0){\factor}}
   \put(\grcalca, \grcalcd){\line(-3,1){\factor}}
   \put(\grcalcb, \grcalcc){\line(0,1){\factor}}
   \advance \grcolumn by 2}

 \newcommand{\gwmu}[1]{    
   \grcalca = \grcolumn
   \multiply \grcalca by \factor
   \grcalcd = \hfactor
   \multiply \grcalcd by #1
   \advance \grcalca by \grcalcd
   \grcalcb = \grrow
   \multiply \grcalcb by \factor
   \grcalcc = \factor
   \advance \grcalcc by \hfactor
   \grcalcd = #1
   \advance \grcalcd by -1
   \multiply \grcalcd by \factor
   \put(\grcalca,\grcalcb){\oval(\grcalcd,\grcalcc)[b]}
   \advance \grcalcb by -\hfactor
   \advance \grcalcb by -\qfactor
   \put(\grcalca,\grcalcb) {\line(0,-1){\qfactor}} 
   \advance \grcolumn by #1}

 \newcommand{\gwcm}[1]{   
   \grcalca = \grcolumn
   \multiply \grcalca by \factor
   \grcalcd = \hfactor
   \multiply \grcalcd by #1
   \advance \grcalca by \grcalcd
   \grcalcb = \grrow
   \advance \grcalcb by -1
   \multiply \grcalcb by \factor
   \grcalcc = \factor
   \advance \grcalcc by \hfactor
   \grcalcd = #1
   \advance \grcalcd by -1
   \multiply \grcalcd by \factor
   \put(\grcalca,\grcalcb){\oval(\grcalcd,\grcalcc)[t]}
   \advance \grcalcb by \factor
   \put(\grcalca,\grcalcb) {\line(0,-1){\qfactor}} 
   \advance \grcolumn by #1}

 \newcommand{\gwmuc}[1]{    
   \grcalca = \grcolumn
   \multiply \grcalca by \factor
   \advance \grcalca by \hfactor
   \grcalcb = \grrow
   \multiply \grcalcb by \factor
   \grcalcc = #1
   \advance \grcalcc by -1
   \multiply \grcalcc by \factor
   \put(\grcalca,\grcalcb){\line(1,0){\grcalcc}}
   \advance \grcalca by -\hfactor
   \grcalcd = \hfactor
   \multiply \grcalcd by #1
   \advance \grcalca by \grcalcd
   \grcalcc = \factor
   \advance \grcalcc by \hfactor
   \grcalcd = #1
   \advance \grcalcd by -1
   \multiply \grcalcd by \factor
   \put(\grcalca,\grcalcb){\oval(\grcalcd,\grcalcc)[b]}
   \advance \grcalcb by -\hfactor
   \advance \grcalcb by -\qfactor
   \put(\grcalca,\grcalcb) {\line(0,-1){\qfactor}} 
   \advance \grcolumn by #1}

 \newcommand{\gwcmc}[1]{   
   \grcalca = \grcolumn
   \multiply \grcalca by \factor
   \advance \grcalca by \hfactor
   \grcalcb = \grrow
   \multiply \grcalcb by \factor
   \advance \grcalcb by -\factor
   \grcalcc = #1
   \advance \grcalcc by -1
   \multiply \grcalcc by \factor
   \put(\grcalca,\grcalcb){\line(1,0){\grcalcc}}
   \grcalcd = #1
   \advance \grcalcd by -1
   \multiply \grcalcd by \hfactor
   \advance \grcalca by \grcalcd
   \grcalcc = \factor
   \advance \grcalcc by \hfactor
   \grcalcd = #1
   \advance \grcalcd by -1
   \multiply \grcalcd by \factor
   \put(\grcalca,\grcalcb){\oval(\grcalcd,\grcalcc)[t]}
   \advance \grcalcb by \factor
   \put(\grcalca,\grcalcb) {\line(0,-1){\qfactor}} 
   \advance \grcolumn by #1}

 \newcommand{\gev}{  
   \grcalca = \grcolumn
   \advance \grcalca by 1
   \multiply \grcalca by \factor
   \grcalcb = \grrow
   \multiply \grcalcb by \factor
   \grcalcc = \factor
   \advance \grcalcc by \hfactor
   \put(\grcalca,\grcalcb){\oval(\factor,\grcalcc)[b]}
   \advance \grcolumn by 2}

 \newcommand{\gdb}{   
   \grcalca = \grcolumn
   \advance \grcalca by 1
   \multiply \grcalca by \factor
   \grcalcb = \grrow
   \advance \grcalcb by -1
   \multiply \grcalcb by \factor
   \grcalcc = \factor
   \advance \grcalcc by \hfactor
   \put(\grcalca,\grcalcb){\oval(\factor,\grcalcc)[t]}
   \advance \grcolumn by 2}

 \newcommand{\gwev}[1]{    
   \grcalca = \grcolumn
   \multiply \grcalca by \factor
   \grcalcd = \hfactor
   \multiply \grcalcd by #1
   \advance \grcalca by \grcalcd
   \grcalcb = \grrow
   \multiply \grcalcb by \factor
   \grcalcc = \factor
   \advance \grcalcc by \hfactor
   \grcalcd = #1
   \advance \grcalcd by -1
   \multiply \grcalcd by \factor
   \put(\grcalca,\grcalcb){\oval(\grcalcd,\grcalcc)[b]}
   \advance \grcolumn by #1}

 \newcommand{\gwdb}[1]{   
   \grcalca = \grcolumn
   \multiply \grcalca by \factor
   \grcalcd = \hfactor
   \multiply \grcalcd by #1
   \advance \grcalca by \grcalcd
   \grcalcb = \grrow
   \advance \grcalcb by -1
   \multiply \grcalcb by \factor
   \grcalcc = \factor
   \advance \grcalcc by \hfactor
   \grcalcd = #1
   \advance \grcalcd by -1
   \multiply \grcalcd by \factor
   \put(\grcalca,\grcalcb){\oval(\grcalcd,\grcalcc)[t]}
   \advance \grcolumn by #1}

 \newcommand{\gbr}{
   \grcalca = \grcolumn
   \multiply \grcalca by \factor
   \advance \grcalca by \hfactor
   \grcalcb = \grcalca
   \advance \grcalcb by \hfactor
   \grcalcc = \grcalca
   \advance \grcalcc by \factor
   \grcalcd = \grrow
   \multiply \grcalcd by \factor
   \grcalce = \grcalcd
   \advance \grcalce by -\tfactor
   \grcalcf = \grcalcd
   \advance \grcalcf by -\hfactor
   \grcalcg = \grcalce
   \advance \grcalcg by -\tfactor
   \grcalch = \grcalcd
   \advance \grcalch by -\factor
   \qbezier(\grcalca,\grcalcd)(\grcalca,\grcalce)(\grcalcb,\grcalcf) 
   \qbezier(\grcalcb,\grcalcf)(\grcalcc,\grcalcg)(\grcalcc,\grcalch) 
   \advance \grcalcf by -\dfactor
   \advance \grcalcb by -\sfactor
   \qbezier(\grcalca,\grcalch)(\grcalca,\grcalcg)(\grcalcb,\grcalcf) 
   \advance \grcalcf by \sfactor
   \advance \grcalcb by \tfactor
   \qbezier(\grcalcc,\grcalcd)(\grcalcc,\grcalce)(\grcalcb,\grcalcf) 
   \advance \grcolumn by 2}

 \newcommand{\gibr}{
   \grcalca = \grcolumn
   \multiply \grcalca by \factor
   \advance \grcalca by \hfactor
   \grcalcb = \grcalca
   \advance \grcalcb by \hfactor
   \grcalcc = \grcalca
   \advance \grcalcc by \factor
   \grcalcd = \grrow
   \multiply \grcalcd by \factor
   \grcalce = \grcalcd
   \advance \grcalce by -\tfactor
   \grcalcf = \grcalcd
   \advance \grcalcf by -\hfactor
   \grcalcg = \grcalce
   \advance \grcalcg by -\tfactor
   \grcalch = \grcalcd
   \advance \grcalch by -\factor
   \qbezier(\grcalcc,\grcalcd)(\grcalcc,\grcalce)(\grcalcb,\grcalcf) 
   \qbezier(\grcalcb,\grcalcf)(\grcalca,\grcalcg)(\grcalca,\grcalch) 
   \advance \grcalcf by -\dfactor
   \advance \grcalcb by \sfactor
   \qbezier(\grcalcc,\grcalch)(\grcalcc,\grcalcg)(\grcalcb,\grcalcf) 
   \advance \grcalcf by \sfactor
   \advance \grcalcb by -\tfactor
   \qbezier(\grcalca,\grcalcd)(\grcalca,\grcalce)(\grcalcb,\grcalcf) 
   \advance \grcolumn by 2}

\newcommand{\gsy}{
   \grcalca = \grcolumn
   \multiply \grcalca by \factor
   \advance \grcalca by \hfactor
   \grcalcb = \grcalca
   \advance \grcalcb by \hfactor
   \grcalcc = \grcalca
   \advance \grcalcc by \factor
   \grcalcd = \grrow
   \multiply \grcalcd by \factor
   \grcalce = \grcalcd
   \advance \grcalce by -\tfactor
   \grcalcf = \grcalcd
   \advance \grcalcf by -\hfactor
   \grcalcg = \grcalce
   \advance \grcalcg by -\tfactor
   \grcalch = \grcalcd
   \advance \grcalch by -\factor
   \qbezier(\grcalcc,\grcalcd)(\grcalcc,\grcalce)(\grcalcb,\grcalcf) 
   \qbezier(\grcalcb,\grcalcf)(\grcalca,\grcalcg)(\grcalca,\grcalch) 
   \advance \grcalcf by -\dfactor
   \advance \grcalcb by \sfactor
   \qbezier(\grcalcc,\grcalch)(\grcalcc,\grcalcg)(\grcalcb,\grcalcf) 
   \qbezier(\grcalca,\grcalcd)(\grcalca,\grcalce)(\grcalcb,\grcalcf) 
   \advance \grcolumn by 2}

 \newcommand{\gbrc}{
   \grcalca = \grcolumn
   \multiply \grcalca by \factor
   \advance \grcalca by \hfactor
   \grcalcb = \grcalca
   \advance \grcalcb by \hfactor
   \grcalcc = \grcalca
   \advance \grcalcc by \factor
   \grcalcd = \grrow
   \multiply \grcalcd by \factor
   \grcalce = \grcalcd
   \advance \grcalce by -\tfactor
   \grcalcf = \grcalcd
   \advance \grcalcf by -\hfactor
   \grcalcg = \grcalce
   \advance \grcalcg by -\tfactor
   \grcalch = \grcalcd
   \advance \grcalch by -\factor
   \put(\grcalcb,\grcalcf){\circle{\hfactor}}
   \qbezier(\grcalca,\grcalcd)(\grcalca,\grcalce)(\grcalcb,\grcalcf) 
   \qbezier(\grcalcb,\grcalcf)(\grcalcc,\grcalcg)(\grcalcc,\grcalch) 
   \advance \grcalcf by -\dfactor
   \advance \grcalcb by -\sfactor
   \qbezier(\grcalca,\grcalch)(\grcalca,\grcalcg)(\grcalcb,\grcalcf) 
   \advance \grcalcf by \sfactor
   \advance \grcalcb by \tfactor
   \qbezier(\grcalcc,\grcalcd)(\grcalcc,\grcalce)(\grcalcb,\grcalcf) 
   \advance \grcolumn by 2}

 \newcommand{\gibrc}{
   \grcalca = \grcolumn
   \multiply \grcalca by \factor
   \advance \grcalca by \hfactor
   \grcalcb = \grcalca
   \advance \grcalcb by \hfactor
   \grcalcc = \grcalca
   \advance \grcalcc by \factor
   \grcalcd = \grrow
   \multiply \grcalcd by \factor
   \grcalce = \grcalcd
   \advance \grcalce by -\tfactor
   \grcalcf = \grcalcd
   \advance \grcalcf by -\hfactor
   \grcalcg = \grcalce
   \advance \grcalcg by -\tfactor
   \grcalch = \grcalcd
   \advance \grcalch by -\factor
   \put(\grcalcb,\grcalcf){\circle{\hfactor}}
   \qbezier(\grcalcc,\grcalcd)(\grcalcc,\grcalce)(\grcalcb,\grcalcf) 
   \qbezier(\grcalcb,\grcalcf)(\grcalca,\grcalcg)(\grcalca,\grcalch) 
   \advance \grcalcf by -\dfactor
   \advance \grcalcb by \sfactor
   \qbezier(\grcalcc,\grcalch)(\grcalcc,\grcalcg)(\grcalcb,\grcalcf) 
   \advance \grcalcf by \sfactor
   \advance \grcalcb by -\tfactor
   \qbezier(\grcalca,\grcalcd)(\grcalca,\grcalce)(\grcalcb,\grcalcf) 
   \advance \grcolumn by 2}

 \newcommand{\gu}[1]{
   \grcalca = \grcolumn
   \multiply \grcalca by \factor
   \grcalcd = \hfactor
   \multiply \grcalcd by #1
   \advance \grcalca by \grcalcd
   \grcalcb = \grrow
   \advance \grcalcb by -1
   \multiply \grcalcb by \factor
   \put(\grcalca,\grcalcb) {\line(0,1){\hfactor}} 
   \advance \grcalcb by \hfactor
   \put(\grcalca,\grcalcb) {\circle*{3}}
   \advance \grcolumn by #1}

 \newcommand{\gcu}[1]{
   \grcalca = \grcolumn
   \multiply \grcalca by \factor
   \grcalcd = \hfactor
   \multiply \grcalcd by #1
   \advance \grcalca by \grcalcd
   \grcalcb = \grrow
   \multiply \grcalcb by \factor
   \put(\grcalca,\grcalcb) {\line(0,-1){\hfactor}} 
   \advance \grcalcb by -\hfactor
   \put(\grcalca,\grcalcb) {\circle*{3}}
   \advance \grcolumn by #1}

 \newcommand{\gmp}[1]{
   \grcalca = \grcolumn
   \multiply \grcalca by \factor
   \advance \grcalca by \hfactor
   \grcalcb = \grrow
   \multiply \grcalcb by \factor
   \put(\grcalca,\grcalcb) {\line(0,-1){\dfactor}} 
   \advance \grcalcb by -\factor
   \put(\grcalca,\grcalcb) {\line(0,1){\dfactor}} 
   \advance \grcalcb by \hfactor
   \grcalcc = \factor
   \advance \grcalcc by -\qfactor
   \put(\grcalca,\grcalcb) {\circle{\grcalcc}}
   \put(\grcalca,\grcalcb) {\makebox(0,0){$\scriptstyle #1$}}
   \advance \grcolumn by 1}
   
\newcommand{\gmpcu}[1]{
   \grcalca = \grcolumn
   \multiply \grcalca by \factor
   \advance \grcalca by \hfactor
   \grcalcb = \grrow
   \multiply \grcalcb by \factor
   \put(\grcalca,\grcalcb) {\line(0,-1){\dfactor}} 
   \advance \grcalcb by -\factor
   \advance \grcalcb by \hfactor
   \grcalcc = \factor
   \advance \grcalcc by -\qfactor
   \put(\grcalca,\grcalcb) {\circle{\grcalcc}}
   \put(\grcalca,\grcalcb) {\makebox(0,0){$\scriptstyle #1$}}
   \advance \grcolumn by 1}
   
\newcommand{\gmpu}[1]{
   \grcalca = \grcolumn
   \multiply \grcalca by \factor
   \advance \grcalca by \hfactor
   \grcalcb = \grrow
   \multiply \grcalcb by \factor
   \advance \grcalcb by -\factor
   \put(\grcalca,\grcalcb) {\line(0,1){\dfactor}} 
   \advance \grcalcb by \hfactor
   \grcalcc = \factor
   \advance \grcalcc by -\qfactor
   \put(\grcalca,\grcalcb) {\circle{\grcalcc}}
   \put(\grcalca,\grcalcb) {\makebox(0,0){$\scriptstyle #1$}}
   \advance \grcolumn by 1}      
 \newcommand{\gbmp}[1]{
   \grcalca = \grcolumn
   \multiply \grcalca by \factor
   \advance \grcalca by \hfactor
   \grcalcb = \grrow
   \multiply \grcalcb by \factor
   \put(\grcalca,\grcalcb) {\line(0,-1){\dfactor}} 
   \advance \grcalcb by -\factor
   \put(\grcalca,\grcalcb) {\line(0,1){\dfactor}} 
   \advance \grcalca by -\hfactor
   \advance \grcalca by \dfactor
   \advance \grcalcb by \dfactor
   \grcalcc = \factor
   \advance \grcalcc by -\sfactor
   \put(\grcalca,\grcalcb) {\framebox(\grcalcc,\grcalcc){$\scriptstyle #1$}}
   \advance \grcolumn by 1}

 \newcommand{\gbmpt}[1]{
   \grcalca = \grcolumn
   \multiply \grcalca by \factor
   \advance \grcalca by \hfactor
   \grcalcb = \grrow
   \multiply \grcalcb by \factor
   \put(\grcalca,\grcalcb) {\line(0,-1){\dfactor}} 
   \advance \grcalcb by -\factor
   \advance \grcalca by -\hfactor
   \advance \grcalca by \dfactor
   \advance \grcalcb by \dfactor
   \grcalcc = \factor
   \advance \grcalcc by -\sfactor
   \put(\grcalca,\grcalcb) {\framebox(\grcalcc,\grcalcc){$\scriptstyle #1$}}
   \advance \grcolumn by 1}

 \newcommand{\gbmpb}[1]{
   \grcalca = \grcolumn
   \multiply \grcalca by \factor
   \advance \grcalca by \hfactor
   \grcalcb = \grrow
   \multiply \grcalcb by \factor
   \advance \grcalcb by -\factor
   \put(\grcalca,\grcalcb) {\line(0,1){\dfactor}} 
   \advance \grcalca by -\hfactor
   \advance \grcalca by \dfactor
   \advance \grcalcb by \dfactor
   \grcalcc = \factor
   \advance \grcalcc by -\sfactor
   \put(\grcalca,\grcalcb) {\framebox(\grcalcc,\grcalcc){$\scriptstyle #1$}}
   \advance \grcolumn by 1}

 \newcommand{\gbmpn}[1]{
   \grcalca = \grcolumn
   \multiply \grcalca by \factor
   \advance \grcalca by \hfactor
   \grcalcb = \grrow
   \multiply \grcalcb by \factor
   \advance \grcalcb by -\factor
   \advance \grcalca by -\hfactor
   \advance \grcalca by \dfactor
   \advance \grcalcb by \dfactor
   \grcalcc = \factor
   \advance \grcalcc by -\sfactor
   \put(\grcalca,\grcalcb) {\framebox(\grcalcc,\grcalcc){$\scriptstyle #1$}}
   \advance \grcolumn by 1}

 \newcommand{\glmptb}{    
   \grcalca = \grcolumn
   \multiply \grcalca by \factor
   \advance \grcalca by \hfactor
   \grcalcb = \grrow
   \multiply \grcalcb by \factor
   \put(\grcalca,\grcalcb) {\line(0,-1){\dfactor}} 
   \advance \grcalcb by -\factor
   \put(\grcalca,\grcalcb) {\line(0,1){\dfactor}} 
   \advance \grcalca by -\hfactor
   \advance \grcalca by \dfactor
   \advance \grcalcb by \dfactor
   \put(\grcalca,\grcalcb) {\line(1,0){\factor}} 
   \advance \grcalcb by \factor
   \advance \grcalcb by -\sfactor
   \put(\grcalca,\grcalcb) {\line(1,0){\factor}} 
   \grcalcc = \factor
   \advance \grcalcc by -\sfactor
   \put(\grcalca,\grcalcb) {\line(0,-1){\grcalcc}} 
   \advance \grcolumn by 1}

 \newcommand{\glmpt}{    
   \grcalca = \grcolumn
   \multiply \grcalca by \factor
   \advance \grcalca by \hfactor
   \grcalcb = \grrow
   \multiply \grcalcb by \factor
   \put(\grcalca,\grcalcb) {\line(0,-1){\dfactor}} 
   \advance \grcalca by -\hfactor
   \advance \grcalca by \dfactor
   \advance \grcalcb by -\dfactor
   \put(\grcalca,\grcalcb) {\line(1,0){\factor}} 
   \advance \grcalcb by -\factor
   \advance \grcalcb by \sfactor
   \put(\grcalca,\grcalcb) {\line(1,0){\factor}} 
   \grcalcc = \factor
   \advance \grcalcc by -\sfactor
   \put(\grcalca,\grcalcb) {\line(0,1){\grcalcc}} 
   \advance \grcolumn by 1}

 \newcommand{\glmpb}{    
   \grcalca = \grcolumn
   \multiply \grcalca by \factor
   \advance \grcalca by \hfactor
   \grcalcb = \grrow
   \multiply \grcalcb by \factor
   \advance \grcalcb by -\factor
   \put(\grcalca,\grcalcb) {\line(0,1){\dfactor}} 
   \advance \grcalca by -\hfactor
   \advance \grcalca by \dfactor
   \advance \grcalcb by \dfactor
   \put(\grcalca,\grcalcb) {\line(1,0){\factor}} 
   \advance \grcalcb by \factor
   \advance \grcalcb by -\sfactor
   \put(\grcalca,\grcalcb) {\line(1,0){\factor}} 
   \grcalcc = \factor
   \advance \grcalcc by -\sfactor
   \put(\grcalca,\grcalcb) {\line(0,-1){\grcalcc}} 
   \advance \grcolumn by 1}

 \newcommand{\glmp}{    
   \grcalca = \grcolumn
   \multiply \grcalca by \factor
   \advance \grcalca by \dfactor
   \grcalcb = \grrow
   \multiply \grcalcb by \factor
   \advance \grcalcb by -\dfactor
   \put(\grcalca,\grcalcb) {\line(1,0){\factor}} 
   \advance \grcalcb by -\factor
   \advance \grcalcb by \sfactor
   \put(\grcalca,\grcalcb) {\line(1,0){\factor}} 
   \grcalcc = \factor
   \advance \grcalcc by -\sfactor
   \put(\grcalca,\grcalcb) {\line(0,1){\grcalcc}} 
   \advance \grcolumn by 1}

 \newcommand{\gcmptb}{    
   \grcalca = \grcolumn
   \multiply \grcalca by \factor
   \advance \grcalca by \hfactor
   \grcalcb = \grrow
   \multiply \grcalcb by \factor
   \put(\grcalca,\grcalcb) {\line(0,-1){\dfactor}} 
   \advance \grcalcb by -\factor
   \put(\grcalca,\grcalcb) {\line(0,1){\dfactor}} 
   \advance \grcalca by -\hfactor
   \advance \grcalcb by \dfactor
   \put(\grcalca,\grcalcb) {\line(1,0){\factor}} 
   \advance \grcalcb by \factor
   \advance \grcalcb by -\sfactor
   \put(\grcalca,\grcalcb) {\line(1,0){\factor}} 
   \advance \grcolumn by 1}

 \newcommand{\gcmpt}{    
   \grcalca = \grcolumn
   \multiply \grcalca by \factor
   \advance \grcalca by \hfactor
   \grcalcb = \grrow
   \multiply \grcalcb by \factor
   \put(\grcalca,\grcalcb) {\line(0,-1){\dfactor}} 
   \advance \grcalcb by -\factor
   \advance \grcalca by -\hfactor
   \advance \grcalcb by \dfactor
   \put(\grcalca,\grcalcb) {\line(1,0){\factor}} 
   \advance \grcalcb by \factor
   \advance \grcalcb by -\sfactor
   \put(\grcalca,\grcalcb) {\line(1,0){\factor}} 
   \advance \grcolumn by 1}

 \newcommand{\gcmpb}{    
   \grcalca = \grcolumn
   \multiply \grcalca by \factor
   \advance \grcalca by \hfactor
   \grcalcb = \grrow
   \multiply \grcalcb by \factor
   \advance \grcalcb by -\factor
   \put(\grcalca,\grcalcb) {\line(0,1){\dfactor}} 
   \advance \grcalca by -\hfactor
   \advance \grcalcb by \dfactor
   \put(\grcalca,\grcalcb) {\line(1,0){\factor}} 
   \advance \grcalcb by \factor
   \advance \grcalcb by -\sfactor
   \put(\grcalca,\grcalcb) {\line(1,0){\factor}} 
   \advance \grcolumn by 1}

 \newcommand{\gcmp}{    
   \grcalca = \grcolumn
   \multiply \grcalca by \factor
   \grcalcb = \grrow
   \multiply \grcalcb by \factor
   \advance \grcalcb by -\factor
   \advance \grcalcb by \dfactor
   \put(\grcalca,\grcalcb) {\line(1,0){\factor}} 
   \advance \grcalcb by \factor
   \advance \grcalcb by -\sfactor
   \put(\grcalca,\grcalcb) {\line(1,0){\factor}} 
   \advance \grcolumn by 1}

 \newcommand{\grmptb}{    
   \grcalca = \grcolumn
   \multiply \grcalca by \factor
   \advance \grcalca by \hfactor
   \grcalcb = \grrow
   \multiply \grcalcb by \factor
   \put(\grcalca,\grcalcb) {\line(0,-1){\dfactor}} 
   \advance \grcalcb by -\factor
   \put(\grcalca,\grcalcb) {\line(0,1){\dfactor}} 
   \advance \grcalca by \hfactor
   \advance \grcalca by -\dfactor
   \advance \grcalcb by \dfactor
   \put(\grcalca,\grcalcb) {\line(-1,0){\factor}} 
   \advance \grcalcb by \factor
   \advance \grcalcb by -\sfactor
   \put(\grcalca,\grcalcb) {\line(-1,0){\factor}} 
   \grcalcc = \factor
   \advance \grcalcc by -\sfactor
   \put(\grcalca,\grcalcb) {\line(0,-1){\grcalcc}} 
   \advance \grcolumn by 1}

 \newcommand{\grmpt}{    
   \grcalca = \grcolumn
   \multiply \grcalca by \factor
   \advance \grcalca by \hfactor
   \grcalcb = \grrow
   \multiply \grcalcb by \factor
   \put(\grcalca,\grcalcb) {\line(0,-1){\dfactor}} 
   \advance \grcalca by \hfactor
   \advance \grcalca by -\dfactor
   \advance \grcalcb by -\dfactor
   \put(\grcalca,\grcalcb) {\line(-1,0){\factor}} 
   \advance \grcalcb by -\factor
   \advance \grcalcb by \sfactor
   \put(\grcalca,\grcalcb) {\line(-1,0){\factor}} 
   \grcalcc = \factor
   \advance \grcalcc by -\sfactor
   \put(\grcalca,\grcalcb) {\line(0,1){\grcalcc}} 
   \advance \grcolumn by 1}

 \newcommand{\grmpb}{    
   \grcalca = \grcolumn
   \multiply \grcalca by \factor
   \advance \grcalca by \hfactor
   \grcalcb = \grrow
   \multiply \grcalcb by \factor
   \advance \grcalcb by -\factor
   \put(\grcalca,\grcalcb) {\line(0,1){\dfactor}} 
   \advance \grcalca by \hfactor
   \advance \grcalca by -\dfactor
   \advance \grcalcb by \dfactor
   \put(\grcalca,\grcalcb) {\line(-1,0){\factor}} 
   \advance \grcalcb by \factor
   \advance \grcalcb by -\sfactor
   \put(\grcalca,\grcalcb) {\line(-1,0){\factor}} 
   \grcalcc = \factor
   \advance \grcalcc by -\sfactor
   \put(\grcalca,\grcalcb) {\line(0,-1){\grcalcc}} 
   \advance \grcolumn by 1}

 \newcommand{\grmp}{    
   \grcalca = \grcolumn
   \multiply \grcalca by \factor
   \advance \grcalca by \factor
   \advance \grcalca by -\dfactor
   \grcalcb = \grrow
   \multiply \grcalcb by \factor
   \advance \grcalcb by -\dfactor
   \put(\grcalca,\grcalcb) {\line(-1,0){\factor}} 
   \advance \grcalcb by -\factor
   \advance \grcalcb by \sfactor
   \put(\grcalca,\grcalcb) {\line(-1,0){\factor}} 
   \grcalcc = \factor
   \advance \grcalcc by -\sfactor
   \put(\grcalca,\grcalcb) {\line(0,1){\grcalcc}} 
   \advance \grcolumn by 1}

 \newcommand{\gwmuh}[3]{    
   \grcalca = \grcolumn
   \multiply \grcalca by \factor
   \grcalcb = #2
   \advance \grcalcb by #3
   \multiply \grcalcb by \qfactor
   \advance \grcalca by \grcalcb
   \grcalcb = \grrow
   \multiply \grcalcb by \factor
   \grcalcc = #3
   \advance \grcalcc by -#2
   \multiply \grcalcc by \hfactor
   \grcalcd = \factor
   \advance \grcalcd by \hfactor
   \put(\grcalca,\grcalcb){\oval(\grcalcc,\grcalcd)[b]}
   \grcalca = \grcolumn
   \multiply \grcalca by \factor
   \grcalcc = #1
   \multiply \grcalcc by \hfactor
   \advance \grcalca by \grcalcc
   \advance \grcalcb by -\hfactor
   \advance \grcalcb by -\qfactor
   \put(\grcalca,\grcalcb) {\line(0,-1){\qfactor}} 
   \advance \grcolumn by #1}

 \newcommand{\gwcmh}[3]{   
   \grcalca = \grcolumn
   \multiply \grcalca by \factor
   \grcalcb = #2
   \advance \grcalcb by #3
   \multiply \grcalcb by \qfactor
   \advance \grcalca by \grcalcb
   \grcalcb = \grrow
   \advance \grcalcb by -1
   \multiply \grcalcb by \factor
   \grcalcc = #3
   \advance \grcalcc by -#2
   \multiply \grcalcc by \hfactor
   \grcalcd = \factor
   \advance \grcalcd by \hfactor
   \put(\grcalca,\grcalcb){\oval(\grcalcc,\grcalcd)[t]}
   \grcalca = \grcolumn
   \multiply \grcalca by \factor
   \grcalcc = #1
   \multiply \grcalcc by \hfactor
   \advance \grcalca by \grcalcc
   \advance \grcalcb by \factor
   \put(\grcalca,\grcalcb) {\line(0,-1){\qfactor}} 
   \advance \grcolumn by #1}

 \newcommand{\gsbox}[1]{
   \grcalca = \grcolumn
   \multiply \grcalca by \factor
   \grcalcb = \grrow
   \multiply \grcalcb by \factor
   \advance \grcalcb by -\factor
   \grcalcc = #1
   \multiply \grcalcc by \factor
   \grcalcd = \factor
   \put(\grcalca,\grcalcb){\framebox(\grcalcc,\grcalcd){}}}

\allowdisplaybreaks[4]

\begin{document}
\title[Frobenius and separable algebras]
{On Frobenius and separable algebra extensions in monoidal categories. Applications to wreaths}
\dedicatory{Dedicated to Margaret Beattie on the occasion of her retirement}
\thanks{{\it Key words and phrases.}~~{\rm 
monoidal category, 2-category, Frobenius algebra, separable algebra, 
Nakayama automorphism, wreath product}}
\author{Daniel Bulacu}
\address{Faculty of Mathematics and Informatics, University
of Bucharest, Str. Academiei 14, RO-010014 Bucharest 1, Romania}
\email{daniel.bulacu@fmi.unibuc.ro}
\author{Blas Torrecillas}
\address{Department of Algebra and Analysis\\
Universidad de Almer\'{\i}a\\
E-04071 Almer\'{\i}a, Spain}
\email{btorreci@ual.es}
\thanks{
The first author was supported by the UEFISCDI Grant PN-II-ID-PCE-2011-
3-0635, contract no. 253/5.10.2011 of CNCSIS. 
The second author was partially supported by FQM 3128 from Junta 
Andaluc\'{\i}a MTM2011-27090 from MCI.
The first author thanks the Universidad de Almer\'{\i}a for its worm hospitality.
The authors also thank Bodo Pareigis for sharing his ``diagrams" program.}

\subjclass[2010]{Primary 16W30; Secondaries 16T05, 18D05, 18D10, 16S34}

\begin{abstract}
We characterize Frobenius and separable monoidal algebra extensions $i: R\ra S$ 
in terms given by $R$ and $S$. For instance, under some conditions, we show that the extension is Frobenius, respectively 
separable, if and only if $S$ is a Frobenius, respectively separable, algebra in the category 
of bimodules over $R$. In the case when $R$ is separable we show that the extension is separable 
if and only if $S$ is a separable algebra. Similarly, in the case when $R$ is Frobenius and separable in a sovereign monoidal 
category we show that the extension is Frobenius if and only if $S$ is a Frobenius algebra and 
the restriction at $R$ of its Nakayama automorphism is equal to the Nakayama automorphism of $R$.      
As applications, we obtain several characterizations for an algebra extension 
associated to a wreath to be Frobenius, respectively separable.    
\end{abstract}
\maketitle

\section{Introduction}
Frobenius algebras appeared for the first time in the work of Frobenius on representation theory. 
These are finite dimensional algebras over a field $k$ having particularly nice duality 
properties (see for instance the Frobenius equation below). 
The study of Frobenius algebras was started in the thirties of the last century by Brauer
and Nesbitt \cite{brauern} which named these algebras after Frobenius. 
Nakayama discovered the duality property of a Frobenius algebra in 
\cite{nak1, nak3}, and Dieudonn\'e used this to characterize Frobenius algebras in
\cite{dio}. Nakayama also studied symmetric algebras in \cite{nak2} but the automorphism that carries
out his name was defined in \cite{nak3}. Besides 
representation theory, Frobenius algebras play an important role in number theory, algebraic geometry, 
combinatorics, coding theory, geometry, Hopf algebra and quantum group theory, in finding solutions 
for the quantum Yang-Baxter equation, the Jones polynomials, etc. More details about the 
connections between Frobenius algebras and the classical, 
respectively modern, directions mentioned above can be found in the books of Lam \cite{lam} 
and Kadison \cite{kad}.  

Recently, the interest for Frobenius algebras has been renewed 
due to connections to monoidal categories and topological quantum field theory (TQFT for short). 
Roughly speaking, if $\mathbf{nCob}$ is the category 
of $n$ cobordisms then a TQFT is a (symmetric) monoidal functor from $\mathbf{nCob}$ to ${}_k\mathcal{M}$, the 
category of $k$-vector spaces. 
For $n=2$ there exists a complete classification of surfaces, and so the cobordism category $\mathbf{2Cob}$ 
is described completely. Furthermore, the relations that hold in $\mathbf{2Cob}$ correspond exactly 
to the axioms of a commutative Frobenius algebra and this leads to the fact that $\mathbf{2TQFT}$ is equivalent 
to the category of commutative Frobenius algebras. For more details on this topic we invite the reader 
to consult \cite{kock}. We also note that the Frobenius equation (that follows from the fact that both surfaces 
are homeomorphic to a sphere with four holes) 
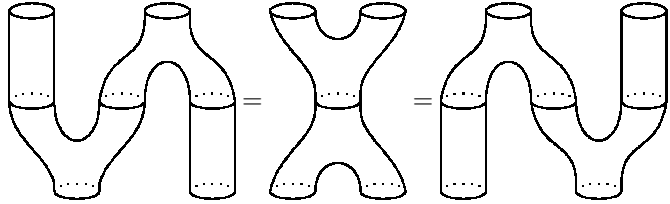
\[
\footnotesize{ \xy
(-9,6)*\ellipse(3,1){-};
(-9,0)*\ellipse(3,1){.};
(-9,0)*\ellipse(3,1)__,=:a(-180){-};
(-21,12)*{}="TL";
(-15,12)*{}="TR";
(-21,0)*{}="BL";
(-15,0)*{}="BR";
"TL"; "BL" **\dir{-};
"TR"; "BR" **\dir{-};
(3,-6)*\ellipse(3,1){.};
(3,-6)*\ellipse(3,1)____,=:a(-180){-};
(0,6)*\ellipse(3,1){-};
(-3,0)*{}="1";
(3,0)*{}="2";
(-9,0)*{}="A2";
(9,0)*{}="B2";
"1";"2" **\crv{(-3,7) & (3,7)};
(-3,12)*{}="A";
(3,12)*{}="B";
(-3,11)*{}="A1";
(3,11)*{}="B1";
"A";"A1" **\dir{-};
"B";"B1" **\dir{-};
"B2";"B1" **\crv{(8,7) & (3,5)};
"A2";"A1" **\crv{(-8,7) & (-3,5)};
(-6,-6)*\ellipse(3,1){.};
(-6,-6)*\ellipse(3,1)__,=:a(-180){-};
(-3,0)*\ellipse(3,1){.};
(-3,0)*\ellipse(3,1)__,=:a(-180){-};
(-15,0)*{}="1"; 
(-9,0)*{}="2"; 
(-21,0)*{}="A2";
(-3,0)*{}="B2"; 
"1";"2" **\crv{(-15,-7) & (-9,-7)};
(-15,-12)*{}="A";
(-9,-12)*{}="B"; 
(-15,-11)*{}="A1";
(-9,-11)*{}="B1"; 
"A";"A1" **\dir{-};
"B";"B1" **\dir{-}; 
"B2";"B1" **\crv{(-3,-5) & (-8,-7)};
"A2";"A1" **\crv{(-21,-5) & (-16,-7)};
(3,0)*\ellipse(3,1){.};
(3,0)*\ellipse(3,1)____,=:a(-180){-};
(3,0)*{}="TL";
(9,0)*{}="TR";
(3,-12)*{}="BL";
(9,-12)*{}="BR";
"TL"; "BL" **\dir{-};
"TR"; "BR" **\dir{-};
\endxy
}
=
\footnotesize{ \xy
(0,0)*\ellipse(3,1){.};
(0,0)*\ellipse(3,1)__,=:a(-180){-};
(-3,-6)*\ellipse(3,1){.};
(3,-6)*\ellipse(3,1){.};
(-3,-6)*\ellipse(3,1)__,=:a(-180){-};
(3,-6)*\ellipse(3,1)__,=:a(-180){-};
(-3,-12)*{}="1";
(3,-12)*{}="2";
(-9,-12)*{}="A2";
(9,-12)*{}="B2";
"1";"2" **\crv{(-3,-7) & (3,-7)};
(-3,0)*{}="A";
(3,0)*{}="B";
(-3,-1)*{}="A1";
(3,-1)*{}="B1";
"A";"A1" **\dir{-};
"B";"B1" **\dir{-};
"B2";"B1" **\crv{(8,-7) & (3,-5)};
"A2";"A1" **\crv{(-8,-7) & (-3,-5)};
(-3,6)*\ellipse(3,1){-};
(3,6)*\ellipse(3,1){-};
(-3,12)*{}="1";
(3,12)*{}="2";
(-9,12)*{}="A2";
(9,12)*{}="B2";
"1";"2" **\crv{(-3,7) & (3,7)};
(-3,0)*{}="A";
(3,0)*{}="B";
(-3,1)*{}="A1";
(3,1)*{}="B1";
"A";"A1" **\dir{-};
"B";"B1" **\dir{-};
"B2";"B1" **\crv{(8,7) & (3,5)};
"A2";"A1" **\crv{(-8,7) & (-3,5)};
\endxy 
}
=
\footnotesize{
\xy
(0,6)*\ellipse(3,1){-};
(-3,0)*\ellipse(3,1){.};
(-3,0)*\ellipse(3,1)__,=:a(-180){-};
(3,0)*\ellipse(3,1){.};
(3,0)*\ellipse(3,1)__,=:a(-180){-};
(-3,0)*{}="1";
(3,0)*{}="2";
(-9,0)*{}="A2";
(9,0)*{}="B2";
"1";"2" **\crv{(-3,7) & (3,7)};
(-3,12)*{}="A";
(3,12)*{}="B";
(-3,11)*{}="A1";
(3,11)*{}="B1";
"A";"A1" **\dir{-};
"B";"B1" **\dir{-};
"B2";"B1" **\crv{(8,7) & (3,5)};
"A2";"A1" **\crv{(-8,7) & (-3,5)};
(9,6)*\ellipse(3,1){-};
(9,0)*\ellipse(3,1){.};
(9,0)*\ellipse(3,1)__,=:a(-180){-};
(15,12)*{}="TL";
(21,12)*{}="TR";
(15,0)*{}="BL";
(21,0)*{}="BR";
"TL"; "BL" **\dir{-};
"TR"; "BR" **\dir{-};
(-3,-6)*\ellipse(3,1){.};
(-3,-6)*\ellipse(3,1)____,=:a(-180){-};
(-9,0)*{}="TL";
(-3,0)*{}="TR";
(-9,-12)*{}="BL";
(-3,-12)*{}="BR";
"TL"; "BL" **\dir{-};
"TR"; "BR" **\dir{-};
(6,-6)*\ellipse(3,1){.}; 
(6,-6)*\ellipse(3,1)____,=:a(-180){-};
(9,0)*{}="1";
(15,0)*{}="2";
(3,0)*{}="A2";
(21,0)*{}="B2";
"1";"2" **\crv{(9,-7) & (15,-7)};
(9,-12)*{}="A";
(15,-12)*{}="B";
(9,-11)*{}="A1";
(15,-11)*{}="B1";
"A";"A1" **\dir{-};
"B";"B1" **\dir{-};
"B2";"B1" **\crv{(20,-7) & (15,-5)};
"A2";"A1" **\crv{(4,-7) & (9,-5)};
\endxy
}\] 
expresses the compatibility between the algebra and coalgebra structure on a Frobenius algebra. It makes sense 
in any monoidal category, and therefore the notion of Frobenius algebra can be defined in any such 
category. This has already been done, see for instance \cite{fust,kock,lauda}. Furthermore, 
in the monoidal categorical framework Frobenius algebras appear in different contexts. Apart from 
the TQFT case mentioned above, we have a correspondence between Frobenius algebras 
in monoidal categories and weak monoidal Morita equivalence of monoidal categories \cite{mu1}, 
Frobenius functors and Frobenius algebras in the category of endofunctors, 
and Frobenius monads in 2-categories and Frobenius 
algebras in some suitable monoidal categories, respectively. Note also that Cayley-Dickson and 
Clifford algebras are example of Frobenius algebras in certain monoidal categories of 
graded vector spaces \cite{bulacu1, bulacu2}.         

Kasch \cite{kasch} extended the notion of Frobenius algebra to an arbitrary algebra extension. 
A $k$-algebra morphism $i: R\ra S$ is called Frobenius if $S$ is finitely generated and projective as right $R$-module 
and ${\rm Hom}_R(S, R)$, the set of right $R$-linear maps from $S$ to $R$, is isomorphic to 
$S$ as $(R, S)$-bimodule. The notion of Frobenius algebra is recovered when $R=k$ and 
$i$ is the unit map of $S$. We should point out that Frobenius extensions have 
a well-developed theory of induced representations, 
investigated in papers by Kasch \cite{kasch0,kasch,kasch1}, Pareigis \cite{par0,par1}, 
Nakayama and Tzuzuku \cite{naktz1, naktz2, naktz3}, Morita \cite{mor1, mor2}, 
and the list may continue. Frobenius extensions, and so Frobenius algebras as well, can be 
characterized in terms of Frobenius functors, first introduced by Morita in \cite{mor1}. Recall that a Frobenius functor 
is a functor having left and right adjoints which are naturally equivalent, and that the terminology is based on 
the fact that an algebra extension $i: R\ra S$ is Frobenius if and only if the restriction of scalars functor 
is Frobenius.   

Due to a famous result of Eilenberg and Nakayama \cite{EiNak}, particular examples of Frobenius 
$k$-algebras are given by separable $k$-algebras. Later on, the result was generalized by Endo and 
Watanabe \cite{EnWa} to the case of algebras over a commutative ring, which are projective 
as modules. More precisely, they showed that if such an algebra is separable then it is symmetric, and 
therefore Frobenius. Although the separability notion extends easily to the algebra extension setting, it is still 
an open question when a separable algebra extension is Frobenius. In this direction the answer is known to be positive 
in some particular cases, see \cite{suga1,suga2}. As far as we are concerned, the separability notion can be 
restated in terms given by separable functors. These were introduced in \cite{nbo} and, similar to the Frobenius case, 
their name is justified by the fact that an algebra 
extension $i: R\ra S$ is separable if and only if the restriction of scalars functor is separable.   

In this paper, which can be seen as a sequel of \cite{bc4} and a predecessor of \cite{bc5}, we have as final 
goal the study of the algebra extensions associated to wreaths in monoidal categories from the Frobenius, and 
respectively separable, point of view. Our motivation is two-fold. On one hand we want to unify most of the Frobenius 
type theories performed for algebra extensions obtained from different types of entwining structures coming from actions 
and coactions of Hopf algebras and their generalizations; corroborated with the results and the set of examples presented 
in \cite{bc4} this leaded us to the study of wreaths in monoidal categories, and then of algebra extensions produced by them. 
On the other hand, we wish to give a monoidal categorical interpretation for the conditions under which these algebra 
extensions are Frobenius, respectively separable, and so to replace a bunch of conditions with one monoidal property. We shall 
explain this better in what follows by presenting the content of this paper.     

Since wreaths in $2$-categories are actually monoidal algebras and the monad extensions produced by them are 
in fact monoidal algebra extensions we first study when a monoidal algebra extension is 
Frobenius, respectively separable. Having in mind the functorial interpretation that exists in the classical 
case, we started by investigating when the restriction of scalars functor and the induction functor are Frobenius, respectively 
separable. To make our theory work, besides the natural conditions that we need to impose (existence of 
coequalizers, coflatnes, robustness), we have to assume that the unit object $\un{1}$ is a $\ot$-generator, too. For short, 
this extra condition is the substitute of the fact that giving an element of a set is equivalent to giving a map 
from a singleton to the set, or giving a vector is equivalent to giving a linear map from the 
base field to the vector space where it resides. Otherwise stated, we can avoid working with ``elements" by  
considering morphisms in a category having the unit object as source, providing that $\un{1}$ is a 
$\ot$-generator. So under these conditions we give in \thref{FrobSepExt} necessary and sufficient conditions for the 
restriction of scalars functor to be Frobenius and in \thref{SepExtMon} the ones for which it is separable, generalizing 
in this way \cite[Theorem 27]{cmz} to the monoidal categorical setting. As expected, we obtain as a consequence 
that an algebra $A$ in a monoidal category $\Cc$ is Frobenius/separable if and only if 
the forgetful functor $F: \Cc_A\ra \Cc$ is so, of course, providing that $\un{1}$ is a $\ot$-generator in $\Cc$. Motivated by 
these results we define in \seref{FrobSepAlgMon} the notion of a Frobenius/separable algebra extension in $\Cc$ in such a way 
that it becomes equivalent to the fact that the restriction of scalars functor is Frobenius/separable, providing again 
that $\un{1}$ is a $\ot$-generator in $\Cc$. 

If $\un{1}$ is a projective object in $\Cc$ then an algebra $A$ in $\Cc$ is separable if and only if it is projective 
as an $A$-bimodule in $\Cc$ (\prref{sepversusproj}). Furthermore, in \prref{Frobalgextcharact} we show that 
an algebra extension $i: R\ra S$ is Frobenius/separable if and only if $S$ is a Frobenius/separable algebra 
in the category of $R$-bimodules in $\Cc$, ${}_R\Cc_R$, and therefore the study of Frobenius/separable algebra extensions 
reduce to the study of monoidal Frobenius/separable algebras. Consequently, we get that when $\un{1}$ is projective or 
$\ot$-generator and $R$ is separable the extension $i: R\ra S$ is separable if and only if $S$ is a separable 
algebra in $\Cc$ (\coref{SepExtVsSep}). Note that \prref{Frobalgextcharact} gives also a new approach for 
dealing with the problem of when a separable algebra extension is Frobenius. More exactly, if we can answer to question 
\begin{itemize}
\item[]{\it What are the monoidal categories for which any separable algebra is Frobenius?}
\end{itemize} 
then one can uncover the conditions under which a separable algebra extension is Frobenius. Nevertheless, we can always handle the 
converse of the above question. More exactly, \prref{WhenFrobisSep} measures how far is a Frobenius monoidal algebra from being 
separable, and consequently how far is a Frobenius algebra extension from being separable (\coref{WhenFrobExtIsSep}). 

The purpose of \seref{FurtherCharFrobExt} is to present new characterizations of Frobenius algebra extensions. 
For an algebra $A$ in a monoidal category, we have several conditions equivalent to $A$ being Frobenius, we collected 
them in \thref{charactFrobalg}. Since an algebra extension $i: R\ra S$ is Frobenius if and only if $S$ is a Frobenius 
algebra in ${}_R\Cc_R$, we get as an immediate consequence of \thref{charactFrobalg} a list of characterizations 
for $i: R\ra S$ to be Frobenius (\coref{CharFrobAlgExt}). The natural problems that show up now are: (1). When is ${}_R\Cc_R$ 
rigid monoidal?; (2). If ${}_R\Cc_R$ is rigid then is $\Cc$ so?; and (3). If (2) is true, then how can we relate the 
dual of an object in ${}_R\Cc_R$ with the dual of the same object regarded now as an object of $\Cc$ via the forgetful functor? 
To answer these questions we adapted the techniques used in \cite[Section 5]{yam}, where it is proved that 
${}_R\Cc_R$ is rigid in the case when $\Cc$ is so and $R$ is a special Frobenius algebra (recall that 
a Frobenius algebra $R$ is called special if $\un{m}_R\un{\Delta}_R=d_R\Id_R$ and $\un{\va}_R\un{\eta}_R=d_R\Id_R$, for some 
``non-zero scalar" $d_R$, where $(R, \un{m}_R, \un{\eta}_R)$ and $(R, \un{\Delta}_R, \un{\va}_R)$ denote the 
algebra and coalgebra structures of the Frobenius algebra $R$). To be more precise, we walked backwards through 
these questions. Firstly, it is well-known that a strong monoidal functor preserves dual objects, so we might have an answer for 
(3) in the case when the forgetful functor $\mathfrak{U}: {}_R\Cc_R\ra \Cc$ is strong monoidal. But this condition on $\mathfrak{U}$ 
is very restrictive, so we replaced it with the Frobenius monoidal one, a weaker condition under which $\mathfrak{U}$ still 
preserves dual situations, cf. \cite[Theorem 2]{DayPastro}. $\mathfrak{U}$ has a trivial monoidal structure and this is part 
of a Frobenius monoidal structure if and only if $R$ is a Frobenius algebra. Furthermore, if this is the case, then 
the opmonoidal structure of $\mathfrak{U}$ is completely determined by the Frobenius structure of $R$, see \thref{ForgFrobMonFunctor}. 
Thus, if $R$ is Frobenius and ${}_R\Cc_R$ is rigid then $\Cc$ is rigid as well, and this answers partially to (3). 
Secondly, we prove in \prref{DualAreDualBimod} that the converse remains true if we assume, in addition, that $R$ is 
separable (a situation different from the one considered in \cite{yam}, and more appropriate to the topic of this paper).  
In particular, if $R$ is a separable Frobenius algebra, we have that ${}_R\Cc_R$ is rigid if and only if $\Cc$ is so, 
and if this is the case, the dual objects coincide (only the evaluation and coevaluation morphisms are different). 
This answers partially the questions (1) and (2). Furthermore, using these results we are able to show that 
if $R$ is a separable Frobenius algebra, then an algebra extension $i: R\ra S$ is Frobenius if and only if $S$ is a Frobenius algebra and 
a condition involving the Frobenius structures of $R$ and $S$ holds (\thref{FrobExtBaseFS}). When $\Cc$ is sovereign monoidal, that is, 
$\Cc$ is rigid and the left dual functor is equal to the right dual functor, then this condition is equivalent to the 
the fact that the restriction at $R$ of the Nakayama automorphism of $S$ coincides with the Nakayama automorphism of $R$, see 
\thref{FrobExtBaseFSsov}.

In \seref{ApplWreaths} we give some applications. Particular examples of monoidal algebras are given by monads 
in an arbitrary category (algebras in a category of endofunctors) and monads in $2$-categories. Thus if we 
specialize our results to these particular situations we get for free necessary and sufficient conditions for 
which a monad extension or a $2$-monad extension is Frobenius, respectively separable. Since the former is a particular 
case of the latter, we restrict ourselves 
in working only in the context provided by $2$-categories. If $\mathbb{A}=(A, t, \mu, \eta)$ is a monad in 
a $2$-category $\mathcal{K}$ then $(t, \mu, \eta)$ is an algebra in the monoidal category $\mathcal{K}(A)$ and, 
moreover, $\mathbb{A}$ is Frobenius/separable in the 2-categorical sense if and only if $(t, \mu, \eta)$ is 
a monoidal Frobenius/separable algebra in $\mathcal{K}(A)$. With this observation in mind we easily deduce 
necessary and sufficient conditions for which a wreath, i.e., a monad in the $2$-category ${\rm EM}(\mathcal{K})$ 
(the Eilenberg-Moore category associated to 
$\mathcal{K}$), is a Frobenius/separable $2$-monad. Besides these characterizations, we give also new ones, providing 
that $\mathcal{K}(A)$ admits coequalizers and any object of it is coflat. More exactly, to any wreath in $\mathcal{K}$ 
we can associate an algebra extension in $\mathcal{K}(A)$ that we call the canonical monad extension associated to the wreath. 
Then the wreath is Frobenius/separable if and only if the associated canonical monad extension is so (Theorems 
\ref{charactextFrob2categ} and \ref{charactextSep2categ}). Finally, all the results obtained throughout the paper can be 
applied to the (monad) algebra extension associated to a wreath in a monoidal category. We summed up all these in 
\coref{CharFrobWreatsMon} for the Frobenius case, and respectively in \coref{CharSepWreatsMon} for the separable case. 
In this way we achieved our main goal. Furthermore, we will see in \cite{bc5} that the Frobenius/separable properties of 
a monoidal wreath play an important role in establishing Frobenius properties and Maschke type theorems for the 
generalized entwined module categories that were introduced in \cite{bc4}.  
      
\section{Preliminaries}
\subsection{Rings and corings in monoidal categories}
Throughout this paper $\Cc$ is a monoidal category with unit object $\un{1}$. Without loss 
of generality we assume that $\Cc$ is strict, this means that the associativity and the left and right 
unit constraints are defined by identity morphisms in $\Cc$. If $A$ is an algebra in $\Cc$ we then denote by 
$\un{m}_A$ and $\un{\eta}_A$ its multiplication and unit morphisms, and by ${}_A\Cc$ ($\Cc_A$) the category of left 
(right) $A$-modules and left (right) $A$-morphisms in $\Cc$. For more details about algebras in a monoidal category 
we invite the reader to consult \cite{bulacu, k, m3, mc}.       

Assume now that $\Cc$ has coequalizers. 
Take an algebra $A$ in $\Cc$, $\mfM\in \Cc_A$ and $X\in {}_A\Cc$,
with structure morphisms $\mu_X^A:\ A\ot X\to X$ and $\nu_\mfM^A: \mfM\ot A\ra \mfM$, respectively. We consider the
coequalizer $(M\ot_AX, q^A_{{\mfM}, X})$ of the parallel morphisms  
$\nu_{\mfM}^A\ot\Id_X$ and $\Id_\mfM\ot \mu_{X}^A$ in $\Cc$:
\[
\xymatrix{
\mfM\ot A\ot X\ar@<-1ex>[rr]_(.54){\Id_\mfM\ot \mu^A_{X}} 
\ar@<1ex>[rr]^(.55){{\nu_{\mfM}^A\ot\Id_X}}&&
M\ot X\ar[r]^-{q^A_{{\mfM}, X}}& M\ot_AX .
}
\]

For a left $A$-linear morphism $f:\ X\ra Y$ in $\Cc$, 
let $\tilde{f}:\ \mfM\ot_AX\ra \mfM\ot_AY$ be the unique morphism 
in $\Cc$ satisfying the equation 
\begin{equation}\eqlabel{coeqmorph}
\tilde{f}q^A_{\mfM, X}=q^A_{\mfM, Y}(\Id_\mfM\ot f).
\end{equation}
Take $X\mapright{f}Y\mapright{f}Z$ in ${}_A\Cc$. It is easily verified that
$\widetilde{gf}=\tilde{g}\tilde{f}$.

Now let $g:\mfM\ra {\mf N}$ in $\Cc_A$ and $Y\in {}_A\Cc$. Then 
$\hat{g}: \mfM\ot_AY\ra {\mf N}\ot_AY$ denotes the unique morphism in 
$\Cc$ obeying   
\begin{equation}\eqlabel{coeqmorph2}
\hat{g} q^A_{{\mfM}, Y}=q^A_{{\mf N}, Y}(g\ot\Id_Y).
\end{equation}
For $\mfM \mapright{f} \mfN\mapright{g}\mfP$ in $\Cc_A$, we have that
$\widehat{gf}=\hat{g}\hat{f}$.

For $\mfM\in \Cc_A$, $X\in \Cc$ and $Y\in {}_A\Cc$, we have canonical
isomorphisms $\Upsilon_\mfM$, $\Upsilon_{\mfM, X}$ and $\Upsilon'_Y$:
\begin{itemize}
\item[-]$\Upsilon_\mfM: \mfM\ot_AA\mapright{\cong} \mfM$, uniquely determined
by the property $\Upsilon_\mfM q^A_{\mfM, A}=\nu_\mfM^A$;
\item[-]$\Upsilon_{\mfM, X}: \mfM\ot_A(A\ot X)\mapright{\cong} \mfM\ot X$, 
uniquely determined by the property 
$\Upsilon_{\mfM, X}q^A_{\mfM, A\ot X}=\nu_\mfM^A\ot \Id_X$; 
\item[-]$\Upsilon'_Y: A\ot _AY\mapright{\cong} Y$, uniquely determined by the property
$\Upsilon'_Yq^A_{A, Y}=\mu_Y^A$.
\end{itemize}
The following properties are now easily verified:
\begin{eqnarray}
\eqlabel{invcanisom1}
\Upsilon^{-1}_\mfM=q^A_{\mfM, A}(\Id_\mfM\ot \un{\eta}_A)&;&
\Upsilon'^{-1}_Y=q^A_{A, Y}(\un{\eta}_A\ot \Id_Y);\\
\eqlabel{rel1}
&&\hspace*{-3cm}
\Upsilon^{-1}_{\mfM, X}=q^A_{\mfM, A\ot X}(\Id_\mfM\ot \un{\eta}_A\ot \Id_X).
\end{eqnarray}   

Next, recall that an object $X$ of $\Cc$ is called right (left) coflat 
if the functor $X\ot-$ (respectively $-\ot X$) preserves coequalizers. An 
object of $\Cc$ is called coflat if it is both left and right coflat.  

Let now $\Cc$ be a monoidal category with coequalizers and $A, R$ algebras in $\Cc$. 
By \cite[Lemma 2.4]{bc3} we have the following.

$(a)$ If $A$ is right coflat then for any $X\in {}_A\Cc_R$ and $Y\in {}_R\Cc$ 
the morphism $\mu^A_{X\ot_RY}: A\ot X\ot_RY\ra X\ot_RY$ uniquely determined by 
$\mu^A_{X\ot_RY}(\Id_A\ot q^R_{X, Y})=q^R_{X, Y}(\mu^A_X\ot\Id_Y)$ defines on $X\ot_RY$ a left 
$A$-module structure in $\Cc$, where, in general, by $\mu^A_X: A\ot X\ra X$ we denote 
a left $A$-module structure on an object $X$ of $\Cc$;

$(b)$ Likewise, if $A$ is left coflat then for any $X\in \Cc_R$ and $Y\in {}_R\Cc_A$ 
the morphism $\nu^A_{X\ot_RY}: X\ot_RY\ot A\ra X\ot_RY$ uniquely determined by 
$\nu^A_{X\ot_RY}(q^R_{X, Y}\ot \Id_A)=q^R_{X, Y}(\Id_X\ot \nu^A_Y)$ defines on $X\ot_RY$ a 
right $A$-module structure in $\Cc$, where, in general, by $\nu^A_Y: Y\ot A\ra Y$ we denote 
the morphism structure of a right $A$-module $Y$ in $\Cc$.

If $R$ is an algebra in $\Cc$ we then say that an object $Y\in {}_R\Cc$ is left robust if for 
any $\mfM\in \Cc$, $X\in \Cc_R$ the morphism $\theta'_{\mfM, X, Y}: (\mfM\ot X)\ot_RY\ra \mfM\ot(X\ot_RY)$ defined 
by the commutativity of the diagram
\[
\xymatrix{
\mfM\ot X\ot R\ot Y\ar@<1ex>[rr]^(.55){\nu^R_{\mfM\ot X}\ot \Id_Y} 
\ar@<-1ex>[rr]_-{\Id_\mfM\ot \Id_X\ot \mu^R_Y}&&
\mfM\ot X\ot Y \ar[dr]_-{\Id_\mfM\ot q^R_{X, Y}} 
\ar[r]^-{q^R_{\mfM\ot X, Y}}&
(\mfM\ot X)\ot_RY\ar@{.>}[d]^-{\theta'_{\mfM, X, Y}}\\
&&&\mfM\ot (X\ot _RY)
}
\]
is an isomorphism. If $R$ is left coflat the it is well-known that 
the category of $R$-bimodules in $\Cc$ that are left coflat and left robust, denoted by 
${}_R^!\Cc_R$, is a monoidal category (see for instance \cite{bc3, par, sch}). 
Notice that the left robustness of an object $Y\in {}_R\Cc$ is needed in order to define 

$(1)$ a left $R$-module structure on $X\ot_RY$ in $\Cc$, for any $X\in{}_R\Cc_R$. Namely, 
if we define $\ov{\mu}_{X\ot_RY}: (R\ot X)\ot_RY\ra X\ot_RY$ as being the unique morphism in $\Cc$ making the 
diagram 
\[
\xymatrix{
R\ot X\ot R\ot Y\ar@<1ex>[rr]^(.55){\Id_R\ot \nu^R_X\ot \Id_Y}\ar@<-1ex>[rr]_-{\Id_{R\ot X}\ot \mu^R_Y}
&&R\ot X\ot Y\ar[rd]_-{\hspace{-6mm}q^R_{X, Y}(\mu^R_X\ot \Id_Y)}\ar[r]^-{q^R_{R\ot X, Y}}&(R\ot X)\ot_RY\ar@{.>}[d]^-{\ov{\mu}_{X\ot_RY}}\\
&&&X\ot_RY
}
\]  
commutative then the morphism $\mu^R_{X\ot_RY}:=\ov{\mu}_{X\ot_RY}\theta'^{-1}_{\mfM, X, Y}$ defines on 
$X\ot_RY$ a left $R$-module structure in $\Cc$. Furthermore, $X\ot_RY$ becomes in this way 
an $R$-bimodule, providing that $Y\in {}_R^!\Cc_R$ and $R$ is left coflat. 
We should also remark that this left $R$-module structure coincide with 
that when $R$ is right coflat, see $(a)$ above; 

$(2)$ a morphism $\Gamma'_{\mfM, X, Y}: \mfM\ot_R(X\ot_RY)\ra (\mfM\ot_RX)\ot_RY$, for all 
$\mfM, X, Y\in {}_R\Cc_R$, uniquely determined by $\Gamma'_{\mfM, X, Y}q^R_{\mfM, X\ot_RY}=
\widehat{q^R_{\mfM, X}}\theta'^{-1}_{\mfM, X, Y}$. It is actually an isomorphism with 
inverse $\Sigma'_{\mfM, X, Y}$ uniquely determined by the property that 
\[
\Sigma'_{\mfM, X, Y}q^R_{\mfM\ot_RX, Y}(q^R_{\mfM, X}\ot\Id_Y)=q^R_{\mfM, X\ot_RY}(\Id_\mfM\ot q^R_{X, Y}).
\]
So when $R$ is left coflat the category ${}^!_R\Cc_R$ is monoidal with tensor product 
$\ot_R$, associativity constraint $\Sigma'_{-, -, -}$, unit object $R$,
and left and right unit constraints $\Upsilon'_{-}$ and $\Upsilon_{-}$. Once more, the 
full details can be found in \cite{bc3, sch, par}. Note also that in most of the cases we assume 
that $\Cc$ is monoidal and such that any object of it is coflat and left robust, thus 
${}^!_R\Cc_R$ identifies with ${}_R\Cc_R$ in this case. 
\subsection{Monads in 2-categories}
Let ${\cal K}$ be a $2$-category; its objects (or $0$-cells) will be denoted by capital letters. $1$-cell between
two 0-cells $U$ and $V$ will be denoted as 
$
\xymatrix{
U\ar[r]^{f}&V
}
$, 
the identity morphism of a $1$-cell $f$ by $1_f$ and, more generally, a $2$-cell by 
$
\xymatrix{
f\ar@2{->}[r]^\rho&f'
}
$. We also denote by $\circ$ the vertical composition of $2$-cells 
$
\xymatrix{
f\ar@2{->}[r]^\rho&f'\ar@2{->}[r]^-{\tau}&f{''}
}
$ in ${\cal K}(U, V)$, 
by $\odot$ the horizontal composition of $2$-cells 
\[
\xymatrix{
U\rtwocell^f_{f'}{\rho} &V\rtwocell^g_{g'}{\rho'}&W 
},~~
\xymatrix{
g\circ f\ar@2{->}[r]^-{\rho'\odot\rho}&g'\circ f', 
}
\] 
and by 
$
(
\xymatrix{
U\ar[r]^{1_U}&U
}, 
\xymatrix{
1_U\ar@2{->}[r]^{i_U}&1_U
}
)
$ 
the pair defined by the image of the unit functor from ${\bf 1}$ to ${\cal K}(U, U)$, where 
${\bf 1}$ is the terminal object of the category of small categories.  
For more detail on 2-categories, we refer the reader to \cite[Ch. 7]{Borceux} or \cite[Ch. XII]{mc}.   

A monad in ${\cal K}$ is a quadruple $(A, t, \mu, \eta)$ consisting in an object $A$ of ${\cal K}$, a $1$-cell 
$
\xymatrix{
A\ar[r]^t&A
}
$  
in ${\cal K}$ and $2$-cells 
$
\xymatrix{
t\circ t\ar@2{->}[r]^{\mu}&t
}
$ 
and 
$
\xymatrix{
1_A\ar@2{->}[r]^\eta&t
}
$ 
in ${\cal K}$ such that 
\[
\mu\circ(\mu \odot 1_t)=\mu\circ (1_t\odot \mu)~,~\mu\circ (1_t\odot \eta)=1_t=\mu\circ (\eta\odot 1_t)~. 
\] 
If $\mathbb{A}=(A, t, \mu_t, \eta_t)$ and 
$\mathbb{B}=(B, s, \mu_s, \eta_s)$ are monads in ${\cal K}$ then a monad morphism between $\mathbb{A}$ and $\mathbb{B}$ 
is a pair $(f, \psi)$ with 
$
\xymatrix{
A\ar[r]^f&B
}
$ 
a $1$-cell in ${\mathbb K}$ and 
$\xymatrix{
s\circ f\ar@2{->}[r]^{\psi}&f\circ t
}
$   
a $2$-cell in ${\mathcal K}$ such that the following equalities hold:
\[
(1_f\odot \mu_t)\circ (\psi\odot 1_t)\circ (1_s\odot \psi)=\psi\circ (\mu_s\odot 1_f)~,~
\psi\circ(\eta_s\odot 1_f)=1_f\odot \eta_t.
\]
\subsection{Frobenius and separable functors}\selabel{FrobSepFunct}
Let ${\cal F}$ be a functor between two arbitrary categories ${\cal D}$ and ${\cal E}$. Recall that 
${\cal F}$ is called Frobenius if it has a right adjoint functor which is also left adjoint, and that 
${\cal F}$ is called separable if for any two objects $X, Y$ of $\Dc$ there exists a map 
${\cal P}_{X, Y}: {\rm Hom}_{\cal E}({\cal F}(X), {\cal F}(Y))\ra {\rm Hom}_\Dc(X, Y)$ such that 
${\cal P}_{X, Y}({\cal F}(f))=f$, for all $f: X\ra Y$ in $\Dc$, and 
${\cal P}_{Y_1, Y_2}(g_2)\circ f_1=f_2\circ {\cal P}_{X_1, X_2}(g_1)$ for every commutative diagram in ${\cal E}$ of type 
\[
\xymatrix{
F(X_1)\ar[r]^-{g_1}\ar[d]_-{F(f_1)}&F(X_2)\ar[d]^-{F(f_2)}\\
F(Y_1)\ar[r]^-{g_2}&F(Y_2)
}.
\]  
When ${\cal F}$ has a right adjoint ${\cal G}: {\cal E}\ra \Dc$ the Rafael's theorem \cite{rafael} gives 
necessary and sufficient conditions for ${\cal F}$ or ${\cal G}$ to be separable. More exactly, if $1_\Dc$ and 
$1_{\cal E}$ are the identity functors on $\Dc$ and ${\cal E}$, and  
$\eta: 1_\Dc\ra {\cal G}{\cal F}$ and $\va: {\cal F}{\cal G}\ra 1_{\cal E}$ are the unit and the counit of the 
adjunction, respectively, then 

$\bullet$ ${\cal F}$ is separable if and only if the unit $\eta$ splits, that is there exists a natural 
transformation $\nu: {\cal G}{\cal F}\ra 1_\Dc$ such that $\nu\circ\eta$ is the identity natural transformation of $1_\Dc$;

$\bullet$ ${\cal G}$ is separable if and only if the counit $\va$ cosplits, that is there exists a natural 
transformation $\mu: 1_{\cal E}\ra {\cal F}{\cal G}$ such that $\va\circ\mu$ is the identity natural transformation 
of $1_{\cal E}$.   
\section{Frobenius and separable type properties for the restriction of scalars functor and the induction functor}\selabel{ResIndFun}
\setcounter{equation}{0}
In the literature there are several Frobenius or separable theories developed for different kind of algebras. 
All these are based on the fact that a certain canonical functor is Frobenius, respectively separable. 
Usually, this canonical functor is a forgetful functor or, more generally, a functor restriction of scalars. 

As far as we are concerned, we are interested to study when the extension defined by a wreath in a monoidal 
category $\Cc$ is Frobenius, respectively separable. This can be done in two ways, depending on the 
point of view: monoidal or 2-categorical. As we will see, both cases 
require actually the study of Frobenius (respectively separable) algebra extensions in a monoidal category, and this is why 
we shall study this problem first. More precisely, for an algebra extension $S/R$ in $\Cc$, that is for 
an algebra morphism $i: R\ra S$ in a monoidal category $\Cc$, we will 
give necessary and sufficient conditions for which the functor restriction of scalars $F: \Cc_S\ra \Cc_R$ 
is Frobenius, and respectively separable. 

If $S$ is left coflat then the functor restriction of scalars $F$ has always a left adjoint functor $G$. 
Namely, $G$ is the induction functor $-\ot_RS: \Cc_R\ra \Cc_S$ which is defined as follows. If $X\in \Cc_R$ then $(-\ot_RS)(X)=X\ot_RS$ endowed with the 
right $S$-module structure induced by the multiplication of $S$, and if $f:X\ra Y$ is a morphism in $\Cc_R$ then $(-\ot_RS)(f)=\hat{f}$. 
The unit and the counit of the adjunction are described as follows, for all $X\in \Cc_R$ and 
$\mfM\in \Cc_S$,
\begin{equation}\eqlabel{indleftadjrestscal}
\eta_X:=q^R_{X, S}(\Id_X\ot \un{\eta}_S): X\to X\ot_RS~~,~~\va_\mfM:=\ov{\nu^S_\mfM}: \mfM\ot_RS\ra \mfM,
\end{equation}
the latest being determined uniquely by the property that $\ov{\nu^S_\mfM}q^R_{\mfM, S}=\nu^S_\mfM$. 

So $F$ is a Frobenius functor if and only if $G=-\ot_RS$ is a right adjoint functor of $F$, and $F$ is separable if 
and only if the counit of the adjunction defined above splits. To see when these conditions hold we first describe 
the sets ${\rm Nat}(F\circ (-\ot_RS), -)$ and ${\rm Nat}(-, (-\ot_RS)\circ F)$, where, in general, 
if $\mathcal{F}, \mathcal{G}: \Dc\ra \mathcal{E}$ are two functors we then denote by ${\rm Nat}(\mathcal{F}, \mathcal{G})$ 
the set of natural transformations from $\mathcal{F}$ to $\mathcal{G}$. To this end we use the techniques performed in 
\cite{cmz}. To make them work in the setting of a monoidal category $\Cc$ we have to assume that $\Cc$ is left $\ot$-generated by 
its unit object. This means the following.

\begin{definition}\delabel{leftotgen}
Let $\Cc$ be a monoidal category. We say that an object $P$ of $\Cc$ is a left $\ot$-generator of 
$\Cc$ if wherever we consider two morphisms 
$
\xymatrix{
Y\ot Z\ar@<1ex>[r]^-{f}\ar@<-1ex>[r]_-{g}& W
}
$ 
in $\Cc$ such that $f(\epsilon\ot \Id_Z)=g(\epsilon\ot\Id_Z)$, for all $\epsilon: P\ra Y$ in $\Cc$, we then have $f=g$. 
\end{definition}    

Observe that, by taking $Z=\un{1}$ in the above definition we get that a left $\ot$-generator of 
a monoidal category $\Cc$ is necessarily a generator for $\Cc$, and this justifies our terminology. 
  
\begin{lemma}\lelabel{Frobenius1}
Let $\Cc$ be a monoidal category with coequalizers and assume that $\un{1}$ is a 
left $\ot$-generator for $\Cc$. If $i: R\ra S$ 
is an algebra morphism in $\Cc$ with $S$ a left coflat object, $F: \Cc_S\ra \Cc_R$ is the restriction 
of scalars functor induced by $i$ and $-\ot_RS: \Cc_R\ra \Cc_S$ is the induction functor, 
then there exists an isomorphism 
\[
{\rm Nat}(F\circ (-\ot_RS), -)\cong {}_R{\rm Hom}_R(S, R),
\]
where ${}_R{\rm Hom}_R(S, R)$ stands for the set of $R$-bimodule morphisms from $S$ to $R$ in $\Cc$.
\end{lemma}
\begin{proof}
Follows the same line as the proof of \cite[Theorem 2.7, 3.]{cmz}. We first show 
that the desired isomorphism is produced by 
\[
\a: {\rm Nat}(F\circ (-\ot_RS), -)\cong {}_R{\rm Hom}_R(S, R),
\]
defined as follows. If $u=\left(u_X: X\ot_RS\ra X\right)_{X\in \Cc_R}$ is 
in ${\rm Nat}(F\circ (-\ot_RS), -)$ 
we then define $\a(u)=u_R\Upsilon'^{-1}_S$. To show that $\a(u)$ is a right $R$-linear morphism 
is sufficient to show that $\Upsilon'^{-1}_S$ is so. Indeed, on one hand we have 
\begin{eqnarray*}
\mu^R_{R\ot_RS}(\Id_R\ot \Upsilon'^{-1}_S)&=&
\mu^R_{R\ot_RS}(\Id_R\ot q^R_{R, S}(\un{\eta}_R\ot\Id_S))\\
&=&q^R_{R, S}(\un{m}_R(\Id_R\ot\un{\eta}_R)\ot \Id_S)=q^R_{R, S}.
\end{eqnarray*} 
On the other hand, 
\begin{eqnarray*}
\Upsilon'^{-1}_S\mu^R_S&=&\Upsilon'^{-1}_S\un{m}_S(i\ot \Id_S)\\
&=&q^R_{R, S}(\un{\eta}_R\ot \Id_S)\un{m}_S(i\ot\Id_S)\\
&=&q^R_{R, S}(\un{m}_R(\un{\eta}_R\ot \Id_R)\ot\Id_S)=q^R_{R, S},
\end{eqnarray*}
as required. 

Take now $\epsilon: \un{1}\ra R$ an arbitrary morphism in $\Cc$ and define 
$f_\epsilon: R\ra R$ by $f_\va=\un{m}_R(\epsilon\ot \Id_R)$. Since 
$\un{m}_R$ is associative we get that $f_\epsilon$ is a morphism in 
$\Cc_R$, and therefore the naturality of $u$ implies that 
$f_\epsilon u_R=u_R\widehat{f_\epsilon}$. This is equivalent to 
$f_\epsilon u_R\Upsilon'^{-1}_S=u_R\widehat{f_\epsilon}\Upsilon'^{-1}_S$, and so 
with 
\[
\un{m}_R(\Id_R\ot \a(u))(\epsilon\ot \Id_S)=u_Rq^R_{R, S}(\epsilon\ot\Id_S).
\]
Since 
\begin{eqnarray*}
\a(u)\un{m}_S(i\ot\Id_S)&=&u_R \Upsilon'^{-1}_S\un{m}_S(i\ot \Id_S)\\
&=&u_Rq^R_{R, S}(\un{\eta}_R\ot\Id_S)\un{m}_S(i\ot \Id_S)\\
&=&u_R q^R_{R, S}(\Id_R\ot \un{m}_S(i\ot \Id_S))(\un{\eta}_R\ot \Id_{R\ot S})\\
&=&u_Rq^R_{R, S}(\un{m}_R(\un{\eta}_R\ot \Id_R)\ot \Id_S)=u_R q^R_{R, S},
\end{eqnarray*}
where in the last by one equality we used that $(q^R_{R, S}, R\ot_RS)$ is a coequalizer, 
we obtain that 
\[
\un{m}_R(\Id_R\ot \a(u))(\epsilon\ot \Id_S)=\a(u)\un{m}_S(i\ot \Id_S)(\epsilon\ot \Id_S),
\]
for all $\epsilon:\un{1}\ra R$. The fact that $\un{1}$ is a left $\ot$-generator for $\Cc$ 
implies now that $\a(u)$ is left $R$-linear as well, and thus an $R$-bimodule morphism 
in $\Cc$. 

We claim now that $\a$ is an isomorphism with inverse defined by 
\[
\a^{-1}(\vartheta)=\left(v_X=\Upsilon_X\tilde{\vartheta}: X\ot_RS\ra X\right)_{X\in \Cc_R}~,~
\vartheta\in {}_R{\rm Hom}_R(S, R).
\]
One can easily see that $v=(v_X)_{X\in \Cc_R}$ is completely determined by the property 
$v_Xq^R_{X, S}=\nu^R_X(\Id_X\ot \vartheta)$, for any $X\in \Cc_R$, and this 
allows us to prove easily that $v$ is indeed a natural transformation. Thus $\a^{-1}$ is well defined. 
Furthermore, $\a^{-1}\a (u)=\a^{-1}(u_R\Upsilon'^{-1}_S)=(v_X=\Upsilon_Xu_R\Upsilon'^{-1}_S)_{X\in \Cc_R}$ 
with $v_X$ characterized by the fact that $v_Xq^R_{X, S}=\nu^R_X(\Id_X\ot \a(u))$. 

Take now $\epsilon: \un{1}\ra X$ an arbitrary morphism in $\Cc$ and define 
$g_\epsilon=\nu^R_X(\epsilon \ot\Id_R): R\ra X$. It is immediate that $g_\epsilon$ is right 
$R$-linear, hence by the naturality of $u$ we have 
\[
g_\epsilon\a(u)=g_\epsilon u_R\Upsilon'^{-1}_S=u_X \widehat{g_\epsilon}\Upsilon'^{-1}_S
=u_X q^R_{X, S}(g_\epsilon\un{\eta}_R\ot\Id_S)=u_X q^R_{X, S}(\epsilon\ot\Id_S).
\]
But $g_\epsilon \a(u)=\nu^R_X(\epsilon\ot \Id_R)\a(u)=\nu^R_X(\Id_X\ot \a(u))(\epsilon\ot \Id_S)$, 
so using again that $\un{1}$ is a left $\ot$-generator in $\Cc$ we deduce that 
$\nu^R_X(\Id_X\ot \a(u))=u_Xq^R_{X, S}$, for all $X$ in $\Cc$. This implies 
$v_Xq^R_{X, S}=u_Xq^R_{X, S}$ and since $q^R_{X, S}$ is an epimorphism in $\Cc$ 
we conclude that $v_X=u_X$, for all $X$ in $\Cc$, and thus $\a^{-1}\a(u)=u$. 

For $\vartheta\in {}_R{\rm Hom}_R(S, R)$ one have  
\begin{eqnarray*}
\a\a^{-1}(\vartheta)&=&\a\left((v_X=\Upsilon_X\tilde{\vartheta}: X\ot_RS\ra X)_{X\in \Cc_R}\right)\\
&=&v_R\Upsilon'^{-1}_S=v_Rq^R_{R, S}(\un{\eta}_R\ot\Id_S)=\nu^R_R(\Id_R\ot \vartheta)(\un{\eta}_R\ot\Id_S)=\vartheta.
\end{eqnarray*}   
Thus $\a$ and $\a^{-1}$ are inverses, as claimed.
\end{proof}

\begin{lemma}\lelabel{Frobenius2}
Let $\Cc$ be a monoidal category with coequalizers and $i: R\ra S$ an algebra extension in $\Cc$ with $S$ a left 
coflat object. If $\un{1}$ is a left $\ot$-generator and any object of $\Cc$ is right coflat then  
\[
\b: {\rm Nat}(-, (-\ot_RS)\circ F)\ra {\cal W}:=\{e:\un{1}\ra S\ot_RS\mid 
\mu^S_{S\ot_RS}(\Id_S\ot e)=\nu^S_{S\ot_RS}(e\ot\Id_S)\}
\]
given by $\b(\zeta)=\zeta_S\un{\eta}_S$ is well defined and an isomorphism. 
\end{lemma}
\begin{proof}
Similar to the one of \leref{Frobenius1}.  
First, observe that since $\zeta_S$ is right $S$-linear we have
\[
\nu^S_{S\ot_RS}(\b(\zeta)\ot\Id_S)=\nu^S_{S\ot_RS}(\zeta_S\ot \Id_S)(\un{\eta}_S\ot \Id_S)=
\zeta_S\un{m}_S(\un{\eta}_S\ot \Id_S)=\zeta_S.
\]
Let now $\epsilon: \un{1}\ra S$ be an arbitrary morphism in $\Cc$ and define 
$h_\epsilon=\un{m}_S(\epsilon\ot \Id_S): S\ra S$. Clearly $h_\epsilon$ is right 
$S$-linear, so by the naturality of $\zeta$ we deduce that $\widehat{h_\epsilon}\zeta_S=\zeta_Sh_\epsilon$. 
A direct computation ensures us that 
\[
\widehat{h_\epsilon}q^R_{S, S}=q^R_{S, S}(\un{m}_S(\epsilon\ot \Id_S)\ot \Id_S)=\mu^S_{S\ot_RS}(\epsilon 
\ot \Id_{S\ot_RS})q^R_{S, S},
\]
thus $\widehat{h_\epsilon}=\mu^S_{S\ot_RS}(\epsilon \ot \Id_{S\ot_RS})$. From here we obtain 
\[
\zeta_S\un{m}_S(\epsilon\ot \Id_S)=\zeta_Sh_\epsilon=\widehat{h_\epsilon}\zeta_S=\mu^S_{S\ot_RS}(\Id_S\ot\zeta_S)(\epsilon\ot \Id_S), 
\]
for all $\epsilon: \un{1}\ra S$ in $\Cc$. Using again that $\un{1}$ is a left $\ot$-generator in $\Cc$ 
we conclude that $\zeta_S$ is also left $S$-linear, and hence 
\[
\mu^S_{S\ot_RS}(\Id_S\ot \b(\zeta))=\mu^S_{S\ot_RS}(\Id_S\ot \zeta_S)(\Id_S\ot \un{\eta}_S)
=\zeta_S\un{m}_S(\Id_S\ot\un{\eta}_S)=\zeta_S=\nu^S_{S\ot_RS}(\b(\zeta)\ot \Id_S),
\]   
proving that $\b$ is well defined. We assert now that $\b$ is an isomorphism. To construct 
its inverse we proceed as follows. If $\mfM$ is a right $S$-module in $\Cc$ then since $\mfM$ is 
right coflat we have that 
\[
\xymatrix{
\mfM\ot S\ot R\ot S\ar@<1ex>[rr]^-{\Id_{\mfM\ot S}\ot \un{m}_S(i\ot \Id_S)}\ar@<-1ex>[rr]_-{\Id_\mfM\ot \un{m}_S(\Id_S\ot i)\ot \Id_S}&&\mfM \ot S\ot S
~~\ar[r]^-{\Id_\mfM\ot q^R_{S, S}}&~~\mfM\ot S\ot_RS 
}
\]  
is a coequalizer in $\Cc$. Furthermore, the morphism $q^R_{\mfM, S}(\nu^S_\mfM\ot \Id_S)$ fits 
in the universal property of this coequalizer, and so there is a unique morphism 
$\ov{\zeta}_\mfM: \mfM\ot S\ot_RS\ra \mfM\ot_RS$ such that $\ov{\zeta}_\mfM(\Id_\mfM\ot q^R_{S, S})=
q^R_{\mfM, S}(\nu^S_\mfM\ot \Id_S)$. 

We claim now that $\b^{-1}: {\cal W}\ra {\rm Nat}(-, (-\ot_RS)\circ F)$ given by 
\[
\b^{-1}(e)=\left(\zeta^e_\mfM=\ov{\zeta}_\mfM(\Id_\mfM\ot e): \mfM\ra \mfM\ot_RS\right)_{\mf{M}\in \Cc_S}
\]
is well defined and the inverse of $\beta$. To see that $\zeta^e_\mfM$ is right $S$-linear observe that 
\begin{eqnarray*}
\ov{\zeta}_\mfM(\nu^S_\mfM\ot \Id_{S\ot_RS})(\Id_{\mfM\ot S}\ot q^R_{S, S})
&=&\ov{\zeta}_\mfM(\Id_\mfM\ot q^R_{S, S})(\nu^S_\mfM\ot \Id_{S\ot S})\\
&=&q^R_{\mfM, S}(\nu^S_\mfM(\nu^S_\mfM\ot \Id_S)\ot \Id_S)\\
&=&q^R_{\mfM, S}(\nu^S_\mfM\ot \Id_S)(\Id_\mfM\ot \un{m}_S\ot \Id_S)\\
&=&\ov{\zeta}_\mfM(\Id_\mfM\ot \mu^S_{S\ot_RS})(\Id_{\mfM\ot S}\ot q^R_{S, S}),
\end{eqnarray*}
and since $\Id_{\mfM\ot S}\ot q^R_{S, S}$ is an epimorphism ($\mfM\ot S$ is right coflat) 
this shows that $\ov{\zeta}_\mfM(\nu^S_\mfM\ot \Id_{S\ot_RS})=\ov{\zeta}_\mfM(\Id_\mfM\ot \mu^S_{S\ot_RS})$. 
The latest equality allows us to compute 
\begin{eqnarray*}
\zeta^e_\mfM\nu^S_\mfM&=&\ov{\zeta}_\mfM(\Id_\mfM\ot e)\nu^S_\mfM=\ov{\zeta}_\mfM(\nu^S_\mfM\ot \Id_{S\ot_RS})
(\Id_{\mfM\ot S}\ot e)\\
&=&\ov{\zeta}_\mfM(\Id_\mfM\ot \mu^S_{S\ot_RS}(\Id_S\ot e))=
\ov{\zeta}_\mfM(\Id_\mfM\ot \nu^S_{S\ot_RS}(e\ot \Id_S)). 
\end{eqnarray*}  
Similar arguments leads us to 
\[
\nu^S_{\mfM\ot_RS}(\ov{\zeta}_\mfM\ot \Id_S)(\Id_\mfM\ot q^R_{S, S}\ot \Id_S)=
\ov{\zeta}_\mfM(\Id_\mfM\ot \nu^S_{S\ot_RS})(\Id_\mfM\ot q^R_{S, S}\ot \Id_S)
\]
and since $\mfM$ is right coflat and $S$ is left coflat we deduce that 
$\nu^S_{\mfM\ot_RS}(\ov{\zeta}_\mfM\ot \Id_S)=\ov{\zeta}_\mfM(\Id_\mfM\ot \nu^S_{S\ot_RS})$. 
From here it is immediate that 
\[
\nu^S_{\mfM\ot_RS}(\zeta^e_\mfM\ot \Id_S)=\ov{\zeta}_\mfM(\Id_\mfM\ot \nu^S_{S\ot_RS}(e\ot \Id_S)),
\]
and so $\zeta^e_\mfM$ is right $S$-linear, as desired. Next, if $f:\mfM\ra \mfN$ is a morphism in $\Cc_S$ 
then by arguments similar to the ones above we get that $\widehat{f}~\ov{\zeta}_\mfM=\ov{\zeta}_\mfN(f\ot\Id_{S\ot_RS})$. 
Thus $\hat{f}\zeta^e_\mfM=\ov{\zeta}_\mfN(f\ot\Id_{S\ot_RS})(\Id_\mfM\ot e)=\ov{\zeta}_\mfN(\Id_\mfN\ot e)f=\zeta^e_\mfN f$, 
and this ends the fact that $\b^{-1}$ is well defined. 

Let now $\zeta\in {\rm Nat}(-, (-\ot_RS)\circ F)$ and $e=\b^{-1}(\zeta)$, that is $e=\zeta_S\un{\eta}_S$. 
If $\mfM$ is a right $S$-module and $\epsilon: \un{1}\ra \mfM$ an arbitrary morphism then 
$\hbar_\epsilon=\nu^S_\mfM(\epsilon\ot \Id_S): S\ra \mfM$ is right $S$-linear. 
By the naturality of $\zeta$ we obtain that $\zeta_\mfM\hbar_\epsilon =\widehat{\hbar_\epsilon}\zeta_S$. Together with  
\[
\widehat{\hbar_\epsilon}q^R_{S, S}=q^R_{\mfM, S}(\hbar_\epsilon\ot \Id_S)=q^R_{\mfM, S}(\nu^S_\mfM(\epsilon\ot \Id_S)\ot \Id_S)
=\ov{\zeta}_\mfM(\epsilon\ot \Id_{S\ot_RS})q^R_{S, S},
\]  
this implies that 
\[
\zeta_\mfM\epsilon=\zeta_\mfM\hbar_\epsilon\un{\eta}_S=\widehat{\hbar_\epsilon}\zeta_S\un{\eta}_S=
\ov{\zeta}_\mfM(\epsilon\ot \Id_{S\ot_RS})e=\ov{\zeta}_\mfM(\Id_\mfM\ot e)\epsilon =\zeta^e_\mfM\epsilon.
\]
By the assumption that $\un{1}$ is a left $\ot$-generator for $\Cc$ we conclude that $\zeta=\zeta^e$ or, 
otherwise stated, $\b^{-1}\b(\zeta)=\zeta$. $\b^{-1}$ is also a right inverse for $\b$ since for any $e\in {\cal W}$ we have 
\[
\b\b^{-1}(e)=\b(\zeta^e)=\zeta^e_S\un{\eta}_S=\ov{\zeta}_S(\un{\eta}_S\ot \Id_{S\ot_RS})e=e,
\]
the last equality being a consequence of the following computation 
\[
\ov{\zeta}_S(\un{\eta}_S\ot \Id_{S\ot_RS})q^R_{S, S}=\ov{\zeta}_S(\Id_S\ot q^R_{S, S})(\un{\eta}_S\ot \Id_{S\ot S})
=q^R_{S, S}(\nu^S_S(\un{\eta}_S\ot \Id_S)\ot \Id_S)=q^R_{S, S}.
\] 
So our proof is finished. 
\end{proof}

One can prove now one of the main results of this section. 

\begin{theorem}\thlabel{FrobSepExt}
Let $\Cc$ be a monoidal category with coequalizers and assume that $\un{1}$ is a 
left $\ot$-generator and that any object of $\Cc$ is right coflat. If $i: R\ra S$ is an algebra morphism 
in $\Cc$ with $S$ a left coflat object then the following assertions are equivalent:
\begin{itemize}
\item[(i)] The functor restriction of scalars $F: \Cc_S\ra \Cc_R$ is a Frobenius functor;
\item[(ii)] There exist an $R$-bimodule morphism $\vartheta: S\ra R$ and a morphism 
$e:\un{1}\ra S\ot_RS$ in $\Cc$ such that the following diagrams are commutative
\[ 
\xymatrix{
S\ar[r]^-{\Id_S\ot e}\ar[d]_-{e\ot \Id_S}&S\ot (S\ot_RS)\ar[d]^{\mu^S_{S\ot_RS}}\\
(S\ot_RS)\ot S\ar[r]^-{\nu^S_{S\ot_RS}}&S\ot_RS
}~,~
\xymatrix{
\un{1}\ar[r]^-{e}\ar[d]_{\un{\eta}_S}&S\ot_RS\ar[d]^-{\tilde{\vartheta}}\\
S\ar[r]^{\Upsilon^{-1}_S}&S\ot_RR
}~,~
\xymatrix{
\un{1}\ar[r]^-{e}\ar[d]_{\un{\eta}_S}&S\ot_RS\ar[d]^-{\hat{\vartheta}}\\
S\ar[r]^{\Upsilon'^{-1}_S}&R\ot_RS
}~. 
\] 
\end{itemize}
\end{theorem} 
\begin{proof}
From the comments made before \deref{leftotgen} we have that $F$ is a Frobenius functor 
if and only if $-\ot_RS: \Cc_R\ra \Cc_S$ is a right adjoint for $F$, and 
from the previous two lemmas we have that the functor $-\ot_RS$ is a right adjoint for $F$ if and only if 
there exist $\vartheta\in {}_R{\rm Hom}_R(S, R)$ and $e\in {\cal W}$ such that 
\[
v_{F(\mfM)}F(\zeta_\mfM)=\Id_\mfM~,~\forall~\mfM\in \Cc_S~\mbox{\rm and}~
\widehat{v_X}\zeta_{X\ot_RS}=\Id_{X\ot_RS}~,~\forall~X\in \Cc_R,
\] 
where $v=\a^{-1}(\vartheta)=\left(v_X=\Upsilon_X\tilde{\vartheta}: X\ot_RS\ra X\right)_{X\in \Cc_R}$ 
in ${\rm Nat}(F\circ (-\ot_RS), -)$ and 
$\zeta:=\b^{-1}(e)=\left(\zeta^e_\mfM=\ov{\zeta}_\mfM(\Id_\mfM\ot e): \mfM\ra \mfM\ot_RS\right)_{\mf{M}\in \Cc_S}$ 
in ${\rm Nat}(-, (-\ot_RS)\circ F)$ are the natural transformations defined by 
$\vartheta$ and $e$, respectively.

We next prove that the first equality above is equivalent to the fact that the second 
diagram in (ii) is commutative. That the second equality above is equivalent to the fact that 
the last diagram in (ii) is commutative can be proved in a similar way, the details are left to the reader. 

Let us start by noting that $v_{F(\mfM)}F(\zeta_\mfM)=\Id_\mfM$ for all $\mfM\in \Cc_S$ is equivalent to 
$\Upsilon_\mfM\tilde{\vartheta}\ov{\zeta}_\mfM(\Id_\mfM\ot e)=\Id_\mfM$, for all $\mfM\in \Cc_S$, 
and since $\Upsilon_\mfM$ is an isomorphism the latest is equivalent to 
$\tilde{\vartheta}\ov{\zeta}_\mfM(\Id_\mfM\ot e)\Upsilon_\mfM=\Id_{\mfM\ot_RR}$ or, equivalently, to 
$\tilde{\vartheta}\ov{\zeta}_\mfM(\Id_\mfM\ot e)\nu^R_\mfM=q^R_{\mfM, R}$, for all $\mfM\in \Cc_S$. 
It is immediate that all these equivalent conditions are also equivalent to 
\begin{equation}\eqlabel{equiv1rightadj}
\tilde{\vartheta}\ov{\zeta}_\mfM(\Id_\mfM\ot e)(\nu^S_\mfM(\Id_\mfM\ot i)\ot \Id_{S\ot_RS})
(\Id_{\mfM\ot R}\ot e)=q^R_{\mfM, R}~,~\forall~\mfM\in \Cc_S. 
\end{equation}  
We claim that \equref{equiv1rightadj} holds if and only if $\Upsilon_S\tilde{\vartheta}e=\un{\eta}_S$. 
Indeed, observe first that a direct computation ensures us that 
\[
\tilde{\vartheta}\ov{\zeta}_\mfM(\nu^S_\mfM\ot \Id_{S\ot_RS})(\Id_{\mfM\ot S}\ot q^R_{S, S})
=q^R_{\mfM, R}(\nu^S_\mfM(\Id_\mfM\ot\un{m}_S)\ot \vartheta)=
\tilde{\vartheta}\ov{\zeta}_\mfM(\Id_\mfM\ot \ov{\zeta}_S)(\Id_{\mfM\ot S}\ot q^R_{S, S}).
\] 
Since every object of $\Cc$ is right coflat we get 
$\tilde{\vartheta}\ov{\zeta}_\mfM(\nu^S_\mfM\ot \Id_{S\ot_RS})=\tilde{\vartheta}\ov{\zeta}_\mfM(\Id_\mfM\ot \ov{\zeta}_S)$, 
and therefore \equref{equiv1rightadj} is equivalent to 
\[
\tilde{\vartheta}\ov{\zeta}_\mfM(\Id_\mfM\ot \ov{\zeta}_S(\Id_S\ot e)i)=q^R_{\mfM, R}~,~
\forall~\mfM\in \Cc_S.
\] 
One can check easily that 
\[
\Upsilon_\mfM\tilde{\vartheta}\ov{\zeta}_\mfM(\Id_\mfM\ot q^R_{S, S})
=\nu^R_\mfM(\nu^S_\mfM\ot \vartheta)=\nu^S_\mfM(\Id_\mfM\ot \Upsilon_S\tilde{\vartheta})(\Id_\mfM\ot q^R_{S, S}),
\]
and so $\Upsilon_\mfM\tilde{\vartheta}\ov{\zeta}_\mfM=\nu^S_\mfM(\Id_\mfM\ot \Upsilon_S\tilde{\vartheta})$. 
Hence \equref{equiv1rightadj} is actually equivalent to 
\begin{equation}\eqlabel{equiv2rightadj}
\nu^S_\mfM(\Id_\mfM\ot \Upsilon_S\tilde{\vartheta}\ov{\zeta}_S(\Id_S\ot e)i)=\nu^R_\mfM~,~
\forall~\mfM\in \Cc_S.
\end{equation}
Notice that if $\Upsilon_S\tilde{\vartheta}\ov{\zeta}_S(\Id_S\ot e)=\Id_S$ then \equref{equiv2rightadj} 
is satisfied since $\nu^S_\mfM(\Id_\mfM\ot i)=\nu^R_\mfM$, for all $\mfM\in \Cc_S$. The converse 
is also true since if we get $\mfM=S$ in \equref{equiv2rightadj} with $\nu^S_S=\un{m}_S$ and compose it 
to the right with $\Id_S\ot \un{\eta}_R$ we obtain 
\[
\Id_S=\un{m}_S(\Id_S\ot \Upsilon_S\tilde{\vartheta}\ov{\zeta}_S(\Id_S\ot e)\un{\eta}_S).
\]
Straightforward computations lead us to 
\begin{eqnarray*}
&&\un{m}_S(\Id_S\ot \Upsilon_S)(\Id_S\ot q^R_{S, R})=\un{m}_S(\un{m}_S\ot i)=\Upsilon_S\mu^S_{S\ot_RR}(\Id_S\ot q^R_{S, R}){~\rm and}\\
&&\mu^S_{S\ot_RR}(\Id_S\ot \tilde{\vartheta})(\Id_S\ot q^R_{S, S})=q^R_{S, R}(\un{m}_S\ot\vartheta)=
\tilde{\vartheta}\ov{\zeta}_S(\Id_S\ot q^R_{S, S}),
\end{eqnarray*}
and since $S$ is right coflat, too, we obtain 
\begin{equation}\eqlabel{equiv3rightadj}
(\un{m}_S(\Id_S\ot \Upsilon_S)=\Upsilon_S\mu^S_{S\ot_RR}~~{\rm and}~~
\mu^S_{S\ot_RR}(\Id_S\ot \tilde{\vartheta})=\tilde{\vartheta}\ov{\zeta}_S.
\end{equation}  
Notice also that 
\[
\ov{\zeta}_S(\Id_S\ot e)\un{\eta}_S=\zeta^e_S\un{\eta}_S=\beta^{-1}(\zeta^e)=\b^{-1}(\b(e))=e.
\]
Summing up we get  
\[
\Id_S=\un{m}_S(\Id_S\ot \Upsilon_S\tilde{\vartheta}e)=\Upsilon_S\mu^R_{S\ot_RR}(\Id_S\ot \tilde{\vartheta}e)
=\Upsilon_S\tilde{\vartheta}\ov{\zeta}_S(\Id_S\ot e),
\]
as needed. Otherwise stated, we have shown that \equref{equiv2rightadj} is equivalent to 
$\Upsilon_S\tilde{\vartheta}\ov{\zeta}_S(\Id_S\ot e)=\Id_S$. If we compose it 
to the right with $\un{\eta}_S$ we get 
\[
\un{\eta}_S=\Upsilon_S\tilde{\vartheta}\ov{\zeta}_S(\Id_S\ot e)\un{\eta}_S=
\Upsilon_S\tilde{\vartheta}\zeta^e_S\un{\eta}_S=\Upsilon_S\tilde{\vartheta}e,
\]
i.e., the second diagram in the statement is commutative. Finally, if 
$\Upsilon_S\tilde{\vartheta}e=\un{\eta}_S$ we then have 
\[
\Upsilon_S\tilde{\vartheta}\ov{\zeta}_S(\Id_S\ot e)
\equal{\equref{equiv3rightadj}}\Upsilon_S\mu^S_{S\ot_RR}(\Id_S\ot\tilde{\vartheta})(\Id_S\ot e)
\equal{\equref{equiv3rightadj}}\un{m}_S(\Id_S\ot \Upsilon_S\tilde{\vartheta}e)=
\un{m}_S(\Id_S\ot \un{\eta}_S)=\Id_S,
\] 
and this finishes the proof of the theorem. 
\end{proof}

The next result generalizes \cite[Theorem 27 1.\&2.]{cmz} to an algebra extension in a monoidal category. 

\begin{theorem}\thlabel{SepExtMon}
Let $\Cc$ be a monoidal category with equalizers and assume that $\un{1}$ is a left $\ot$-generator for 
$\Cc$. Let $i: R\ra S$ be an algebra morphism in $\Cc$ 
with $S$ a left coflat object. Then the following assertions hold:
\begin{itemize}
\item[(i)] If any object of $\Cc$ is right coflat then the restriction 
of scalars functor $F: \Cc_S\ra \Cc_R$ is separable if and only if there 
exists a morphism $e\in {\cal W}=\{e: \un{1}\ra S\ot_RS\mid \mu^S_{S\ot_RS}(\Id_S\ot e)=\nu^S_{S\ot_RS}(e\ot \Id_S)\}$ 
such that $\un{m}_S^Re=\un{\eta}_S$, where $\un{m}^R_S: S\ot_RS\ra S$ is the unique morphism in $\Cc$ obeying $\un{m}^R_Sq^R_{S, S}=\un{m}_S$;
\item[(ii)] The induction functor $-\ot_RS: \Cc_R\ra \Cc_S$ is separable if and only if there exists an $R$-bimodule 
morphism $\vartheta: S\ra R$ such that $\vartheta\un{\eta}_S=\un{\eta}_R$.    
\end{itemize} 
\end{theorem}
\begin{proof}
(i) The functor $-\ot_RS$ is a left adjoint of $F$, so by the Rafael's theorem (see \seref{FrobSepFunct}) 
it follows that $F$ is separable if and only if the counit $\va$ defined in \equref{indleftadjrestscal} cosplits. 
By \leref{Frobenius2} this happens if and only if there is $e\in {\cal W}$ such that if 
\[
\zeta^e=\left(\zeta^e_\mfM=\ov{\zeta}_\mfM(\Id_\mfM\ot e): \mfM\ra \mfM\ot_RS\right)_{\mf{M}\in \Cc_S}
\]  
is the natural transformation associated to $e$ as in \leref{Frobenius2} then $\va\circ \zeta^e$ is the 
identity natural transformation of $1_{\Cc_S}$. Clearly this is equivalent to the existence of an element 
$e\in {\cal W}$ such that 
\[
\ov{\nu^S_\mfM}\ov{\zeta}_\mfM(\Id_\mfM\ot e)=\Id_\mfM~,~\forall~\mfM\in \Cc_S~.
\]
Since $\ov{\nu^S_\mfM}\ov{\zeta}_\mfM(\Id_\mfM\ot q^R_{S, S})=\ov{\nu^S_\mfM}q^R_{\mfM, S}
(\nu^S_\mfM\ot\Id_S)=\nu^S_\mfM(\nu^S_\mfM\ot \un{m}_S)=\nu^S_\mfM(\Id_\mfM\ot \un{m}^R_Sq^R_{S, S})$, 
by the right coflatness of $\mfM$ we deduce that $\ov{\nu^S_\mfM}\ov{\zeta}_\mfM=\nu^S_\mfM(\Id_\mfM\ot \un{m}^R_S)$. 
Thus $F$ is separable if and only if there exists $e\in {\cal W}$ such that 
$\nu^S_\mfM(\Id_\mfM\ot \un{m}^R_Se)=\Id_\mfM$, for all $\mfM\in \Cc_S$. If $\un{m}^R_Se=\un{\eta}_S$ 
then the latest condition is clearly satisfied. For the converse take $\mfM=S$ and $\nu^S_S=\un{m}_S$. We then have 
$\un{m}_S(\Id_S\ot \un{m}^R_Se)=\Id_S$, and therefore 
\[
\un{\eta}_S=\un{m}_S(\Id_S\ot \un{m}^R_Se)\un{\eta}_S=\un{m}_S(\un{\eta}_S\ot \Id_S)\un{m}^R_Se=\un{m}^R_Se,
\]
as needed. 

(ii) Similar to the one given for (i), so we leave it to the reader. 
\end{proof}

\section{Frobenius and separable algebra extensions in monoidal categories}\selabel{FrobSepAlgMon}
\setcounter{equation}{0} 
Let $k$ be a field and $A$ a $k$-algebra. We say that $A$ is a Frobenius algebra iff $A$ is isomorphic 
to its $k$-dual $A^*:={\rm Hom}_k(A, k)$ as a left or right $A$-module, and we call $A$ separable if the 
multiplication of $A$ cosplits in the category of $A$-bimodules. It is immediate that these definitions can be 
adapted for an algebra in a monoidal category; note that in the Frobenius case the existence of the dual object of 
$A$ can be avoided if we use the characterization of a Frobenius algebra given in terms of a Frobenius pair or the  
one given in the language of Frobenius functors. Concretely, we have the following notions. 

\begin{definition}\delabel{defFrobext}
Let $\Cc$ be a monoidal category and $A$ an algebra in $\Cc$. 

(i) $A$ is called Frobenius if there exists a pair $(\vartheta, e)$ with 
$\vartheta: A\ra \un{1}$ and $e:\un{1}\ra A\ot A$ morphisms in $\Cc$ such that 
\[
(\un{m}_A\ot \Id_A)(\Id_A\ot e)=(\Id_A\ot \un{m}_A)(e\ot \Id_A)~,~(\vartheta\ot\Id_A)e=\un{\eta}_A=
(\Id_A\ot \vartheta)e.
\] 

(ii) $A$ is called separable if there exists a morphism $\gamma: A\ra A\ot A$ of $A$-bimodules such that 
$\un{m}_A\gamma=\Id_A$, where both $A$ and $A\ot A$ are viewed as $A$-bimodules via the multiplication $\un{m}_A$ of $A$. 
\end{definition}

It is clear that $A$ is Frobenius in $\Cc$ if and only if the forgetful functor $F: \Cc_A\ra \Cc$ is 
Frobenius, providing that $\un{1}$ is a left $\ot$-generator for $\Cc$. Note that in this case the coflatness 
condition and the existence of the coequalizers in $\Cc$ can be deleted from the statement of \thref{FrobSepExt}, 
since when the algebra extension is given by the unit morphism $\un{\eta}_A: \un{1}\ra A$ the tensor product 
over the source algebra is just the tensor product $\ot$ of $\Cc$.  

In the separable case we have a similar result. 

\begin{proposition}\prlabel{SepAlgSepFun}
Let $\Cc$ be a monoidal category such that $\un{1}$ is a left $\ot$-generator for $\Cc$. If $A$ is an algebra 
in $\Cc$ then $A$ is separable if and only if the forgetful functor $F: \Cc_A\ra \Cc$ is a separable functor.    
\end{proposition}
\begin{proof}
It is straightforward to see that an algebra $A$ is separable if and only if 
there exists a morphism $e: \un{1}\ra A\ot A$ such that 
$(\un{m}_A\ot \Id_A)(\Id_A\ot e)=(\Id_A\ot \un{m}_A)(e\ot \Id_A)$ and $\un{m}_Ae=\un{\eta}_A$. 
Indeed, for $e$ as above if we define $\gamma=(\un{m}_A\ot \Id_A)(\Id_A\ot e)=(\Id_A\ot e)\un{m}_A$ it then 
follows that $\gamma$ is $A$-bilinear and $\un{m}_A\gamma=\Id_A$. Conversely, if $\gamma$ is an $A$-bimodule 
morphism in $\Cc$ then $e=\gamma \un{\eta}_A$ has the desired properties, we leave the details to the reader.  

But when $\un{1}$ is a left $\ot$-generator in $\Cc$ the existence of such a morphism $e$ is equivalent 
to the fact that the functor $F$ is separable, cf. \thref{SepExtMon}. So our proof is complete.  
\end{proof}

\begin{definition}
If $A$ is a separable algebra in a monoidal category $\Cc$ we then call the morphism $e: \un{1}\ra A\ot A$ 
satisfying $(\un{m}_A\ot \Id_A)(\Id_A\ot e)=(\Id_A\ot e)\un{m}_A$ and $\un{m}_Ae=\un{\eta}_A$ 
the separability morphism of $A$. 
\end{definition}

Another characterization for a separable algebra can be given if we assume that 
the unit object of the category is projective. The result below is a straightforward 
generalization of the classical result asserting that an algebra $A$ over a commutative ring $k$ 
is separable if and only if $A$ is a projective $A$-bimodule. 

\begin{proposition}\prlabel{sepversusproj} 
Let $\Cc$ be a monoidal category having the unit $\un{1}$ a projective object. Then an algebra 
$A$ in $\Cc$ is separable if and only if $A$ is projective as an $A$-bimodule. 
\end{proposition} 
\begin{proof}
Assume that $A$ is a separable algebra and let $e: \un{1}\ra A\ot A$ be its separability morphism. 
Consider $\mfM\stackrel{p}{\rightarrow} \mfN\ra 0$ a short exact sequence 
of $A$-bimodules in $\Cc$ and $f: A\ra \mfN$ a morphism in ${}_A\Cc_A$. We have to show that there is 
an $A$-bimodule morphism $g: A\ra \mfM$ such that $pg=f$.  

If $\tilde{f}:=f\un{\eta}_A$ then clearly $f=\mu^A_\mfN(\Id_A\ot \tilde{f})=\nu^A_\mfN(\tilde{f}\ot \Id_A)$. 
Since $\un{1}$ is projective it follows that there exists a morphism $\tilde{g}: \un{1}\ra \mfM$ in $\Cc$ such that 
$p\tilde{g}=\tilde{f}$. So if we define $g: A\ra \mfM$ given by 
\[
g=\nu^A_\mfM(\mu^A_\mfM\ot \Id_A)(\Id_A\ot \tilde{g}\ot \Id_A)(\un{m}_A\ot \Id_A)(\Id_A\ot e) 
\]  
then $g$ is $A$-bilinear since $\un{m}_A$ is associative and $(\un{m}_A\ot \Id_A)(\Id_A\ot e)=(\Id_A\ot e)\un{m}_A$, and 
\begin{eqnarray*}
pg&=&\nu^A_\mfM(\mu^A_\mfM\ot \Id_A)(\Id_A\ot f\un{\eta}_A\ot \Id_A)(\un{m}_A\ot \Id_A)(\Id_A\ot e)\\
&=&\nu^A_\mfM(f\ot \Id_A)(\un{m}_A\ot \Id_A)(\Id_A\ot e)\\
&=&f\un{m}_A(\Id_A\ot \un{m}_Ae)=f\un{m}_A(\Id_A\ot \un{\eta}_A)=f, 
\end{eqnarray*}
as required. 

Conversely, if $A$ is projective in ${}_A\Cc_A$ then since $\un{m}_A: A\ot A\ra A$ is an 
epimorphism in $\Cc$ it follows that there exists an $A$-bilinear morphism 
$\gamma: A\ra A\ot A$ such that $\un{m}_A\gamma=\Id_A$, and so $A$ is a separable algebra. 
Observe that for this implication we do not need the unit object to be projective. 
\end{proof}

Motivated by the above results and the theory of Frobenius (respectively separable) algebra extensions in a category 
of vector spaces we propose the following terminology. 

\begin{definition}
Let $\Cc$ be a monoidal category with coequalizers and such that any object of it is coflat, and let 
$i: R\ra S$ be an algebra extension in $\Cc$. 

(i) We call the algebra extension $i: R\ra S$ Frobenius if there exist an $R$-bimodule morphism 
$\vartheta: S\ra R$ and a morphism $e: \un{1}\ra S\ot_RS$ in $\Cc$ such that the three conditions 
stated in \thref{FrobSepExt} (ii) are satisfied. If this is 
the case then we call the pair $(\vartheta, e)$ the Frobenius pair of the Frobenius extension $i: R\ra S$. 
Furthermore, we call $\vartheta$ the Frobenius morphism and $e$ the Casimir morphism of the Frobenius 
algebra extension $i$.

(ii) We call the algebra extension $i: R\ra S$ separable if there exists a morphism 
$e\in {\cal W}=\{e: \un{1}\ra S\ot_RS\mid \mu^S_{S\ot_RS}(\Id_S\ot e)=\nu^S_{S\ot_RS}(e\ot \Id_S)\}$ 
such that $\un{m}_S^Re=\un{\eta}_S$, where $\un{m}^R_S: S\ot_RS\ra S$ is the unique morphism in $\Cc$ 
obeying $\un{m}^R_Sq^R_{S, S}=\un{m}_S$. 
\end{definition}
  
In the case when $\un{1}$ is a left $\ot$-generator for $\Cc$ the definitions above are equivalent to 
the fact that the functor restriction of scalars is Frobenius, respectively separable.  
We next show that these notions reduce to the notions of 
Frobenius, respectively separable, algebra in a suitable monoidal category, provided that an extra condition is fulfilled. 
To this end we first need a preliminary result.

\begin{lemma}\lelabel{FrobSepExtMonInt}
Let $\Cc$ be a monoidal category with coequalizers and $i:R\ra S$ an algebra 
extension in $\Cc$ such that $R$ is coflat and $S$ is left coflat and left robust, so that 
$S\in {}^!_R\Cc_R$. Denote by $\mu'^S_{S\ot_RS}=\widehat{\un{m}^R_S}\Gamma'_{S, S, S}$ and $\nu'^S_{S\ot_RS}=
\widetilde{\un{m}^R_S}\Sigma'_{S, S, S}$ the canonical morphisms in ${}_R\Cc_R$ 
that give on $S\ot_RS$ an $S$-bimodule structure in ${}_R\Cc_R$, and by $\un{m}^R_S$ 
the multiplication of the algebra $S$ in ${}^!_R\Cc_R$. 
Then there exists a one to one correspondence between ${\cal W}'$, 
the set of morphisms $\mf{e}: R\ra S\ot_RS$ in ${}_R\Cc_R$ obeying 
\[
\mu'^S_{S\ot_RS}\tilde{\mf{e}}\Upsilon^{-1}_S=\nu'^S_{S\ot_RS}\hat{\mf{e}}\Upsilon'^{-1}_S~,~
\] 
and 
${\cal W}=\{e: \un{1}\ra S\ot_RS\mid \mu^S_{S\ot_RS}(\Id_S\ot e)=\nu^S_{S\ot_RS}(e\ot \Id_S)\}$. 
\end{lemma}
\begin{proof}
Let us start by noting that $S$ admits an algebra structure in 
${}^!_R\Cc_R$ if and only if $S$ admits an algebra structure in $\Cc$ 
such that the unit morphism $i$ of $S$ in ${}^!_R\Cc_R$ becomes an algebra 
morphism in $\Cc$. Then the $R$-bimodule structure of $S$ identifies to the one induced by $i$. 

Also, directly from the definitions it follow that the two structure morphisms $\mu'^S_{S\ot_RS}$ and 
$\nu'^S_{S\ot_RS}$ are completely determined by the equalities  
\[
\mu'^S_{S\ot_RS}q^R_{S, S\ot_RS}=\mu^S_{S\ot_RS}~~{\rm and}~~
\nu'^S_{S\ot_RS}q^R_{S\ot_RS, S}=\nu^S_{S\ot_RS}~,
\]
respectively. 

Now, if $\mf{e}$ is an element of ${\cal W}'$ we then claim that $e=\mf{e}\un{\eta}_R: \un{1}\ra S\ot_RS$ 
belongs to ${\cal W}$. Indeed,  
\[
\mu'^S_{S\ot_RS}\tilde{\mf{e}}\Upsilon^{-1}_S=\mu'^S_{S\ot_RS}\tilde{\mf{e}}q^R_{R, S}(\Id_S\ot \un{\eta}_R)
=\mu'^S_{S\ot_RS}q^R_{S, S\ot_RS}(\Id_S\ot \mf{e}\un{\eta}_R)=\mu^S_{S\ot_RS}(\Id_S\ot e),
\]
and, similarly, $\nu'^S_{S\ot_RS}\hat{\mf{e}}\Upsilon'^{-1}_S=\nu^S_{S\ot_RS}(e\ot \Id_S)$, 
proving the commutativity of the first diagram in the statement (ii) of \thref{FrobSepExt}.

Conversely, if $e$ is in ${\cal W}$ we then define $\mf{e}=\mu^S_{S\ot_RS}(i\ot e)$. 
That $\mf{e}$ is left $R$-linear follows easily from the fact that $\mu^S_{S\ot_RS}$ defines a left 
$S$-module structure on $S\ot_RS$ and since $i$ is an algebra map, the details are left to the reader. 
It is also right $R$-linear since 
\begin{eqnarray*}
\nu^R_{S\ot_RS}(\mf{e}\ot\Id_R)&=&\nu^S_{S\ot_RS}(\mu^S_{S\ot_RS}(\Id_S\ot e)i\ot i)\\
&=&\nu^S_{S\ot_RS}(\nu^S_{S\ot_RS}(e\ot\Id_S)i\ot i)\\
&=&\nu^S_{S\ot_RS}(e\ot\Id_S)\un{m}_S(i\ot i)\\
&=&\mu^S_{S\ot_RS}(\Id_S\ot e)i\un{m}_R=\mf{e}\un{m}_R,
\end{eqnarray*}
as needed. Furthermore, since ${\mf e}\un{\eta}_R=\mu^S_{S\ot_RS}(i\ot e)\un{\eta}_R=\mu^S_{S\ot_RS}
(\un{\eta}_S\ot \Id_{S\ot_RS})e=e$, as in the case of the first correspondence one have that
\[
\mu'^S_{S\ot_RS}\tilde{\mf{e}}\Upsilon^{-1}_S=\mu^S_{S\ot_RS}(\Id_S\ot e)=
\nu^S_{S\ot_RS}(e\ot \Id_S)=\nu'^S_{S\ot_RS}\hat{\mf{e}}\Upsilon'^{-1}_S,
\] 
and so $\mf{e}$ is an element in ${\cal W}'$. We also have shown that ${\cal W}\ni e\mapsto \mf{e}\mapsto \mf{e}\un{\eta}_R=e$. 
Using that an element $\mf{e}$ is left $R$-linear we get that 
\[
{\cal W}'\ni \mf{e}\mapsto e=\mf{e}\un{\eta}_R\mapsto \mu^S_{S\ot_RS}(i\ot \mf{e})(\Id_R\ot \un{\eta}_R)=
\mf{e}\un{m}_R(\Id_R\ot \un{\eta}_R)=\mf{e},
\]
and this finishes our proof.
\end{proof}

\begin{proposition}\prlabel{Frobalgextcharact}
Let $\Cc$ be a monoidal category with coequalizers and $i:R\ra S$ an algebra 
extension in $\Cc$ such that $R$ is coflat and $S$ is coflat and left robust, so that 
$S\in {}^!_R\Cc_R$. Then the following assertions hold.
\begin{itemize}
\item[(i)] The extension $i: R\ra S$ is Frobenius if and only if $S$ is a Frobenius algebra in ${}^!_R\Cc_R$.
\item[(ii)] The extension $i: R\ra S$ is separable if and only if $S$ is a separable algebra in ${}^!_R\Cc_R$.
\end{itemize} 
\end{proposition}
\begin{proof}
(i) By \deref{defFrobext} applied to $S\in {}^!_R\Cc_R$ we have that $S$ 
is a Frobenius algebra in ${}^!_R\Cc_R$ if and only if there 
exists a pair $(\vartheta, \mf{e})$ with $\vartheta: S\ra R$ and 
$\mf{e}: R\ra S\ot_RS$ morphisms in ${}_R\Cc_R$ such that 
\[
\mu'^S_{S\ot_RS}\tilde{\mf{e}}\Upsilon^{-1}_S=\nu'^S_{S\ot_RS}\hat{\mf{e}}\Upsilon'^{-1}_S~,~
\Upsilon_S\tilde{\vartheta}\mf{e}=i~{\rm and}~\Upsilon'_S\hat{\vartheta}\mf{e}=i, 
\] 
where $\mu'^S_{S\ot_RS}$ and $\nu'^S_{S\ot_RS}$ are the morphisms in ${}_R\Cc_R$ defined in 
the statement of \leref{FrobSepExtMonInt}. 

We show now that to give a pair $(\vartheta, \mf{e})$ as above is equivalent 
to give a pair $(\vartheta, e)$ with $\vartheta: S\ra R$ morphism in ${}_R\Cc_R$ and 
$e: \un{1}\ra S\ot_RS$ morphism in $\Cc$ such that the three diagrams in \thref{FrobSepExt} (ii) 
are commutative, and this would end the proof. 

Indeed, $e$ can be obtained from $\mf{e}$ as $e=\mf{e}\un{\eta}_R$. Then by \leref{FrobSepExtMonInt} 
the first diagram in \thref{FrobSepExt} (ii) is commutative; the commutativity of the other two follow easily 
by composing the remaining properties of $\mf{e}$ to the right with $\un{\eta}_R$.  

Conversely, to a pair $(\vartheta, e)$ corresponds $(\vartheta, \mf{e})$ with $\mf{e}=\mu^S_{S\ot_RS}(i\ot e)$. 
The first required property on $\mf{e}$ is satisfied because of \leref{FrobSepExtMonInt}. Moreover, 
by a direct computation we can check that 
\[
\un{m}_S(i\ot \Id_S)(\Id_R\ot \Upsilon_S\tilde{\vartheta})(\Id_R\ot q^R_{S, S})
=\un{m}_S(\un{m}_S(i\ot\Id_S)\ot i\vartheta)=\Upsilon_S\tilde{\vartheta}\mu^R_{S\ot_RS}(\Id_R\ot q^R_{S, S}),
\]
and thus $\un{m}_S(i\ot \Id_S)(\Id_R\ot \Upsilon_S\tilde{\vartheta})=\Upsilon_S\tilde{\vartheta}\mu^R_{S\ot_RS}$, 
since $R$ is right coflat, too. One then have 
\[
\Upsilon_S\tilde{\vartheta}\mf{e}=\Upsilon_S\tilde{\vartheta}\mu^R_{S\ot_RS}(\Id_R\ot e)
=\un{m}_S(i\ot\Id_S)(\Id_R\ot \Upsilon_S\tilde{\vartheta}e)=\un{m}_S(\Id_S\ot \un{\eta}_S)i=i,
\] 
as required. In a similar manner we can show that $\Upsilon'_S\hat{\vartheta}\mf{e}=i$, 
we leave the verification of the details to the reader. 

(ii) Similar to the one above. By \deref{defFrobext} $S$ is a separable algebra in ${}_R^!\Cc_R$ 
if and only if there exists $\mf{e}\in {\cal W}'$ such that $\un{m}^R_S\mf{e}=i$. 
Hence it is sufficient to show that the bijection in \leref{FrobSepExtMonInt} behaves well with 
respect to the extra property of $\mf{e}$. So we show that giving an element $\mf{e}\in {\cal W}'$ 
such that $\un{m}^R_S\mf{e}=i$ is equivalent to giving an element 
$e\in {\cal W}$ such that $\un{m}^R_Se=\un{\eta}_S$. 

Indeed, on one hand, to $\mf{e}\in {\cal W}'$ corresponds $e=\mf{e}\un{\eta}_R$, and so 
$\un{m}^R_Se=\un{m}^R_S\mf{e}\un{\eta}_R=i\un{\eta}_R=\un{\eta}_S$, as 
desired. On the other hand, if $e\in {\cal W}$ such that $\un{m}^R_Se=\un{\eta}_S$ then 
since $\mf{e}=\mu^S_{S\ot_RS}(i\ot e)$ we have 
\[
\un{m}^R_S\mf{e}=\un{m}^R_S\mu^S_{S\ot_RS}(i\ot e)\equal{(*)}\un{m}_S(\Id_S\ot \un{m}^R_Se)i=
\un{m}_S(\Id_S\ot \un{\eta}_S)i=i,
\] 
as needed. Note that $(*)$ is a consequence of the fact that $S$ is right coflat and 
$\un{m}^R_S\mu^S_{S\ot_RS}(\Id_S\ot q^R_{S, S})=\un{m}_S(\un{m}_S\ot \Id_S)=\un{m}_S(\Id_S\ot \un{m}^R_Sq^R_{S, S})$. 
\end{proof}

\begin{corollary}\colabel{SepExtVsSep}
Let $\Cc$ be a monoidal category for which any object is coflat and left robust and such that 
$\un{1}$ is either a projective object or a left $\ot$-generator for $\Cc$. If $i: R\ra S$ is an algebra extension 
in $\Cc$ with $R$ separable then $S$ is separable if and only if the algebra extension $i: R\ra S$ is separable. 
\end{corollary}
\begin{proof}
Assume first that $\un{1}$ is projective. 
By \prref{sepversusproj} we know that $R$, the unit object of ${}_R\Cc_R$, is projective in ${}_R\Cc_R$. 
Hence the algebra extension $i: R\ra S$ is separable if and only if $S$ is a separable algebra in ${}_R\Cc_R$, 
if and only if $S$ is projective as an $S$-bimodule in ${}_R\Cc_R$. But ${}_S({}_R\Cc_R)_S$ identifies with 
${}_S\Cc_S$, as any object of ${}_S({}_R\Cc_R)_S$ inherits the $R$-module structures from the $S$-module ones 
(due to the $R$-balanced conditions). Thus the algebra extension $i: R\ra S$ is separable if and only if $S$ 
is projective in ${}_S\Cc_S$, and since $\un{1}$ is projective this is equivalent to the fact that 
$S$ is a separable algebra.   

Assume now that $\un{1}$ is a left $\ot$-generator for $\Cc$. Since $R$ is a separable algebra in $\Cc$ 
it follows by \prref{SepAlgSepFun} that the forgetful functor $F': \Cc_R\ra \Cc$ is separable. Now, 
the extension $i: R\ra S$ is separable if and only if the restriction of scalars functor 
$F: \Cc_S\ra \Cc_R$ is separable, cf. \thref{SepExtMon}, and again by \prref{SepAlgSepFun} we have that 
$S$ is a separable algebra in $\Cc$ if and only if the forgetful functor $F{''}: \Cc_S\ra \Cc$ is separable. 

Thus, if the extension $i: R\ra S$ is separable then according to \cite[Proposition 46 (1)]{cmz} 
the functor $F{''}=F'\circ F$ is separable, and so $S$ is a separable algebra in $\Cc$. Conversely, if 
$S$ is separable then $F{''}$ is a separable functor and since $F{''}=F'\circ F$ by 
\cite[Proposition 46 (2)]{cmz} we get that $F$ is a separable functor, and hence the extension $i: R\ra S$ is 
separable. 
\end{proof}

A result of Eilenberg and Nakayama asserts that any separable algebra over a field is a Frobenius algebra 
(even more, it is a symmetric algebra, see \cite{EiNak}). As we explained in Introduction, 
it is an open (and quite difficult) problem to see if this result remains true in the setting of monoidal categories. 
Nevertheless, what one can prove is a sort of converse for the above result. It measures how far is a Frobenius 
algebra to be separable.

\begin{proposition}\prlabel{WhenFrobisSep}
Let $A$ be a Frobenius algebra in a monoidal category $\Cc$ and $(\nu, e)$ a Frobenius 
pair for $A$. Then $A$ is separable if and only if there exists a morphism $\a: \un{1}\ra A$ 
such that $\un{m}_A(\un{m}_A\ot\Id_A)(\Id_A\ot \a\ot \Id_A)e=\un{\eta}_A$.
\end{proposition}
\begin{proof}
Assume that $A$ is separable and let ${\bf e}:\un {1}\ra A\ot A$ be a separability morphism for it. 
If we define $\a:=(\Id_A\ot \nu){\bf e}: \un{1}\ra A$ then 
\begin{eqnarray*}
(\un{m}_A\ot\Id_A)(\Id_A\ot \a\ot \Id_A)e&=&(\un{m}_A\ot \Id_A)(\Id_{A\ot A}\ot \nu\ot \Id_A)(\Id_A\ot {\bf e}\ot \Id_A)e\\
&=&(\Id_A\ot \nu\ot \Id_A)((\un{m}_A\ot \Id_A)(\Id_A\ot {\bf e})\ot \Id_A)e\\
&=&(\Id_A\ot \nu\ot \Id_A)(\Id_A\ot \un{m}_A\ot \Id_A)(\Id_{A\ot A}\ot e){\bf e}\\
&=&(\Id_A\ot (\nu\ot \Id_A)(\Id_A\ot \un{m}_A)(e\ot \Id_A)){\bf e}\\
&=&(\Id_A\ot \un{m}_A((\nu\ot \Id_A)e\ot \Id_A)){\bf e}\\
&=&(\Id_A\ot \un{m}_A(\un{\eta}_A\ot \Id_A)){\bf e}={\bf e},
\end{eqnarray*}
from where we get that $\un{m}_A(\un{m}_A\ot\Id_A)(\Id_A\ot \a\ot \Id_A)e=\un{m}_A{\bf e}=\un{\eta}_A$, as desired. 

Conversely, let $\a: \un{1}\ra A$ be a morphism in $\Cc$ satisfying the condition in the statement. Then a simple computation 
ensures us that ${\bf e}:=(\un{m}_A\ot \Id_A)(\Id_A\ot \a\ot \Id_A)e: \un{1}\ra A\ot A$ is a separability 
morphism for $A$, and so $A$ is separable. 
\end{proof}

\begin{corollary}\colabel{WhenFrobExtIsSep}
Let $\Cc$ be a monoidal category with coequalizers and $i: R\ra S$ a Frobenius algebra extension
in $\Cc$ such that $R$ is coflat and $S$ is coflat and left robust. Then the extension $i: R\ra S$ is separable 
if and only if there exists a morphism $\alpha_0: \un{1}\ra S$ in $\Cc$ such that 
\[
\un{m}_S(i\ot \a_0)=\un{m}_S(\a_0\ot i)~~\mbox{and}~~\un{m}^R_S~\widehat{\un{m}_S}~\widehat{\Id_S\ot \a_0}~e=\un{\eta}_S,
\] 
where $e:\un{1}\ra S\ot_RS$ is the morphism in $\Cc$ associated to the Frobenius extension $i: R\ra S$ as in the proof of 
\prref{Frobalgextcharact}. 
\end{corollary}
\begin{proof}
By \prref{Frobalgextcharact} we know that $i: R\ra S$ is a Fobenius/separable extension if and only if 
$S$ is a Frobenius/separable algebra in ${}^!_R\Cc_R$. If we denote by $(\nu , {\bf e})$ the 
Frobenius system of $S$ in ${}^!_R\Cc_R$ then $S$ is separable in ${}^!_R\Cc_R$ if and only 
if there exists $\a: R\ra S$ an $R$-bimodule morphism in 
$\Cc$ such that 
\[
\un{m}^R_S\widehat{\un{m}^R_S\widetilde{\alpha}\Upsilon_S^{-1}}{\bf e}=i.
\]
Clearly $\a$ is uniquely determined by a morphism $\a_0: \un{1}\ra S$ obeying $\un{m}_S(i\ot \a_0)=\un{m}_S(\a_0\ot i)$. 
Also, by \leref{FrobSepExtMonInt} we have that ${\bf e}$ is completely determined by $e={\bf e}\un{\eta}_R: \un{1}\ra S\ot_RS$, 
an element in ${\cal W}$. Finally, since $\un{m}^R_S\tilde{\alpha}\Upsilon^{-1}_S=\un{m}_S(\Id_S\ot \a_0)$ it follows that 
$S$ is a separable algebra in ${}^!_R\Cc_R$ if and only if there exists $\a_0: \un{1}\ra S$ a morphism in $\Cc$ 
satisfying the two conditions in the statement. So we are done. 
\end{proof}

\section{Further characterizations for monoidal Frobenius algebra extensions}\selabel{FurtherCharFrobExt}
\setcounter{equation}{0}
In view of \prref{Frobalgextcharact} it is clear that characterizing Frobenius or separable algebra extensions 
is equivalent to characterizing Frobenius or separable algebras, of course if some coflatness and 
robustness properties are satisfied. We will do this in the next results. In the Frobenius case we do not have 
to assume from the beginning that the category is rigid (see the definition below). As we will see 
the existence of the dual object for a Frobenius algebra in a monoidal catgory $\Cc$ is automatic.  
The result can be viewed as the analogue of the classical result asserting that a Frobenius algebra 
in a category of vector spaces is always finite dimensional. 

Recall that an object $X$ of a monoidal category $\Cc$ admits a left dual if there exist 
an object $X^*$ in $\Cc$ and morphisms  
${\rm ev}_X: X^*\ot X\ra \un{1}$ and ${\rm coev}_X: \un{1}\ra X\ot X^*$ in $\Cc$ such that  
\begin{equation}\eqlabel{lrigid1}
(\Id_X\ot {\rm ev}_X)({\rm coev}_X\ot \Id_X)=\Id_X~~{\rm and}~~
({\rm ev}_X\ot \Id_{X^*})(\Id_{X^*}\ot {\rm coev}_X)=\Id_{X^*}.
\end{equation} 
In what follows we denote ${\rm ev}_X=
\footnotesize{\gbeg{2}{3}
\got{1}{X^*}\got{1}{X}\gnl
\gev\gnl
\gob{2}{\un{1}}
\gend}
$ and ${\rm coev}_X=
{\footnotesize
\gbeg{2}{3}
\got{2}{\un{1}}\gnl
\gdb\gnl
\gob{1}{X}\gob{1}{X^*}
\gend}
$. 
Hence, the following relations hold:
\begin{equation}\eqlabel{lrigid2}
{\footnotesize
\gbeg{3}{4}
\gvac{2}\got{1}{X}\gnl
\gdb\gcl{1}\gnl
\gcl{1}\gev\gnl
\gob{1}{X}
\gend} =
{\footnotesize 
\gbeg{1}{3}
\got{1}{X}\gnl
\gcl{1}\gnl
\gob{1}{X}
\gend}
\hspace{3mm}{\rm and}\hspace{3mm}
{\footnotesize 
\gbeg{3}{4}
\got{1}{X^*}\gnl
\gcl{1}\gdb\gnl
\gev\gcl{1}\gnl
\gvac{2}\gob{1}{X^*}
\gend} =
{\footnotesize 
\gbeg{1}{3}
\got{1}{X^*}\gnl
\gcl{1}\gnl
\gob{1}{X^*}
\gend}\hspace{2mm}.
\end{equation} 
If any object of $\Cc$ admits a left dual we then say that $\Cc$ is left rigid. 

Likewise, $\Cc$ is right rigid if for any $X\in \Cc$ there exist 
an object ${}^*X\in \Cc$ and morphisms  
${\rm ev}'_X:\ X\ot {}^*X\to \un{1}$ and ${\rm coev}'_X:\ \un{1}\to {}^*X\ot X$ 
such that 
\begin{equation}\eqlabel{rrigid1}
({\rm ev}'_X\ot \Id_X)(\Id_X\ot {\rm coev}'_X)=\Id_X~~{\rm and}~~
(\Id_{{}^*X}\ot {\rm ev}'_X)({\rm coev}'_X\ot \Id_{{}^*X})=\Id_{{}^*X}.
\end{equation}
In what follows we will denote 
${\rm ev}'_X:=
\footnotesize{
\gbeg{2}{3}
\got{1}{X}\got{1}{{}^*X}\gnl
\gvac{1}\gnot{\hspace*{-4mm}\vspace*{-2mm}\bullet}\gvac{-1}\gev\gnl
\gob{2}{\un{1}}
\gend}
$ 
and ${\rm coev}'_X:=
\footnotesize{\gbeg{2}{3}
\got{2}{\un{1}}\gnl
\gvac{1}\gnot{\hspace*{-4mm}\vspace*{2mm}\bullet}\gvac{-1}\gdb\gnl
\gob{1}{{}^*X}\gob{1}{X}
\gend \hspace{2mm}}
$. Then the relations above can be written as 
\begin{equation}\eqlabel{rrigid2}
{\footnotesize
\gbeg{3}{4}
\got{1}{X}\gnl
\gcl{1}\gvac{1}\gnot{\hspace*{-4mm}\vspace*{2mm}\bullet}\gvac{-1}\gdb\gnl
\gvac{1}\gnot{\hspace*{-4mm}\vspace*{-2mm}\bullet}\gvac{-1}\gev\gcl{1}\gnl
\gob{5}{X}
\gend} =
{\footnotesize 
\gbeg{1}{3}
\got{1}{X}\gnl
\gcl{1}\gnl
\gob{1}{X}
\gend} 
\hspace{3mm}\mbox{~~and~~}\hspace{3mm}
{\footnotesize 
\gbeg{3}{4}
\got{5}{{}^*X}\gnl
\gvac{1}\gnot{\hspace*{-4mm}\vspace*{2mm}\bullet}\gvac{-1}\gdb\gcl{1}\gnl
\gcl{1}\gvac{1}\gnot{\hspace*{-4mm}\vspace*{-2mm}\bullet}\gvac{-1}\gev\gnl
\gob{1}{{}^*X}
\gend }=
{\footnotesize 
\gbeg{1}{3}
\got{1}{{}^*X}\gnl
\gcl{1}\gnl
\gob{1}{{}^*X}
\gend} \hspace{2mm},
\end{equation} 
respectively. Thus a right dual for $X$ in $\Cc$ is nothing than a left dual for $X$ 
in $\ov{\Cc}$, the reverse monoidal category associated to $\Cc$ ($\ov{\Cc}$ is the category 
$\Cc$ endowed with the reverse monoidal structure of $\Cc$, that is, with the tensor 
product $\ov{\ot}=\ot\circ \tau$, where $\tau: \Cc\times \Cc\ra \Cc\times \Cc$ is the switch 
functor). In what follows by 
$(\rho, \l): Y\dashv X$ we denote the fact that $Y$ with $\rho: \un{1}\ra X\ot Y$ 
and $\l: Y\ot X\ra \un{1}$ is a left dual for $X$ or, equivalently, that $(X, \r, \l)$ is a 
right dual for $Y$. The pair $(\r, \l)$ is called an adjunction between $Y$ and $X$.  

Consider $A$ an algebra in a monoidal category $\Cc$ that has a left dual object $A^*$ (respectively a right 
dual object ${}^*A$). Then it is well known that $A^*$ (respectively ${}^*A$) is a right (left) $A$-module via the structure morphism 
\[{\footnotesize
\gbeg{3}{3}
\got{1}{A^*}\gvac{1}\got{1}{~~A}\gnl
\gwmu{3}\gnl
\gvac{1}\gob{1}{A^*}
\gend} =
{\footnotesize 
\gbeg{4}{6}
\got{1}{A^*}\got{1}{~~A}\gnl
\gcl{1}\gcl{1}\gdb\gnl
\gcl{1}\gmu\gcl{3}\gnl
\gcl{1}\gcn{1}{1}{2}{1}\gnl
\gev\gnl
\gvac{3}\gob{1}{A^*}
\gend}
\hspace{1cm}
\left(\mbox{respectively}\hspace{1cm}
{\footnotesize 
\gbeg{3}{3}
\got{1}{A}\gvac{1}\got{1}{{}^*A}\gnl
\gwmuh{2}{1}{5}\gnl
\gob{2}{{}^*A}
\gend} =
{\footnotesize  
\gbeg{4}{6}
\gvac{2}\got{1}{A}\got{1}{{}^*A}\gnl
\gvac{1}\gnot{\hspace*{-4mm}\vspace*{2mm}\bullet}\gvac{-1}\gdb\gcl{1}\gcl{1}\gnl
\gcl{1}\gmu\gcl{1}\gnl
\gcl{1}\gcn{1}{1}{2}{3}\gvac{1}\gcl{1}\gnl
\gcl{1}\gvac{1}\gvac{1}\gnot{\hspace*{-4mm}\vspace*{-2mm}\bullet}\gvac{-1}\gev\gnl
\gob{1}{{}^*A}
\gend}
\right)\hspace{2mm}.
\] 
Finally, if $X$ is an arbitrary object of $\Cc$ and $B: A\ot A\ra X$ is a morphism in $\Cc$ 
we then say that $B$ is associative if and only if $B(\un{m}_A\ot \Id_A)=B(\Id_A\ot \un{m}_A)$. 

One can now present characterizations for Frobenius algebras in arbitrary monoidal categories, 
so not necessarily rigid monoidal. Most of them are collected from \cite{fust, street, street2, yam}. 

\begin{theorem}\thlabel{charactFrobalg}
Let $\Cc$ be a monoidal category and $A$ an algebra in $\Cc$. Then the following assertions 
are equivalent:
\begin{itemize}
\item[(i)] $A$ is a Frobenius algebra;
\item[(ii)] $A$ admits a left dual $A^*$ and $A$ is isomorphic to $A^*$ as a right $A$-module;
\item[(iii)] $A$ admits a right dual ${}^*A$ and $A$ is isomorphic to ${}^*A$ as a left $A$-module;
\item[(iv)] $A$ admits a coalgebra structure $(A, \un{\Delta}_A, \un{\va}_A)$ in $\Cc$ 
such that $\un{\Delta}_A$ is an $A$-bimodule map, where both $A$ and $A\ot A$ are considered 
bimodules via the multiplication $\un{m}_A$ of $A$;
\item[(v)] $A$ admits a right dual ${}^*A$ and there exists a morphism $B: A\ot A\ra \un{1}$ 
that is associative and such that $\Phi^r_B:=(\Id_{{}^*A}\ot B)({\rm coev}'_A\ot \Id_A): A\ra {}^*A$ 
is an isomorphism in $\Cc$;
\item[(vi)] $A$ admits a left dual $A^*$ and there exists a morphism $B: A\ot A\ra \un{1}$ in $\Cc$ 
that is associative and such that $\Phi^l_B:=(B\ot \Id_{A^*})(\Id_A\ot {\rm coev}_A): A\ra A^*$ 
is an isomorphism;
\item[(vii)] There exists an adjunction $(\rho, \l): A\dashv A$ for which $\l: A\ot A\ra \un{1}$ 
is associative;
\item[(viii)] There exists an adjunction $(\rho, \l): A\dashv A$ such that $\l=\vartheta\un{m}_A$ 
for some $\vartheta: A\ra \un{1}$ morphism in $\Cc$.  
\end{itemize}
Furthermore, if $\un{1}$ is a left $\ot$-generator then the above 
assertions are also equivalent to 
\begin{itemize}
\item[(ix)] $-\ot S: \Cc\ra \Cc_S$ is a right adjoint for the forgetful functor $F: \Cc_S\ra \Cc$ 
or, in other words, $F$ is a Frobenius functor.  
\end{itemize}
\end{theorem}
\begin{proof}
We sketch the proof. For each implication below the complete proof can be found in 
the quoted references or can be done directly by the reader.

$(i)\Leftrightarrow (ii)$. This is pointed out in \cite[Proposition 2.1]{yam}. Consider $(\vartheta, e)$ a Frobenius 
system for $A$ and define ${\rm ev}_A=\vartheta\un{m}_A: A\ot A\ra \un{1}$ and ${\rm coev}_A=e: \un{1}\ra A\ot A$. 
Then one can see easily that $(A, {\rm ev}_A, {\rm coev}_A)$ is a left dual for $A$ and, moreover, 
that the right action of $A$ on this left dual coincides with the multiplication of $A$. Then 
$A$ has a left dual and is isomorphic to it as a right $A$-module.  

For the converse, let $\Psi: A\ra A^*$ be a right $A$-linear isomorphism in $\Cc$ and define 
\[\vartheta=
\footnotesize{
\gbeg{2}{4}
\got{1}{A}\gnl
\gmp{\Psi}\gu{1}\gnl
\gev\gnl
\gob{2}{\un{1}}\\
\gend}
~~{\rm and~~}
e={\footnotesize 
\gbeg{2}{4}
\got{2}{\un{1}}\gnl
\gdb\gnl
\gcl{1}\gmp{\psi}\gnl
\gob{1}{A}\gob{1}{A}
\gend}~,
\] 
where $\psi$ is the inverse of $\Psi$. Then $(\vartheta, e)$ is a Frobenius pair for $A$, 
and therefore $A$ is a Frobenius algebra.

$(ii)\Leftrightarrow (iii)$. Follows from \cite[Lemma 5]{fust}. 

$(ii)\Leftrightarrow (iv)$. See for instance \cite[Propositions 8\& 9]{fust}. A direct proof, based 
on a monoidal approach, is the following. 
 
Let $A^*$ be a left dual object for $A$ and $\Psi: A\ra A^*$ a right 
$A$-module isomorphism in $\Cc$. Since the left dual functor $()^*: \Cc\ra \ov{\Cc}^{\rm opp}$ is 
monoidal the left dual $A^*$ admits a coalgebra structure in $\ov{\Cc}$, and therefore in $\Cc$ as well. 
If we transport this coalgebra structure on $A$ through the isomorphism $\Psi$ we get that $A$ admits a 
coalgebra structure in $\Cc$. More precisely, with 
\[
\un{\Delta}_A={\footnotesize
\gbeg{5}{8}
\got{1}{A}\gnl
\gcl{1}\gvac{1}\gdb\gnl
\gmp{\Psi}\gvac{1}\gcn{1}{1}{1}{-1}\gcn{1}{1}{1}{3}\gnl
\gcl{1}\gcl{1}\gdb\gcl{1}\gnl
\gcl{1}\gmu\gcl{1}\gcl{1}\gnl
\gcl{1}\gcn{1}{1}{2}{1}\gvac{1}\gmp{\psi}\gmp{\psi}\gnl
\gev\gvac{1}\gcl{1}\gcl{1}\gnl
\gvac{3}\gob{1}{A}\gob{1}{A}
\gend}
~~\mbox{~~and~~}~~
\un{\va}_A={\footnotesize 
\gbeg{2}{4}
\got{1}{A}\gnl
\gcl{1}\gu{1}\gnl
\gev\gnl
\gob{2}{\un{1}}
\gend}
\]
$A$ becomes a coalgebra in $\Cc$ where, as before, $\psi$ stands for the inverse of $\Psi$. 
Using that $\Psi$ is right $A$-linear we get that $\un{\Delta}_A$ is an $A$-bimodule morphism in $\Cc$, as desired.

Conversely, if $A$ admits a coalgebra structure $(A, \un{\Delta}_A, \un{\va}_A)$ 
such that $\un{\Delta}_A$ is an $A$-bilinear morphism then $A$ itself together with 
${\rm ev}_A=\un{\va}_A\un{m}_A: A\ot A\ra \un{1}$ and 
${\rm coev}_A=\un{\Delta}_A\un{\eta}_A: \un{1}\ra A\ot A$ is a left dual for $A$. Furthermore, the right action of $A$ 
on this left dual of it is just $\un{m}_A$. Thus $A$ admits a left dual and is isomorphic to it as 
a right $A$-module. 

$(iii)\Leftrightarrow (v)$ and $(ii)\Leftrightarrow (vi)$ follow from \cite[Proposition 9]{fust}.  

The implication $(vii)\Rightarrow (vi)$ is trivial. To prove $(vi)\Rightarrow (vii)$ we proceed as 
in the proof of \cite[Theorem 1.6]{street} or \cite[Theorem]{street2}. 
Namely, if $A^*$ is a left dual object for $A$ and $\Psi: A\ra A^*$ is a right $A$-module 
isomorphism with inverse $\psi$ then it can be easily verified that 
$((\Id_A\ot \psi){\rm coev}_A, {\rm ev}_A(\Psi\ot \Id_A)): A\dashv A$ is an adjunction 
for which ${\rm ev}_A(\Psi\ot \Id_A): A\ot A\ra \un{1}$ is associative.

$(vii)\Rightarrow (viii)$. If $\l$ is associative then $\vartheta:=\l(\un{\eta}_A\ot \Id_A)=
\l(\Id_A\ot \un{\eta}_A)$ is the desired morphism since $\vartheta\un{m}_A=\l$. The converse 
is also true because $\vartheta\un{m}_A$ is clearly associative. 
 
The equivalence between (i) and (ix) follows from the comments made before \prref{SepAlgSepFun}, so our proof is finished.  
\end{proof}

Since Frobenius algebra extensions are particular cases of Frobenius algebras in 
monoidal categories we get the following list of characterizations for a Frobenius 
algebra extension.

\begin{corollary}\colabel{CharFrobAlgExt}
Let $\Cc$ be a monoidal category for which any object is coflat and left robust. Then for an algebra extension 
$i: R\ra S$ in $\Cc$ the following assertions are equivalent:
\begin{itemize}
\item[(i)] The extension $i: R\ra S$ is Frobenius;
\item[(ii)] $S$ admits a left dual object $S^\surd$ in ${}_R\Cc_R$ and $S$ and $S^\surd$ are isomorphic as right 
$S$-modules in ${}_R\Cc_R$;
\item[(iii)] $S$ admits a right dual object ${}^\surd S$ in ${}_R\Cc_R$ and $S$ and ${}^\surd S$ are isomorphic 
as left $S$-modules in ${}_R\Cc_R$;
\item[(iv)] $S$ admits a coalgebra structure in ${}_R\Cc_R$, that is an $R$-coring structure, such that the 
comultiplication morphism is $S$-bilinear in ${}_R\Cc_R$.
\item[(v)] $S$ admits a right dual ${}^\surd S$ in ${}_R\Cc_R$ and there exists a morphism 
$B: S\ot_R S\ra R$ in ${}_R\Cc_R$ that is associative and such 
$\Phi^r_B:=\Upsilon_{{}^\surd S}\widetilde{B}\Sigma'_{{}^\surd S, S, S}\widehat{\rm{coev}'_S}\Upsilon^{'-1}_S: 
S\ra {}^\surd S$ is an isomorphism in ${}_R\Cc_R$;  
\item[(vi)] $S$ admits a left dual $S^\surd$ in ${}_R\Cc_R$ and there exists a morphism 
$B: S\ot_R S\ra R$ in ${}_R\Cc_R$ that is associative and such that 
$\Phi^l:=\Upsilon'_{S^\surd }\widehat{B}\Gamma'_{S, S, S^\surd}\widetilde{\rm{coev}_S}\Upsilon^{-1}_S: 
S\ra S^\surd$ is an isomorphism in ${}_R\Cc_R$;
\item[(vii)] There exists an adjunction $(\r, \l): S\dashv S$ in ${}_R\Cc_R$ 
for which $\l: S\otimes_R S\ra R$ is associative; 
\item[(viii)] There exists an adjunction $(\r, \l): S\dashv S$ in ${}_R\Cc_R$ such that 
$\l=\vartheta \un{m}_S^R$ for some $\vartheta: S\ra R$ morphism in ${}_R\Cc_R$.   
\end{itemize}
Furthermore, if $R$ is a left $\ot_R$-generator for ${}_R\Cc_R$ then the above assertions are also 
equivalent to 
\begin{itemize}
\item[(ix)] $-\ot_R S: {}_R\Cc_R\ra {}_R\Cc_S$ is a right adjoint of the forgetful functor 
$F: {}_R\Cc_S\ra {}_R\Cc_R$ or, otherwise stated, $F$ is a Frobenius functor. 
\end{itemize}
\end{corollary} 

So if an algebra extension $i: R\ra S$ in $\Cc$ is Frobenius then $S$ has left and 
right dual objects in ${}_R\Cc_R$. In the sequel we go further with this observation, by 
investigating when the existence of the dual object of $S$ in ${}_R\Cc_R$ implies the existence of the 
dual object in $\Cc$, and vice-versa. The final aim is to characterize the Frobenius 
property of an algebra extensions $i: R\ra S$ in terms given by the algebras $R$ and $S$. We shall 
see that this is possible in the case when $R$ is Frobenius and separable. 

We start by presenting connections between the existence of the dual of an object $X$ in ${}_R\Cc_R$ and the existence 
of the dual of the same object $X$, considered now as an object in $\Cc$ via 
the canonical forgetful functor $\mathfrak{U}: {}_R\Cc_R\ra \Cc$. In this direction, 
it is well-known that a strong monoidal 
functor preserves duals. Even more, it has been proved in \cite[Theorem 2]{DayPastro} that this remains 
true in the weaker hypothesis when in place of a strong monoidal functor we consider a Frobenius monoidal one. 
So a partial answer to our problem is offered by the case when $\mathfrak{U}$ is a Frobenius monoidal functor. 
This is why we start by giving a necessary and sufficient condition for which $\mathfrak{U}$ is Frobenius monoidal. 

Recall first from \cite[Definition 6.1]{sz} and \cite[Definition 1]{DayPastro} the concept of 
(separable) Frobenius monoidal functor. 

\begin{definition}
Let $(\Cc, \ot, \un{1})$ and $(\Dc, \square, \un{I})$ be (strict) monoidal categories and 
$\mathfrak{F}: \Cc\ra \Dc$ a functor. 
\begin{itemize}
\item[(i)] $\mathfrak{F}$ is called monoidal if there exist a family of morphisms 
$\phi_2=(\phi_{X, Y}: \mathfrak{F}(X)\square \mathfrak{F}(Y)\ra \mathfrak{F}(X\ot Y))_{X, Y\in \Cc}$, natural in $X$ and $Y$, 
and $\phi_0: \un{I}\ra \mathfrak{F}(\un{1})$ morphism in $\Dc$ such that, for all $X, Y, Z\in \Cc$, 
the corresponding diagrams in \equref{FrobMonFunctor1} 
are commutative.
\item[(ii)] $\mathfrak{F}$ is called opmonoidal if there exist a family of morphisms 
$\psi_2=(\psi_{X, Y}: \mathfrak{F}(X\otimes Y)\ra \mathfrak{F}(X)\square \mathfrak{F}(Y))_{X, Y\in \Cc}$, 
natural in $X$ and $Y$, and $\psi_0: \mathfrak{F}(\un{1})\ra \un{I}$ morphism in $\Dc$ such that, for all 
$X, Y, Z\in \Cc$, the corresponding diagrams in \equref{FrobMonFunctor1} are commutative.   
\begin{eqnarray}
&&\hspace*{1cm}\xymatrix{
\mathfrak{F}(X)\square \mathfrak{F}(Y)\square \mathfrak{F}(Z)
\ar@<1ex>[rr]^-{\phi_{X, Y}\square \Id_{\mathfrak{F}(Y)}}\ar@<-1ex>[d]_-{\Id_{\mathfrak{F}(X)}\square \phi_{Y, Z}}&& 
\ar@<1ex>[ll]^-{\psi_{X, Y}\square \Id_{\mathfrak{F}(Z)}}\mathfrak{F}(X\ot Y)\square \mathfrak{F}(Z)
\ar@<1ex>[d]^-{\phi_{X\ot Y, Z}}\\
\mathfrak{F}(X)\square \mathfrak{F}(Y\otimes Z)
\ar@<-1ex>[rr]_-{\phi_{X, Y\ot Z}}\ar@<-1ex>[u]_-{\Id_{\mathfrak{F}(X)}\square \psi_{Y, Z}}&&
\ar@<-1ex>[ll]_-{\psi_{X, Y\ot Z}}\mathfrak{F}(X\ot Y\ot Z)\ar@<1ex>[u]^-{\psi_{X\ot Y, Z}}
}~,~\nonumber\\
&&\eqlabel{FrobMonFunctor1}\\
&&\xymatrix{
\un{I}\square \mathfrak{F}(X)\ar@<1ex>[rr]^-{\phi_0\square \Id_{\mathfrak{F}(X)}}\ar[d]_-{l_{\mathfrak{F}(X)}}&&
\ar@<1ex>[ll]^-{\psi_0\square \Id_{\mathfrak{F}(X)}}\mathfrak{F}(\un{1})\square \mathfrak{F}(X)
\ar@<1ex>[d]^-{\phi_{\un{1}, X}}\\
\mathfrak{F}(X)&&\ar[ll]^-{\mathfrak{F}(l_X)}\mathfrak{F}(\un{1}\otimes X)\ar@<1ex>[u]^-{\psi_{\un{1}, X}}
}~,~
\xymatrix{
\mathfrak{F}(X)\square \un{I}\ar@<1ex>[rr]^-{\Id_{\mathfrak{F}(X)}\square \phi_0}\ar[d]_-{r_{\mathfrak{F}(X)}}&&
\ar@<1ex>[ll]^-{\Id_{\mathfrak{F}(X)}\square \psi_0}\mathfrak{F}(X)\square \mathfrak{F}(\un{1})
\ar@<1ex>[d]^-{\phi_{X, \un{1}}}\\
\mathfrak{F}(X)&&\ar[ll]^-{\mathfrak{F}(r_X)}\mathfrak{F}(X\otimes \un{1})\ar@<1ex>[u]^-{\psi_{X, \un{1}}}
}~.\nonumber
\end{eqnarray}
\item[(iii)] We call $\mathfrak{F}$ Frobenius monoidal if $\mathfrak{F}$ is equipped with a monoidal 
$(\phi_2, \phi_0)$ and comonoidal $(\psi_2, \psi_0)$ structure such that, for all $X, Y, Z\in \Cc$, the diagrams 
\begin{equation}\eqlabel{FrobMonFunctor2}
\xymatrix{
\mathfrak{F}(X\ot Y)\square \mathfrak{F}(Z)\hspace{3mm}\ar[r]^-{\psi_{X, Y}\square \Id_{\mathfrak{F}(Z)}}
\ar[d]_-{\phi_{X\ot Y, Z}}&\hspace{3mm}
\mathfrak{F}(X)\square\mathfrak{F}(Y)\square\mathfrak{F}(Z)\ar[d]^-{\Id_{\mathfrak{F}(X)}\square \phi_{Y, Z}}\\
\mathfrak{F}(X\ot Y\ot Z)\ar[r]^-{\psi_{X, Y\ot Z}}&\mathfrak{F}(X)\square \mathfrak{F}(Y\ot Z)
}~,~
\xymatrix{
\mathfrak{F}(X)\square \mathfrak{F}(Y\otimes Z)\hspace{3mm}\ar[r]^-{\Id_{\mathfrak{F}(X)}\square \psi_{Y, Z}}
\ar[d]_-{\phi_{X, Y\otimes Z}}&\hspace{3mm}\mathfrak{F}(X)\square \mathfrak{F}(Y)\square \mathfrak{F}(Z)
\ar[d]^-{\phi_{X, Y}\square \Id_{\mathfrak{F}(Z)}}\\
\mathfrak{F}(X\ot Y\ot Z)\ar[r]^-{\psi_{X\otimes Y, Z}}&\mathfrak{F}(X\ot Y)\square \mathfrak{F}(Z)
}
\end{equation}
are commutative. If, furthermore, $\phi_{X, Y}\psi_{X, Y}=\Id_{\mathfrak{F}(X\ot Y)}$, for all 
$X, Y\in \Cc$, we then say that $\mathfrak{F}$ is a separable Frobenius monoidal functor. 
\item[(iv)] $\mathfrak{F}$ is called strong monoidal if it is a separable Frobenius 
monoidal functor, $\phi_{X, Y}$ is an isomorphism for all $X, Y\in \Cc$ (and thus 
$\phi_{X, Y}^{-1}=\psi_{X, Y}$), and $\phi_0$ and $\psi_0$ are inverses of each other.   
\end{itemize} 
\end{definition}  

We leave to the reader to check that $\phi_2=(\phi_{X, Y}=q^R_{X, Y}: X\ot Y\ra X\ot_R Y)_{X, Y\in {}_R\Cc_R}$ and 
$\phi_0=\un{\eta}_R: \un{1}\ra R$ define on the forgetful functor $\mathfrak{U}: {}_R\Cc_R\ra \Cc$ a monoidal structure. 
We refer to it as being the trivial monoidal structure of the functor $\mathfrak{U}$.   

The next result gives the connection between the notions of Frobenius monoidal functor 
and Frobenius monoidal algebra. It improves and at the same time generalizes \cite[Lemmas 6.3 \& 6.4]{sz}.  

\begin{theorem}\thlabel{ForgFrobMonFunctor}
Let $\Cc$ be a monoidal category with coequalizers and such that any object of it is 
coflat and left robust. If $R$ is an algebra in $\Cc$ then the forgetful functor 
$\mathfrak{U}: {}_R\Cc_R\ra \Cc$ endowed with the trivial monoidal structure 
$(q^R_{-, -}, \un{\eta}_R)$ is Frobenius if and only if $R$ is a Frobenius algebra. 
Furthermore, if this is the case then the opmonoidal $(\psi_2, \psi_0)$ structure of $\mathfrak{U}$ is 
uniquely determined by a Frobenius structure of $R$, in the sense that 
there exists $(\vartheta, e)$ a Frobenius pair for $R$ such that  
$\psi_0=\vartheta$ and 
\begin{equation}\eqlabel{ComStrForgFunctor}
\psi_2q^R_{-,-}:=(\psi_{X, Y}q^R_{X, Y})_{X, Y\in {}_R\Cc_R}=\left((\nu^R_X\ot \mu^R_Y)(\Id_X\ot e\ot \Id_Y)\right)_{X, Y\in {}_R\Cc_R}.
\end{equation}   
\end{theorem}
\begin{proof}
Assume that the functor $\mathfrak{U}$ admits a Frobenius monoidal structure. Then by 
\cite[Corollary 5]{DayPastro} $\mathfrak{U}$ carries Frobenius algebras to Frobenius algebras. Since $R$ 
is a Frobenius algebra in ${}_R\Cc_R$ (as it is the unit object of a monoidal category) 
we get that $R$ is a Frobenius algebra in $\Cc$, as needed. Note that an alternative proof for 
this implication can be obtained from \cite[Lemma 6.3]{sz}. It has the advantage that we can obtain 
the coalgebra structure $(\Delta, \va)$ of $R$ in $\Cc$ as follows. If $(\psi_2, \psi)$ is the comonoidal structure 
of $\mathfrak{U}$ that together with the trivial monoidal structure gives on $\mathfrak{U}$ a 
Frobenius monoidal functor structure then $\Delta=\psi_{R, R}\Upsilon^{-1}_R: R\ra R\ot R$ and 
$\va=\psi_0: R\ra \un{1}$, respectively. 

Conversely, suppose that $R$ is a Frobenius algebra in $\Cc$ and let $(\vartheta, e)$ be a Frobenius system for $R$.  
For $X, Y\in {}_R\Cc_R$ denote by $\psi_{X, Y}$ the morphism in $\Cc$ uniquely determined by $\psi_{X, Y}q^R_{X, Y}=
(\nu^R_X\ot \mu^R_Y)(\Id_X\ot e\ot \Id_Y)$. Note that $(\nu^R_X\ot \mu^R_Y)(\Id_X\ot e\ot \Id_Y): X\ot Y\ra X\ot Y$ 
fits in the universal property of the cooequalizer 
\[
\xymatrix{
X\ot R\ot Y\ar@<1ex>[rr]^-{\nu^R_X\ot \Id_Y}\ar@<-1ex>[rr]_-{\Id_X\ot \mu^R_Y}&&X\ot Y\ar[r]^-{q^R_{X, Y}}&X\ot_RY
}
\]
since $(\un{m}_R\ot \Id_R)(\Id_R\ot e)=(\Id_R\ot \un{m}_R)(e\ot \Id_R)$. Now a simple computation shows that the diagrams 
in \equref{FrobMonFunctor1} and \equref{FrobMonFunctor2} are commutative, and so $\mathfrak{U}$ is a Frobenius monoidal functor. 

It remains to prove that if $(q^R_{-,-}, \un{\eta}_R, \psi_2, \psi_0)$ defines a Frobenius monoidal functor structure 
on $\mathfrak{U}$ then there exists $(\vartheta, e)$ a Frobenius system for $R$ such that the comonoidal 
structure $(\psi_2, \psi_0)$ of $\mathfrak{U}$ is completely determined by it, in the sense that $\psi_0=\vartheta$ 
and $\psi_2$ is defined by \equref{ComStrForgFunctor}. To this end we first show that $\psi_{X, Y}$ is $R$-bilinear 
for any $X, Y\in {}_R\Cc_R$. 

Indeed, if we take $Z=R$ in the first diagram of \equref{FrobMonFunctor2} we then have 
\[
(\Id_X\ot q^R_{Y, R})(\psi_{X, Y}\ot \Id_R)=\psi_{X, Y\ot_RR}\Sigma'_{X, Y, R}q^R_{X\ot_RY, R}.
\]  
Composing the both sides of the above equality to the right with $\Id_X\ot \Upsilon_Y$, and taking 
in consideration that the diagram 
\[
\xymatrix{
X\ot_R(Y\ot_RR)\ar[rr]^-{\psi_{X, Y\ot_RR}}\ar[d]_-{\widetilde{\Upsilon_Y}}&&
X\ot (Y\ot_RR)\ar[d]^-{\Id_X\ot \Upsilon_Y}\\
X\ot_RY\ar[rr]^-{\psi_{X, Y}}&&X\ot Y
}
\]
is commutative ($\psi_{X, -}$ is natural and $\Upsilon_Y: Y\ot_RR\ra Y$ is a morphism 
in ${}_R\Cc_R$) we deduce that 
\begin{eqnarray*}
(\Id_X\ot \nu^R_Y)(\psi_{X, Y}q^R_{X, Y}\ot \Id_R)&=&
\psi_{X, Y}\widetilde{\Upsilon_Y}q^R_{X, Y\ot_RR}(\Id_X\ot q^R_{Y, R})\\
&=&\psi_{X, Y}q^R_{X, Y}(\Id_X\ot \nu^R_Y)\\
&=&\psi_{X, Y}\nu^R_{X\ot_RY}(q^R_{X, Y}\ot \Id_R).
\end{eqnarray*}
Hence $(\Id_X\ot \nu^R_Y)(\psi_{X, Y}\ot \Id_R)=\psi_{X, Y}\nu^R_{X\ot_RY}$, and this shows that 
$\psi_{X, Y}$ is right $R$-linear in $\Cc$. 

Likewise, if we take $X=R$ in the second diagram in \equref{FrobMonFunctor2} and use the fact 
that $(\Upsilon'_Y\ot \Id_Z)\psi_{R\ot_RY, Z}=\psi_{Y, Z}\widehat{\Upsilon'_Y}$ ($\psi_{-, Z}$ is natural 
and $\Upsilon'_Y: R\ot_RY\ra Y$ is a morphism in ${}_R\Cc_R$), by a computation similar to the one above 
we get that $\psi_{Y, Z}$ is left $R$-linear, for all $Y, Z\in {}_R\Cc_R$. Therefore $\psi_2$ is defined 
by a family of $R$-bilinear morphisms in $\Cc$, as desired. 

Now, if we define $\Delta:=\psi_{R, R}\Upsilon^{-1}_R=\psi_{R, R}\Upsilon^{'-1}_R: R\ra R\ot R$ 
it then follows that $\Delta$ is $R$-bilinear. So $\Delta$ is uniquely determined by 
$e:=\Delta\un{\eta}_R: \un{1}\ra R\ot R$, a morphism in $\Cc$ which has the property that 
$(\un{m}_R\ot \Id_R)(\Id_R\ot e)=(\Id_R\ot \un{m}_R)(e\ot \Id_R)$. 

We claim that $e$ determines 
completely $\psi_2$. Indeed, if we take $Y=R$ in the first diagram of \equref{FrobMonFunctor2} 
and use the naturalness of $\psi_{X, -}$ for the morphism 
$\Upsilon'_Z:R\ot_RZ\ra Z$ in ${}_R\Cc_R$ we obtain 
\begin{eqnarray*}
(\Id_X\ot \mu^R_Z)(\psi_{X, R}q^R_{X, R}\ot \Id_Z)&=&(\Id_X\ot \Upsilon'_Z)(\Id_X\ot q^R_{R, Z})(\psi_{X, R}q^R_{X, R}\ot \Id_Z)\\
&=&\psi_{X, Z}\widetilde{\Upsilon'_Z}q^R_{X, R\ot_RZ}(\Id_X\ot q^R_{R, Z})\\
&=&\psi_{X, Z}q^R_{X, Z}(\Id_X\ot \mu^R_Z).   
\end{eqnarray*}
Composing the both sides of the latest equality to the right with $\Id_X\ot \un{\eta}_R\ot \Id_Z$ we 
deduce that $\psi_{X, Z}q^R_{X, Z}=(\Id_X\ot \mu^R_Z)(\psi_{X, R}\Upsilon^{-1}_X\ot \Id_Z)$, for all $X, Z\in {}_R\Cc_R$. 

Similarly, take $Y=R$ in the second diagram of \equref{FrobMonFunctor2} and use 
$(\Upsilon_X\ot \Id_Z)\psi_{X\ot_RR, Z}=\psi_{X, Z}\widehat{\Upsilon_X}$ to get 
$\psi_{X, Z}q^R_{X, Z}=(\nu^R_X\ot \Id_Z)(\Id_X\ot \psi_{R, Z}\Upsilon^{'-1}_Z)$, 
for all $X, Z\in {}_R\Cc_R$. Combining these two equalities we get 
\begin{eqnarray*}
\psi_{X, Z}q^R_{X, Z}&=&(\nu^R_X\ot \Id_Z)(\Id_X\ot \psi_{R, Z}q^R_{R, Z}(\un{\eta}_R\ot \Id_Z))\\
&=&(\nu^R_X\ot \Id_Z)(\Id_X\ot (\Id_R\ot \mu^R_Z)(\psi_{R, R}\Upsilon^{-1}_R\un{\eta}_R\ot \Id_Z))\\
&=&(\nu^R_X\ot \Id_Z)(\Id_X\ot (\Id_R\ot \mu^R_Z)(e\ot \Id_Z))\\
&=&(\nu^R_X\ot \mu^R_Z)(\Id_X\ot e\ot \Id_Z),
\end{eqnarray*} 
for all $X, Z\in {}_R\Cc_R$, as claimed. Furthermore, it is immediate now that the commutativity of 
the two square diagrams in \equref{FrobMonFunctor1} is equivalent to 
$(\psi_0\ot \Id_R)e=\un{\eta}_R$ and $(\Id_R\ot \psi_0)e=\un{\eta}_R$, respectively. In conclusion, 
$(\vartheta=\psi_0, e=\psi_{R, R}\Upsilon^{-1}_R\eta_R)$ is a Frobenius system for $R$ and determines 
completely the opmonoidal structure $(\psi_2, \psi_0)$ of $\mathfrak{U}$. This finishes the proof.  
\end{proof}

\begin{corollary}\colabel{FrobExtVsFrobAlg}
Let $\Cc$ be a category with coequalizers and such that any object of it is coflat and left robust. 
If $i: R\ra S$ is a Frobenius algebra extension in $\Cc$ and $R$ is a Frobenius algebra in $\Cc$ 
then $S$ is a Frobenius algebra in $\Cc$, too.
\end{corollary}
\begin{proof}
Since $R$ is a Frobenius algebra in $\Cc$ we have that the forgetful functor $\mathfrak{U}: {}_R\Cc_R\ra \Cc$ 
is Frobenius monoidal, so it carries Frobenius algebras in ${}_R\Cc_R$ to Frobenius algebras in $\Cc$. 
The fact that $i: R\ra S$ is a Frobenius algebra extension is equivalent to the fact that $S$ is a 
Frobenius algebra in ${}_R\Cc_R$, cf. \prref{Frobalgextcharact}. Thus $S=\mathfrak{U}(S)$ is a Frobenius algebra in $\Cc$. 
\end{proof}

\begin{corollary}\colabel{DualRbimAredual}
Let $R$ be a Frobenius algebra in a monoidal category $\Cc$ with coequalizers and for which 
any object is coflat and left robust. If $X\in {}_R\Cc_R$ admits a (left) right object in ${}_R\Cc_R$ 
then it admits a (left) right dual object in $\Cc$, too.
\end{corollary}
\begin{proof}
Follows from the fact that a Frobenius monoidal functor behaves well with respect to dual objects, 
see \cite[Theorem 2]{DayPastro}. Actually, if $(\rho, \lambda): Y\dashv X$ is an adjunction in 
${}_R\Cc_R$ then 
\[
\rho':=\xymatrix{
(\un{1}\ar[r]^-{\un{\eta}_R}&R\ar[r]^-{\rho}&X\ot_R Y\ar[r]^-{\psi_{X, Y}}&X\ot Y)
}~\mbox{and}~ 
\lambda'=\xymatrix{
(X\ot Y\ar[r]^-{q^R_{X, Y}}&X\ot_RY\ar[r]^-{\lambda}&R\ar[r]^-{\psi_0}&\un{1})
}
\]
defines an adjunction $(\rho', \l'): Y\dashv X$ in $\Cc$, where $(\psi_2, \psi_0)$ is uniquely determined 
by a Frobenius system $(\vartheta, e)$ of $R$ as in \thref{ForgFrobMonFunctor}.  
\end{proof} 

\begin{corollary}
In the hypothesis and notions of \thref{ForgFrobMonFunctor} we have that 
$\mathfrak{U}: {}_R\Cc_R\ra \Cc$ is a separable 
Frobenius monoidal functor if and only if $R$ is a Frobenius separable algebra in $\Cc$.  
\end{corollary}
\begin{proof}
It is obvious. Note only that the condition $\phi_{X, Y}\psi_{X, Y}=\Id_{X\ot_RY}$, for all 
$X, Y\in {}_R\Cc_R$, is equivalent to $\un{m}_Re=\un{\eta}_R$.
\end{proof}

We focus now on the opposite direction. As we will see we need this 
time $\mathfrak{U}$ to be a separable Frobenius monoidal functor. 
 
\begin{proposition}\prlabel{DualAreDualBimod}
Consider $R$ a Frobenius separable algebra in a monoidal category $\Cc$ with coequalizers and with 
the property that any object of it is coflat and left robust. Let $(\vartheta, e)$ be a Frobenius pair for $R$, 
$\alpha: \un{1}\ra R$ as in \prref{WhenFrobisSep} and $(\rho, \l): Y\dashv X$ an adjunction in $\Cc$.
\begin{itemize}
\item[(i)] If $Y$ has an $R$-bimodule structure in $\Cc$ then $X$ has an $R$-bimodule 
structure in $\Cc$, too, and via these structures the morphisms     
\[
\rho_0=(\xymatrix{
R\ar[r]^-{\rho\ot \Id_R}&X\ot Y\ot R\ar[rr]^-{\Id_X\ot \alpha \ot \Id_{Y\ot R}}&&X\ot R\ot Y\ot R
\ar[rr]^-{\Id_X\ot \mu^R_Y\nu^R_Y}&&X\ot Y\ar[r]^-{q^R_{X, Y}}&X\ot_RY
})
\]
and $\l_0: Y\ot_R X\ra R$ uniquely determined by 
\[
\l_0q^R_{Y, X}=(\xymatrix{
Y\ot X\ar[rr]^-{e\ot \Id_{Y\ot X}}&&R\ot R\ot Y\ot X\ar[rr]^-{\Id_R\ot \mu^R_Y\ot \Id_X}&&R\ot Y\ot X\ar[r]^-{\Id_R\ot \lambda}&R
}) 
\]
define an adjunction $(\rho_0, \l_0): Y\dashv X$ in ${}_R\Cc_R$.
\item[(ii)] Similarly, if $X$ admits an $R$-bimodule structure in $\Cc$ then $Y$ admits 
an $R$-bimodule structure in $\Cc$ as well, and via these structures the morphisms 
\[
\rho^0=(\xymatrix{
R\ar[r]^-{\Id_R\ot \rho}&R\ot X\ot Y\ar[rr]^-{\Id_{R\ot X}\ot\alpha\ot\Id_Y}&&R\ot X\ot R\ot Y
\ar[rr]^-{\mu^R_X\nu^R_X\ot \Id_Y}&&X\ot Y
\ar[r]^-{q^R_{X, Y}}&X\ot_RY
})
\]
and $\l^0: Y\ot_RX\ra R$ uniquely determined by 
\[
\l^0q^R_{Y, X}=(\xymatrix{
Y\ot X\ar[rr]^-{\Id_{Y\ot X}\ot e}&&Y\ot X\ot R\ot R\ar[rr]^-{\Id_Y\ot \nu^R_X\ot \Id_R}&&
Y\ot X\ot R\ar[r]^-{\lambda \ot \Id_R}&R
})
\]
define an adjunction $(\rho^0, \l^0): Y\dashv X$ in ${}_R\Cc_R$. 
\end{itemize} 
\end{proposition}
\begin{proof}
We prove only (i). The proof of (ii) is similar, and can be obtained from that of (i) 
by reversing the structures, and so the diagrammatic computations as well, through a mirror. 
In other words the statement (ii) is the statement (i) for the monoidal category $\ov{\Cc}$, 
where $\ov{\Cc}$ is the reverse monoidal category associated to $\Cc$. 

(i) From now on, for $Z$ a left $R$-module in $\Cc$ we denote by 
$\footnotesize{
\gbeg{2}{3}
\got{1}{R}\got{1}{Z}\gnl
\glm\gnl
\gvac{1}\gob{1}{Z}
\gend} 
$ 
the structure morphism $\mu^R_Z$. Similarly, if $Z$ is a right $R$-module in $\Cc$ then 
the structure morphism $\nu^R_Z$ will be denoted by 
$\footnotesize{
\gbeg{2}{3}
\got{1}{Z}\got{1}{R}\gnl
\grm\gnl
\gob{1}{Z}
\gend} 
$.  
Also, for the Frobenius separable system $(e, \vartheta, \alpha)$ of $R$, and respectively for the 
morphisms of the adjunction $(\r, \l): Y\dashv X$, we will use the following 
diagrammatic notations,
\[
e={\footnotesize
\gbeg{2}{5}
\got{1}{\hspace{5mm}\un{1}}\gnl
\gnl
\gsbox{2}\gnot{\hspace{5mm}e}\gnl
\gcl{1}\gcl{1}\gnl
\gob{1}{R}\gob{1}{R}
\gend}
~,~
\vartheta=
{\footnotesize
\gbeg{1}{3}
\got{1}{R}\gnl
\gmpcu{\vartheta}\gnl
\gob{1}{\un{1}}
\gend}
~,~
\alpha=
{\footnotesize 
\gbeg{1}{3}
\got{1}{\un{1}}\gnl
\gmpu{\alpha}\gnl
\gob{1}{R}
\gend}
~\mbox{, and}~
\rho={\footnotesize
\gbeg{2}{3}
\got{1}{\hspace{5mm}\un{1}}\gnl
\gdb\gnl
\gob{1}{X}\gob{1}{Y}
\gend} 
~,~
\lambda=
{\footnotesize 
\gbeg{2}{3}
\got{1}{Y}\got{1}{X}\gnl
\gev\gnl
\gob{1}{\hspace{5mm}\un{1}}
\gend}
~\mbox{, respectively.}
\]

With their help define the morphisms $\d: R\ra X\ot Y$ and $\gamma: Y\ot X\ra R$ given by 
\[
{\footnotesize
\d=\gbeg{4}{6}
\gvac{3}\got{1}{R}\gnl
\gdb\gvac{1}\gcl{1}\gnl
\gcl{1}\gcn{1}{1}{1}{3}\gvac{1}\gcl{1}\gnl
\gcl{1}\gmpu{\alpha}\grm\gnl
\gcl{1}\glm\gnl
\gob{1}{X}\gvac{1}\gob{1}{Y}
\gend
}
~~\mbox{and}~~
\gamma=
{\footnotesize
\gbeg{4}{6}
\gvac{2}\got{1}{Y}\got{1}{X}\gnl
\gvac{2}\gcl{1}\gcl{1}\gnl
\gsbox{2}\gnot{\hspace{5mm}e}\gvac{2}\gcl{1}\gcl{1}\gnl
\gcl{1}\glm\gcl{1}\gnl
\gcl{1}\gvac{1}\gev\gnl
\gob{1}{R}
\gend}~.
\] 
Furthermore, if $Y$ is an $R$-bimodule in $\Cc$ we then define on $X$ the following left and right $R$-actions,
\[
\mu^R_X={\footnotesize
\gbeg{4}{5}
\gvac{2}\got{1}{R}\got{1}{X}\gnl
\gdb\gcl{1}\gcl{1}\gnl
\gcl{1}\grm\gcn{1}{1}{1}{-1}\gnl
\gcl{1}\gev\gnl
\gob{1}{X}
\gend}
~\mbox{and}~
\nu^R_X=
{\footnotesize
\gbeg{6}{10}
\gvac{4}\got{1}{X}\got{1}{R}\gnl
\gvac{1}\gdb\gvac{1}\gcl{1}\gcl{1}\gnl
\gvac{1}\gcn{1}{1}{1}{-1}\gcn{1}{1}{1}{3}\gvac{1}\gcl{3}\gcl{3}\gnl
\gcl{1}\gsbox{2}\gnot{\hspace{5mm}e}\gvac{2}\gcl{1}\gnl
\gcl{1}\gcl{1}\glm\gcl{1}\gcl{1}\gnl
\gcl{1}\gcn{1}{1}{1}{5}\gvac{1}\gev\gcn{1}{1}{1}{-1}\gnl
\gcl{1}\gvac{2}\gmu\gnl
\gcl{1}\gvac{2}\gcn{1}{1}{2}{1}\gnl
\gcl{1}\gvac{2}\gmpcu{\vartheta}\gnl
\gob{1}{X}
\gend}~,
\]
respectively. In fact, it can be easily checked that $\mu^R_X$ yields a left $R$-module structure on $X$ and since 
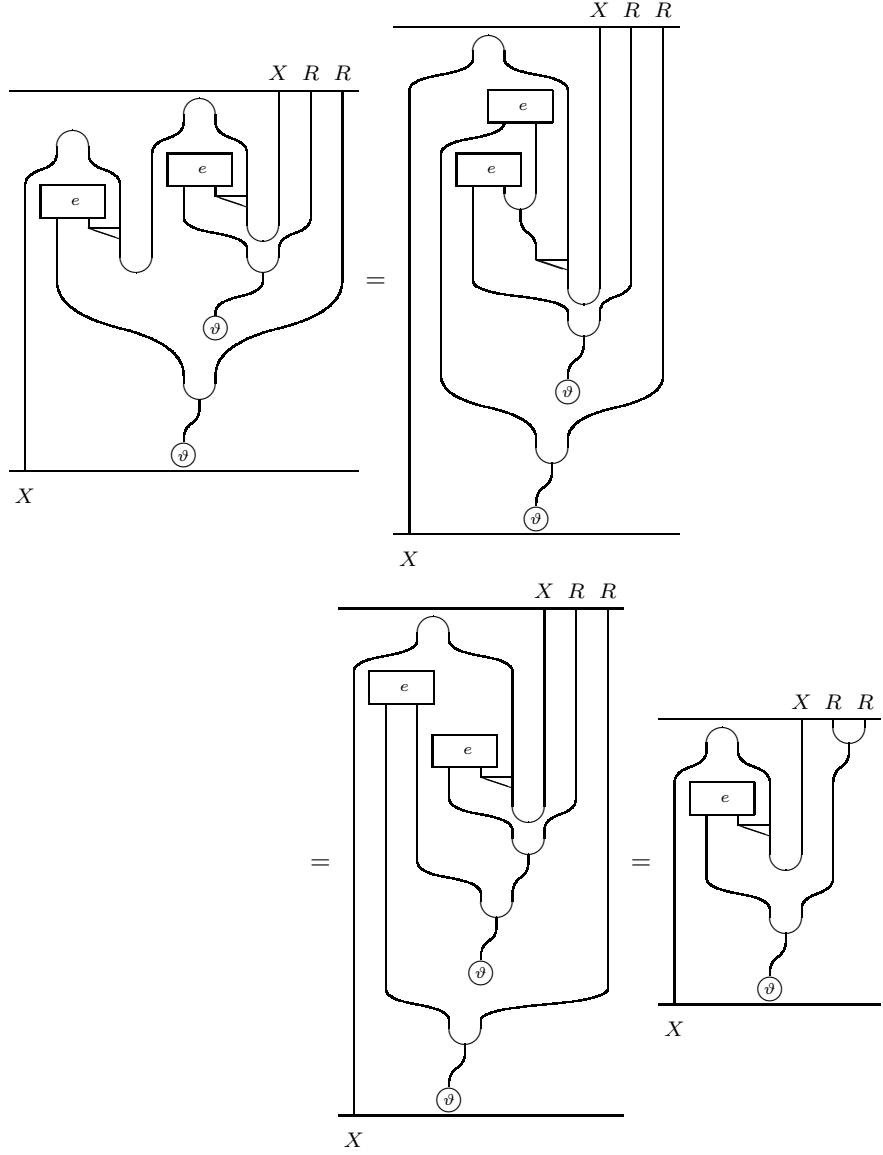
\begin{eqnarray*}
&&\hspace*{-2cm}
{\footnotesize
\gbeg{11}{14}
\gvac{8}\got{1}{X}\got{1}{R}\got{1}{R}\gnl
\gvac{5}\gdb\gvac{1}\gcl{4}\gcl{4}\gcl{6}\gnl
\gvac{1}\gdb\gvac{2}\gcn{1}{1}{1}{-1}\gcn{1}{1}{1}{3}\gnl
\gvac{1}\gcn{1}{1}{1}{-1}\gcn{1}{1}{1}{3}\gvac{1}\gcl{1}\gsbox{2}\gnot{\hspace{5mm}e}\gvac{2}\gcl{1}\gnl
\gcl{1}\gsbox{2}\gnot{\hspace{5mm}e}\gvac{2}\gcl{1}\gcl{1}\gcl{1}\glm\gnl
\gcl{5}\gcl{1}\glm\gcl{1}\gcn{1}{1}{1}{5}\gvac{1}\gev\gcn{1}{1}{1}{-1}\gnl
\gvac{1}\gcl{1}\gvac{1}\gev\gvac{2}\gmu\gnl
\gvac{3}\gcn{1}{3}{-3}{5}\gvac{4}\gcn{1}{1}{0}{-3}\gvac{1}\gcn{1}{3}{1}{-7}\gnl
\gvac{6}\gmpcu{\vartheta}\gnl
\gcl{4}\gnl
\gvac{5}\gmu\gnl
\gvac{5}\gcn{1}{1}{2}{1}\gnl
\gvac{5}\gmpcu{\vartheta}\gnl
\gob{1}{X}
\gend
=
\gbeg{9}{18}
\gvac{6}\got{1}{X}\got{1}{R}\got{1}{R}\gnl
\gvac{2}\gdb\gvac{2}\gcl{8}\gcl{8}\gcl{8}\gnl
\gvac{2}\gcn{1}{1}{1}{-3}\gcn{1}{1}{1}{5}\gnl
\gcl{1}\gvac{2}\gsbox{2}\gnot{\hspace{5mm}e}\gvac{2}\gcl{5}\gnl
\gcl{1}\gvac{2}\gcn{1}{1}{1}{-3}\gcl{2}\gnl
\gcl{1}\gcl{1}\gsbox{2}\gnot{\hspace{5mm}e}\gnl
\gcl{3}\gcl{3}\gcl{3}\gmu\gnl
\gvac{3}\gcn{1}{1}{2}{3}\gnl
\gvac{4}\glm\gnl
\gcl{8}\gcl{3}\gcn{1}{1}{1}{7}\gvac{2}\gev\gcn{1}{1}{1}{-1}\gcl{3}\gnl
\gvac{5}\gmu\gnl
\gvac{5}\gcn{1}{1}{2}{1}\gnl
\gvac{2}\gcn{1}{2}{-1}{5}\gvac{2}\gmpcu{\vartheta}\gvac{1}\gcn{1}{2}{3}{-3}\gnl
\gvac{4}\gnl
\gvac{4}\gmu\gnl
\gvac{4}\gcn{1}{1}{2}{1}\gnl
\gvac{4}\gmpcu{\vartheta}\gnl
\gob{1}{X}
\gend}\\
&&\hspace*{2cm}
{\footnotesize
=
\gbeg{9}{18}
\gvac{6}\got{1}{X}\got{1}{R}\got{1}{R}\gnl
\gvac{2}\gdb\gvac{2}\gcl{6}\gcl{6}\gcl{6}\gnl
\gvac{2}\gcn{1}{1}{1}{-3}\gcn{1}{1}{1}{5}\gnl
\gcl{1}\gsbox{2}\gnot{\hspace{5mm}e}\gvac{4}\gcl{1}\gnl
\gcl{1}\gcl{1}\gcl{1}\gvac{2}\gcl{1}\gnl
\gcl{1}\gcl{1}\gcl{1}\gsbox{2}\gnot{\hspace{5mm}e}\gvac{2}\gcl{1}\gnl
\gcl{1}\gcl{1}\gcl{1}\gcl{1}\glm\gnl
\gcl{1}\gcl{1}\gcl{1}\gcn{1}{1}{1}{5}\gvac{1}\gev\gcn{1}{1}{1}{-1}\gcl{6}\gnl
\gcl{1}\gcl{1}\gcl{1}\gvac{2}\gmu\gnl
\gcl{1}\gcl{1}\gcn{1}{1}{1}{5}\gvac{2}\gcn{1}{1}{2}{1}\gnl
\gcl{2}\gcl{2}\gvac{2}\gmu\gnl
\gcl{2}\gcl{2}\gvac{2}\gcn{1}{1}{2}{1}\gnl
\gvac{4}\gmpcu{\vartheta}\gnl
\gcl{1}\gcn{1}{1}{1}{5}\gvac{4}\gcn{1}{1}{5}{-3}\gnl
\gcl{1}\gvac{2}\gmu\gnl
\gcl{1}\gvac{2}\gcn{1}{1}{2}{1}\gnl
\gcl{1}\gvac{2}\gmpcu{\vartheta}\gnl
\gob{1}{X}
\gend
=
\gbeg{7}{11}
\gvac{4}\got{1}{X}\got{1}{R}\got{1}{R}\gnl
\gvac{1}\gdb\gvac{1}\gcl{1}\gmu\gnl
\gvac{1}\gcn{1}{1}{1}{-1}\gcn{1}{1}{1}{3}\gvac{1}\gcl{1}\gcn{1}{1}{2}{1}\gnl
\gcl{1}\gsbox{2}\gnot{\hspace{5mm}e}\gvac{2}\gcl{1}\gcl{1}\gcl{1}\gnl
\gcl{1}\gcl{1}\glm\gcl{1}\gcl{1}\gnl
\gcl{1}\gcl{1}\gvac{1}\gev\gcl{1}\gnl
\gcl{1}\gcn{1}{1}{1}{5}\gvac{3}\gcn{1}{1}{1}{-1}\gnl
\gcl{1}\gvac{2}\gmu\gnl
\gcl{1}\gvac{2}\gcn{1}{1}{2}{1}\gnl
\gcl{1}\gvac{2}\gmpcu{\vartheta}\gnl
\gob{1}{X}
\gend}
\end{eqnarray*}
it follows that $X$ is a right $R$-module in $\Cc$ via $\nu^R_X$, too. It is, moreover, an $R$-bimodule 
because of the defining properties of an adjunction and since $Y$ is so, we leave to the reader the 
verification of these details.  

It is clear that $\r_0=\d q^R_{X, Y}$. A simple computation based on the properties of an 
adjunction and on the fact that $Y$ is an $R$-bimodule ensures us that $\gamma$ fits in the 
universal property of the coequalizer 
\[
\xymatrix{
Y\ot R\ot X\ar@<1ex>[rr]^-{\nu^R_Y\ot \Id_X}\ar@<-1ex>[rr]_-{\Id_Y\ot \mu^R_X}&&Y\ot X
\ar[r]^-{q^R_{Y, X}}&Y\ot_R X~. 
}
\]
Thus there is a unique morphism $\l_0: Y\ot_RX\ra R$ such that $\l_0q^R_{Y, X}=\gamma$. 
 
The morphisms $\rho_0$ and $\l_0$ constructed above are morphisms in ${}_R\Cc_R$. For instance, 
\begin{eqnarray*}
\l_0\mu^R_{Y\ot_RX}(\Id_R\ot q^R_{Y, X})&=&\l_0q^R_{Y, X}(\mu^R_Y\ot \Id_X)=\gamma (\mu^R_Y\ot \Id_X)\\
&=&\footnotesize{
\gbeg{4}{6}
\gvac{1}\got{1}{R}\got{1}{Y}\got{1}{X}\gnl
\gvac{1}\glm\gcl{1}\gnl
\gsbox{2}\gnot{\hspace{5mm}e}\gvac{2}\gcl{1}\gcl{1}\gnl
\gcl{1}\glm\gcl{1}\gnl
\gcl{1}\gvac{1}\gev\gnl
\gob{1}{R}
\gend}
=
\footnotesize{
\gbeg{5}{8}
\gvac{2}\got{1}{R}\got{1}{Y}\got{1}{X}\gnl
\gvac{2}\gcl{1}\gcl{1}\gcl{1}\gnl
\gsbox{2}\gnot{\hspace{5mm}e}\gvac{2}\gcl{1}\gcl{1}\gcl{1}\gnl
\gcl{1}\gmu\gcl{1}\gcl{1}\gnl
\gcl{1}\gcn{1}{1}{2}{3}\gvac{1}\gcl{1}\gcl{1}\gnl
\gcl{1}\gvac{1}\glm\gcl{1}\gnl
\gcl{1}\gvac{2}\gev\gnl
\gob{1}{R}
\gend
}
=
\footnotesize{
\gbeg{5}{6}
\got{1}{R}\gvac{2}\got{1}{Y}\got{1}{X}\gnl
\gcl{1}\gvac{2}\gcl{1}\gcl{1}\gnl
\gcl{1}\gsbox{2}\gnot{\hspace{5mm}e}\gvac{2}\gcl{1}\gcl{1}\gnl
\gmu\glm\gcl{1}\gnl
\gcn{1}{1}{2}{2}\gvac{2}\gev\gnl
\gob{2}{R}
\gend
}\\
&=&\un{m}_R(\Id_R\ot \gamma)=\un{m}_R(\Id_R\ot \l_0)(\Id_R\ot q^R_{Y, X}),
\end{eqnarray*}
and this shows that $\l_0$ is left $R$-linear in $\Cc$. Similarly,
\begin{eqnarray*}
\l_0\nu^R_{Y\ot_RX}(q^R_{Y, X}\ot \Id_R)
&=&\l_0q^R_{Y, X}(\Id_X\ot \mu^R_X)=\gamma(\Id_X\ot \mu^R_X)\\
&=&{\footnotesize
\gbeg{9}{10}
\gvac{2}\got{1}{Y}\gvac{4}\got{1}{X}\got{1}{R}\gnl
\gvac{2}\gcl{1}\gvac{1}\gdb\gvac{1}\gcl{1}\gcl{1}\gnl
\gsbox{2}\gnot{\hspace{5mm}e}\gvac{2}\gcl{1}\gcn{1}{1}{3}{1}\gvac{1}\gcn{1}{1}{1}{3}\gvac{1}\gcl{1}\gcl{1}\gnl
\gcl{1}\glm\gcl{1}\gsbox{2}\gnot{\hspace{5mm}e}\gvac{2}\gcl{1}\gcl{1}\gcl{1}\gnl
\gcl{1}\gvac{1}\gev\gcl{1}\glm\gcl{1}\gcl{1}\gnl
\gcl{1}\gvac{3}\gcn{1}{1}{1}{5}\gvac{1}\gev\gcn{1}{1}{1}{-1}\gnl
\gcl{1}\gvac{5}\gmu\gnl
\gcl{1}\gvac{5}\gcn{1}{1}{2}{1}\gnl
\gcl{1}\gvac{5}\gmpcu{\vartheta}\gnl
\gob{1}{R}
\gend
=
\gbeg{7}{13}
\gvac{4}\got{1}{Y}\got{1}{X}\got{1}{R}\gnl
\gvac{4}\gcl{1}\gcl{1}\gcl{1}\gnl
\gvac{2}\gsbox{2}\gnot{\hspace{5mm}e}\gvac{2}\gcl{1}\gcl{1}\gcl{1}\gnl
\gvac{2}\gcn{1}{1}{1}{-3}\gcl{1}\gcl{1}\gcl{1}\gcl{1}\gnl
\gcl{1}\gsbox{2}\gnot{\hspace{5mm}e}\gvac{2}\gcl{1}\gcl{1}\gcl{1}\gcl{1}\gnl
\gcl{1}\gcl{1}\gmu\gcl{1}\gcl{1}\gcl{1}\gnl
\gcl{1}\gcl{1}\gcn{1}{1}{2}{3}\gvac{1}\gcl{1}\gcl{1}\gcl{1}\gnl
\gcl{1}\gcl{1}\gvac{1}\glm\gcl{1}\gcl{1}\gnl
\gcl{1}\gcn{1}{1}{1}{7}\gvac{2}\gev\gcn{1}{1}{1}{-1}\gnl
\gcl{1}\gvac{3}\gmu\gnl
\gcl{1}\gvac{3}\gcn{1}{1}{2}{1}\gnl
\gcl{1}\gvac{3}\gmpcu{\vartheta}\gnl
\gob{1}{R}
\gend
=
\gbeg{7}{11}
\gvac{4}\got{1}{Y}\got{1}{X}\got{1}{R}\gnl
\gvac{4}\gcl{1}\gcl{1}\gcl{1}\gnl
\gvac{2}\gsbox{2}\gnot{\hspace{5mm}e}\gvac{2}\gcl{1}\gcl{1}\gcl{1}\gnl
\gvac{2}\gcl{1}\glm\gcl{1}\gcl{1}\gnl
\gsbox{2}\gnot{\hspace{5mm}e}\gvac{2}\gcn{1}{1}{1}{5}\gvac{1}\gev\gcn{1}{1}{1}{-1}\gnl
\gcl{1}\gcl{1}\gvac{2}\gmu\gnl
\gcl{1}\gcn{1}{1}{1}{3}\gvac{2}\gcn{1}{1}{2}{-1}\gnl
\gcl{1}\gvac{1}\gmu\gnl
\gcl{1}\gvac{1}\gcn{1}{1}{2}{1}\gnl
\gcl{1}\gvac{1}\gmpcu{\vartheta}\gnl
\gob{1}{R}
\gend
}\\
&=&{\footnotesize 
\gbeg{5}{7}
\gvac{2}\got{1}{Y}\got{1}{X}\got{1}{R}\gnl
\gvac{2}\gcl{1}\gcl{1}\gcl{1}\gnl
\gsbox{2}\gnot{\hspace{5mm}e}\gvac{2}\gcl{1}\gcl{1}\gcl{1}\gnl
\gcl{1}\glm\gcl{1}\gcl{1}\gnl
\gcn{1}{1}{1}{5}\gvac{1}\gev\gcn{1}{1}{1}{-1}\gnl
\gvac{2}\gmu\gnl
\gvac{2}\gob{2}{R}
\gend
}
=\un{m}_R(\gamma\ot \Id_R)=\un{m}_R(\l_0q^R_{Y, X}\ot \Id_X),
\end{eqnarray*}
and so $\l_0$ is right $R$-linear, too. In a similar manner we can show that $\r_0$ is $R$-bilinear,  
once more we leave the details to the reader.  

It remains to prove that $(\rho_0, \l_0): Y\dashv X$ is an adjunction in ${}_R\Cc_R$. Towards this 
end we compute 
\begin{eqnarray*}
\Upsilon_X\widetilde{\lambda_0}\Sigma'_{X, Y, X}\widehat{\rho_0}\Upsilon^{'-1}_X&=&
\Upsilon_X\widetilde{\l_0}\Sigma'_{X, Y, X}q^R_{X\ot_RY, X}(q^R_{X, Y}\delta\un{\eta}_R\ot \Id_X)\\
&=&\Upsilon_Xq^R_{X, R}(\Id_X\ot \l_0q^R_{Y, X})(\d\un{\eta}_R\ot \Id_X)=\nu^R_X(\Id_X\ot \gamma)(\delta\un{\eta}_R\ot \Id_X)\\
&=&{\footnotesize
\gbeg{6}{10}
\gvac{5}\got{1}{X}\gnl
\gvac{1}\gdb\gvac{2}\gcl{1}\gnl
\gvac{1}\gcn{1}{1}{1}{-1}\gcn{1}{1}{1}{5}\gvac{2}\gcl{1}\gnl
\gcl{1}\gsbox{2}\gnot{\hspace{5mm}e}\gvac{3}\gcl{1}\gcl{1}\gnl
\gcl{1}\gcl{1}\gcl{1}\gmpu{\alpha}\gcl{1}\gcl{1}\gnl
\gcl{1}\gcl{1}\gmu\gcl{1}\gcl{1}\gnl
\grm\gcn{1}{1}{2}{3}\gvac{1}\gcl{1}\gcl{1}\gnl
\gcl{1}\gvac{2}\glm\gcl{1}\gnl
\gcl{1}\gvac{3}\gev\gnl
\gob{1}{X}
\gend
=\gbeg{7}{13}
\gvac{6}\got{1}{X}\gnl
\gvac{3}\gdb\gvac{1}\gcl{1}\gnl
\gvac{3}\gcn{1}{1}{1}{-5}\gcn{1}{1}{1}{3}\gvac{1}\gcl{1}\gnl
\gcl{1}\gvac{2}\gsbox{2}\gnot{\hspace{5mm}e}\gvac{2}\gcl{1}\gcl{1}\gnl
\gcl{1}\gvac{2}\gcn{1}{1}{1}{-3}\glm\gcl{2}\gnl
\gcl{1}\gcl{1}\gvac{1}\gsbox{2}\gnot{\hspace{5mm}e}\gvac{2}\gcl{1}\gnl
\gcl{1}\gcl{1}\gmpu{\alpha}\gcl{1}\glm\gcl{1}\gnl
\gcl{1}\gcl{1}\gmu\gvac{1}\gev\gnl
\gcl{1}\gcl{1}\gcn{1}{1}{2}{1}\gnl
\gcl{1}\gmu\gnl
\gcl{1}\gcn{1}{1}{2}{1}\gnl
\gcl{1}\gmpcu{\vartheta}\gnl
\gob{1}{X}
\gend
=
\gbeg{8}{14}
\gvac{7}\got{1}{X}\gnl
\gvac{3}\gdb\gvac{2}\gcl{1}\gnl
\gvac{3}\gcn{1}{1}{1}{-5}\gcn{1}{1}{1}{5}\gvac{2}\gcl{1}\gnl
\gcl{1}\gsbox{2}\gnot{\hspace{5mm}e}\gvac{5}\gcl{4}\gcl{4}\gnl
\gcl{1}\gcl{1}\gcn{1}{1}{1}{3}\gnl
\gcl{1}\gcl{1}\gmpu{\alpha}\gcl{1}\gsbox{2}\gnot{\hspace{5mm}e}\gnl
\gcl{1}\gmu\gcl{1}\gcl{4}\glm\gnl
\gcl{1}\gcn{1}{1}{2}{3}\gvac{1}\gcl{1}\gcl{1}\gvac{1}\gev\gnl
\gcl{1}\gvac{1}\gmu\gnl
\gcl{1}\gvac{1}\gcn{1}{1}{2}{3}\gnl
\gcl{1}\gvac{2}\gmu\gnl
\gcl{1}\gvac{2}\gcn{1}{1}{2}{1}\gnl
\gcl{1}\gvac{2}\gmpcu{\vartheta}\gnl
\gob{1}{R}
\gend}\\
&=&{\footnotesize
\gbeg{5}{7}
\gvac{4}\got{1}{X}\gnl
\gvac{1}\gdb\gvac{1}\gcl{1}\gnl
\gvac{1}\gcn{1}{1}{1}{-1}\gcn{1}{1}{1}{3}\gvac{1}\gcl{1}\gnl
\gcl{1}\gsbox{2}\gnot{\hspace{5mm}e}\gvac{2}\gcl{1}\gcl{1}\gnl
\gcl{1}\gcl{1}\glm\gcl{1}\gnl
\gcl{1}\gmpcu{\vartheta}\gvac{1}\gev\gnl
\gob{1}{X}
\gend
=
\gbeg{3}{4}
\gvac{2}\got{1}{X}\gnl
\gdb\gcl{1}\gnl
\gcl{1}\gev\gnl
\gob{1}{X}
\gend
=\Id_X,
} 
\end{eqnarray*}
as required. Analogously we have $\Upsilon'_Y\widehat{\l_0}\Gamma'_{Y, X, Y}\widetilde{\rho_0}\Upsilon^{-1}_Y=\Id_Y$. 
This ends the proof.
   
\end{proof}
 
\begin{corollary}\colabel{DualBimequalDual}
Let $R$ be a Frobenius separable algebra in a monoidal category $\Cc$ with coequalizers and such that 
any object of it is coflat and left robust, and let $X$ be an $R$-bimodule in $\Cc$. Then $X$ admits a (left) right 
dual in ${}_R\Cc_R$ if and only if $X$ admits a (left) right dual in $\Cc$. 
\end{corollary}
\begin{proof}
The if part follows from \coref{DualRbimAredual}. For the converse, if $X$ has a (left) right dual ($X^*$) ${}^*X$ 
then we have an adjunction ($({\rm ev}_X, {\rm coev}_X): X^*\dashv X$) 
$({\rm ev}'_X, {\rm coev}'_X): X\dashv {}^*X$ in $\Cc$. Since $X$ is an $R$-bimodule in $\Cc$ it follows that 
we have an adjunction ($(({\rm ev}_X)^0, ({\rm coev}_X)^0): X^*\dashv X$) 
$(({\rm ev}'_X)_0, ({\rm coev}'_X)_0): X\dashv {}^*X$ in ${}_R\Cc_R$, cf. \prref{DualAreDualBimod}. 
This finishes the proof.
\end{proof}

We are now in position to prove one of the main results of the paper.

\begin{theorem}\thlabel{FrobExtBaseFS}
Let $\Cc$ be a monoidal category with coequalizers and such that any object of it is coflat and left robust. 
Let $R$ be a Frobenius separable algebra and $i: R\ra S$ an algebra extension in $\Cc$. Then $i: R\ra S$ 
is Frobenius if and only if $S$ is a Frobenius algebra in $\Cc$ and the following equality holds
\begin{equation}\eqlabel{CharFrobExt}
{\footnotesize
\gbeg{2}{6}
\got{1}{S}\got{1}{R}\gnl
\gcl{1}\gmp{i}\gnl
\gmu\gnl
\gcn{1}{1}{2}{1}\gnl
\gmpcu{\widetilde{\vartheta}}\gnl
\gob{1}{\un{1}}
\gend
=
\gbeg{4}{12}
\gvac{2}\got{1}{S}\got{1}{R}\gnl
\gvac{2}\gcl{1}\gcl{1}\gnl
\gsbox{2}\gnot{\hspace{5mm}e}\gvac{2}\gcl{1}\gcl{1}\gnl
\gcl{1}\gmp{i}\gcl{1}\gcl{1}\gnl
\gcl{1}\gmu\gcl{1}\gnl
\gcl{1}\gcn{1}{1}{2}{1}\gvac{1}\gcl{1}\gnl
\gcl{1}\gmpcu{\widetilde{\vartheta}}\gvac{1}\gcn{1}{2}{1}{-3}\gnl
\gcl{1}\gnl
\gmu\gnl
\gcn{1}{1}{2}{3}\gnl
\gvac{1}\gmpcu{\vartheta}\gnl
\gvac{1}\gob{1}{\un{1}}
\gend
},
\end{equation} 
where $(\vartheta, e)$ is a Frobenius pair for $R$ and $(\widetilde{\vartheta}, \widetilde{e})$ 
is a Frobenius pair for $S$, respectively.
\end{theorem}
\begin{proof}
If $i: R\ra S$ is a Frobenius algebra extension with $R$ a Frobenius separable 
algebra in $\Cc$ then we have seen that $S$ is a Frobenius algebra in $\Cc$ (\coref{FrobExtVsFrobAlg}). 
Furthermore, if $\vartheta: R\ra \un{1}$ and $\vartheta': S\ra R$ are the Frobenius morphisms corresponding to 
$R$ and to extension $i: R\ra S$, respectively, then $S$ is Frobenius with Frobenius morphism given by 
$\widetilde{\vartheta}=\vartheta\vartheta'$. Consider now $e$ the Casimir morphism of the Frobenius algebra 
$R$, corresponding to $\vartheta$. Since $\vartheta'$ is a morphism in ${}_R\Cc_R$ we compute that 
\[
{\footnotesize
\gbeg{4}{12}
\gvac{2}\got{1}{S}\got{1}{R}\gnl
\gvac{2}\gcl{1}\gcl{1}\gnl
\gsbox{2}\gnot{\hspace{5mm}e}\gvac{2}\gcl{1}\gcl{1}\gnl
\gcl{1}\gmp{i}\gcl{1}\gcl{1}\gnl
\gcl{1}\gmu\gcl{1}\gnl
\gcl{1}\gcn{1}{1}{2}{1}\gvac{1}\gcl{1}\gnl
\gcl{1}\gmpcu{\widetilde{\vartheta}}\gvac{1}\gcn{1}{2}{1}{-3}\gnl
\gcl{1}\gnl
\gmu\gnl
\gcn{1}{1}{2}{3}\gnl
\gvac{1}\gmpcu{\vartheta}\gnl
\gvac{1}\gob{1}{\un{1}}
\gend
=
\gbeg{4}{12}
\gvac{2}\got{1}{S}\got{1}{R}\gnl
\gvac{2}\gcl{1}\gcl{1}\gnl
\gsbox{2}\gnot{\hspace{5mm}e}\gvac{2}\gmp{\vartheta'}\gcl{1}\gnl
\gcl{1}\gmu\gcl{1}\gnl
\gcl{1}\gcn{1}{1}{2}{3}\gvac{1}\gcl{1}\gnl
\gcl{1}\gvac{1}\gmpcu{\vartheta}\gcl{1}\gnl
\gcn{1}{1}{1}{5}\gvac{2}\gcl{1}\gnl
\gvac{2}\gmu\gnl
\gvac{2}\gcn{1}{1}{2}{1}\gnl
\gvac{2}\gmpcu{\vartheta}\gnl
\gvac{2}\gob{1}{\un{1}}
\gend
=
\gbeg{2}{6}
\got{1}{S}\got{1}{R}\gnl
\gmp{\vartheta'}\gcl{1}\gnl
\gmu\gnl
\gcn{1}{1}{2}{1}\gnl
\gmpcu{\vartheta}\gnl
\gob{1}{\un{1}}\gnl
\gend
=
\gbeg{2}{6}
\got{1}{S}\got{1}{R}\gnl
\gcl{1}\gmp{i}\gnl
\gmu\gnl
\gcn{1}{1}{2}{1}\gnl
\gmpcu{\widetilde{\vartheta}}\gnl
\gob{1}{\un{1}}
\gend
},
\]
as wished. Conversely, assume that $S$ is a Frobenius algebra in $\Cc$ and let 
$(\widetilde{\vartheta}, \widetilde{e})$ be a Frobenius pair for it. By \thref{charactFrobalg} 
we know that $S$ admits a right dual object ${}^*S$ in $\Cc$ and, moreover, the morphism 
$\Phi: S\ra {}^*S$ given by 
$\Phi={\footnotesize 
\gbeg{3}{6}
\gvac{2}\got{1}{S}\gnl
\gvac{1}\gnot{\hspace*{-4mm}\vspace*{2mm}\bullet}\gvac{-1}\gdb\gcl{1}\gnl
\gcl{1}\gmu\gnl
\gcl{1}\gcn{1}{1}{2}{1}\gnl
\gcl{1}\gmpcu{\widetilde{\vartheta}}\gnl
\gob{1}{{}^*S}
\gend
}$ 
is an isomorphism of left $S$-modules in $\Cc$. This comes out explicitly as    
\[
{\footnotesize
\gbeg{2}{5}
\got{1}{S}\got{1}{S}\gnl
\gmu\gnl
\gcn{1}{1}{2}{1}\gnl
\gmp{\Phi}\gnl
\gob{1}{{}^*S}
\gend
=
\gbeg{4}{6}
\gvac{2}\got{1}{S}\got{1}{S}\gnl
\gvac{1}\gnot{\hspace*{-4mm}\vspace*{2mm}\bullet}\gvac{-1}\gdb\gcl{1}\gcl{1}\gnl
\gcl{1}\gmu\gmp{\Phi}\gnl
\gcl{1}\gcn{1}{1}{2}{3}\gvac{1}\gcl{1}\gnl
\gcl{1}\gvac{2}\gnot{\hspace*{-4mm}\vspace*{-2mm}\bullet}\gvac{-1}\gev\gnl
\gob{1}{{}^*S}
\gend
}~,~\mbox{and implies}~~
{\footnotesize
\gbeg{4}{5}
\got{1}{S}\got{1}{S}\got{1}{S}\gnl
\gmu\gcl{1}\gnl
\gcn{1}{1}{2}{3}\gvac{1}\gmp{\Phi}\gnl
\gvac{2}\gnot{\hspace*{-4mm}\vspace*{-2mm}\bullet}\gvac{-1}\gev\gnl
\gvac{1}\gob{2}{\un{1}}
\gend
=
\gbeg{3}{6}
\got{1}{S}\got{1}{S}\got{1}{S}\gnl
\gcl{1}\gmu\gnl
\gcl{1}\gcn{1}{1}{2}{1}\gnl
\gcl{1}\gmp{\Phi}\gnl
\gvac{1}\gnot{\hspace*{-4mm}\vspace*{-2mm}\bullet}\gvac{-1}\gev\gnl
\gob{2}{\un{1}}
\gend
}.
\]
By \coref{DualBimequalDual} we have that ${}^*S$ is an $R$-bimodule in $\Cc$ via the structure morphisms 
given by 
\[
\mu^R_{{}^*S}=
{\footnotesize
\gbeg{4}{7}
\gvac{2}\got{1}{R}\got{1}{S}\gnl
\gvac{1}\gnot{\hspace*{-4mm}\vspace*{2mm}\bullet}\gvac{-1}\gdb\gcl{1}\gcl{1}\gnl
\gcl{1}\gcl{1}\gmp{i}\gcl{1}\gnl
\gcl{1}\gmu\gcl{1}\gnl
\gcl{1}\gcn{1}{1}{2}{3}\gvac{1}\gcl{1}\gnl
\gcl{1}\gvac{2}\gnot{\hspace*{-4mm}\vspace*{-2mm}\bullet}\gvac{-1}\gev\gnl
\gob{1}{{}^*S}
\gend}
~\mbox{and}~
\nu^R_{{}^*S}=
{\footnotesize
\gbeg{6}{13}
\gvac{4}\got{1}{{}^*S}\got{1}{R}\gnl
\gvac{2}\gnot{\hspace*{-4mm}\vspace*{2mm}\bullet}\gvac{-1}\gdb\gvac{1}\gcl{1}\gcl{1}\gnl
\gcn{1}{1}{3}{1}\gvac{1}\gcn{1}{1}{1}{3}\gvac{1}\gcl{1}\gcl{1}\gnl
\gcl{1}\gsbox{2}\gnot{\hspace{5mm}e}\gvac{2}\gcl{1}\gcl{1}\gcl{1}\gnl
\gcl{1}\gcl{1}\gmp{i}\gcl{1}\gcl{1}\gcl{1}\gnl
\gcl{1}\gcl{1}\gmu\gcl{1}\gcl{1}\gnl
\gcl{1}\gcl{1}\gcn{1}{1}{2}{3}\gvac{1}\gcl{1}\gcl{2}\gnl
\gcl{1}\gcl{1}\gvac{2}\gnot{\hspace*{-4mm}\vspace*{-2mm}\bullet}\gvac{-1}\gev\gnl
\gcl{1}\gcn{1}{1}{1}{5}\gvac{3}\gcn{1}{1}{1}{-1}\gnl
\gcl{1}\gvac{2}\gmu\gnl
\gcl{1}\gvac{2}\gcn{1}{1}{2}{1}\gnl
\gcl{1}\gvac{2}\gmpcu{\vartheta}\gnl
\gob{1}{{}^*S}
\gend}.
\] 
Observe that $\mu^R_{{}^*S}$ is nothing that the restriction at $R$ of the canonical left $S$-module 
structure on ${}^*S$ via the algebra morphism $i: R\ra S$. Thus $\Phi$ is left $R$-linear. 
$\Phi$ is also right $R$-linear since with the help of \equref{CharFrobExt} we compute that  
\[{\footnotesize
\gbeg{6}{12}
\gvac{4}\got{1}{S}\got{1}{R}\gnl
\gvac{2}\gnot{\hspace*{-4mm}\vspace*{2mm}\bullet}\gvac{-1}\gdb\gvac{1}\gcl{1}\gcl{1}\gnl
\gvac{1}\gcn{1}{1}{1}{-1}\gcn{1}{1}{1}{3}\gvac{1}\gmp{\Phi}\gcl{1}\gnl
\gcl{1}\gsbox{2}\gnot{\hspace{5mm}e}\gvac{2}\gcl{1}\gcl{1}\gcl{1}\gnl
\gcl{1}\gcl{1}\gmp{i}\gcl{1}\gcl{1}\gcl{1}\gnl
\gcl{1}\gcl{1}\gmu\gcl{1}\gcl{1}\gnl
\gcl{1}\gcl{1}\gcn{1}{1}{2}{3}\gvac{1}\gcl{1}\gcl{1}\gnl
\gcl{1}\gcn{1}{1}{1}{5}\gvac{2}\gnot{\hspace*{-4mm}\vspace*{-2mm}\bullet}\gvac{-1}\gev\gcn{1}{1}{1}{-1}\gnl
\gcl{1}\gvac{2}\gmu\gnl
\gcl{1}\gvac{2}\gcn{1}{1}{2}{1}\gnl
\gcl{1}\gvac{2}\gmpcu{\vartheta}\gnl
\gob{1}{{}^*S}
\gend
}
=
{\footnotesize
\gbeg{6}{12}
\gvac{4}\got{1}{S}\got{1}{R}\gnl
\gvac{2}\gnot{\hspace*{-4mm}\vspace*{2mm}\bullet}\gvac{-1}\gdb\gvac{1}\gcl{1}\gcl{1}\gnl
\gvac{1}\gcn{1}{1}{1}{-1}\gcn{1}{1}{1}{3}\gvac{1}\gcl{1}\gcl{1}\gnl
\gcl{1}\gsbox{2}\gnot{\hspace{5mm}e}\gvac{2}\gcl{1}\gcl{1}\gcl{1}\gnl
\gcl{1}\gcl{1}\gcl{1}\gmu\gcl{1}\gnl
\gcl{1}\gcl{1}\gmp{i}\gcn{1}{1}{2}{1}\gvac{1}\gcl{1}\gnl
\gcl{1}\gcl{1}\gcl{1}\gmp{\Phi}\gvac{1}\gcl{1}\gnl
\gcl{1}\gcn{1}{1}{1}{3}\gvac{1}\gnot{\hspace*{-4mm}\vspace*{-2mm}\bullet}\gvac{-1}\gev\gvac{1}\gcn{1}{1}{1}{-3}\gnl
\gcl{1}\gvac{1}\gmu\gnl
\gcl{1}\gvac{1}\gcn{1}{1}{2}{1}\gnl
\gcl{1}\gvac{1}\gmpcu{\vartheta}\gnl
\gob{1}{{}^*S}
\gend
}
=
{\footnotesize
\gbeg{6}{13}
\gvac{4}\got{1}{S}\got{1}{R}\gnl
\gvac{2}\gnot{\hspace*{-4mm}\vspace*{2mm}\bullet}\gvac{-1}\gdb\gvac{1}\gcl{1}\gcl{1}\gnl
\gvac{1}\gcn{1}{1}{1}{-1}\gcn{1}{1}{1}{3}\gvac{1}\gcl{1}\gcl{1}\gnl
\gcl{1}\gsbox{2}\gnot{\hspace{5mm}e}\gvac{2}\gmu\gcl{1}\gnl
\gcl{1}\gcl{1}\gmp{i}\gcn{1}{1}{2}{1}\gvac{1}\gcl{1}\gnl
\gcl{1}\gcl{1}\gmu\gvac{1}\gcl{2}\gnl
\gcl{1}\gcl{1}\gcn{1}{1}{2}{1}\gnl
\gcl{1}\gcl{1}\gmpcu{\widetilde{\vartheta}}\gvac{2}\gcl{1}\gnl
\gcl{1}\gcn{1}{1}{1}{5}\gvac{3}\gcn{1}{1}{1}{-1}\gnl
\gcl{1}\gvac{2}\gmu\gnl
\gcl{1}\gvac{2}\gcn{1}{1}{2}{1}\gnl
\gcl{1}\gvac{2}\gmpcu{\vartheta}\gnl
\gob{1}{{}^*S}
\gend
}
=
\footnotesize{
\gbeg{4}{8}
\gvac{2}\got{1}{S}\got{1}{R}\gnl
\gvac{1}\gnot{\hspace*{-4mm}\vspace*{2mm}\bullet}\gvac{-1}\gdb\gcl{1}\gcl{1}\gnl
\gcl{1}\gmu\gcl{1}\gnl
\gcl{1}\gcn{1}{1}{2}{3}\gvac{1}\gmp{i}\gnl
\gcl{1}\gvac{1}\gmu\gnl
\gcl{1}\gvac{1}\gcn{1}{1}{2}{1}\gnl
\gcl{1}\gvac{1}\gmpcu{\widetilde{\vartheta}}\gnl
\gob{1}{{}^*S}
\gend
}
=
{\footnotesize
\gbeg{2}{6}
\got{1}{S}\got{1}{R}\gnl
\gcl{1}\gmp{i}\gnl
\gmu\gnl
\gcn{1}{1}{2}{1}\gnl
\gmp{\Phi}\gnl
\gob{1}{{}^*S}
\gend,
}
\]
as needed. Therefore $\Phi: S\ra {}^*S$ is an isomorphism in ${}_R\Cc_R$ and a left 
$S$-module morphism in $\Cc$. We show next that $\Phi$ is a left $S$-linear morphism in ${}_R\Cc_R$ 
between $S$ and ${}^*S$, where ${}^*S$ is considered this time as the right dual of $S$ in 
${}_R\Cc_R$. According to \coref{CharFrobAlgExt} this would end the proof. 
 
Indeed, recall that ${}^*S$ is also a right dual for 
$S$ in ${}_R\Cc_R$ via the evaluation and coevaluation morphisms in ${}_R\Cc_R$ completely determined 
by 
\[
{\rm ev}'^R_Sq^R_{S, {}^*S}:=({\rm ev}'_S)_0q^R_{S, {}^*S}=\gamma={\footnotesize
\gbeg{4}{8}
\gvac{2}\got{1}{S}\got{1}{{}^*S}\gnl
\gvac{2}\gcl{1}\gcl{1}\gnl
\gsbox{2}\gnot{\hspace{5mm}e}\gvac{2}\gcl{1}\gcl{1}\gnl
\gcl{1}\gmp{i}\gcl{1}\gcl{1}\gnl
\gcl{1}\gmu\gcl{1}\gnl
\gcl{1}\gcn{1}{1}{2}{3}\gvac{1}\gcl{1}\gnl
\gcl{1}\gvac{2}\gnot{\hspace*{-4mm}\vspace*{-2mm}\bullet}\gvac{-1}\gev\gnl
\gob{1}{R}
\gend
}
~\mbox{and}~
{\rm coev}'^R_S:=({\rm coev}'_S)_0=q^R_{{}^*S, S}\circ \d=
q^R_{{}^*S, S}\circ 
{\footnotesize
\gbeg{4}{7}
\gvac{3}\got{1}{R}\gnl
\gvac{2}\gnot{\hspace*{-4mm}\vspace*{2mm}\bullet}\gvac{-1}\gdb\gcl{1}\gnl
\gvac{1}\gcn{1}{1}{1}{-1}\gcl{1}\gmp{i}\gnl
\gcl{1}\gmpu{\alpha}\gmu\gnl
\gcl{1}\gmp{i}\gcn{1}{1}{2}{1}\gnl
\gcl{1}\gmu\gnl
\gob{1}{{}^*S}\gob{2}{S}
\gend
}
~,
\]
respectively. Thus, if $\mu^{R, S}_{{}^*S}$ denotes the left $S$-module structure of 
${}^*S$ in ${}_R\Cc_R$ then 
\begin{eqnarray*}
\mu^{R, S}_{{}^*S}q^R_{S, {}^*S}&=&\Upsilon_{{}^*S}\widetilde{{\rm ev}'^R_S}
\widetilde{\widehat{\un{m}^R_S}}\widetilde{\Gamma'_{S, S, {}^*S}}\Sigma'_{{}^*S, S, S\ot_R{}^*S}
\widehat{{\rm coev}'^R_S}{\Upsilon}'^{-1}_{S\ot_R{}^*S}q^R_{S, {}^*S}\\
&=&\Upsilon_{{}^*S}\widetilde{{\rm ev}'^R_S}
\widetilde{\widehat{\un{m}^R_S}}\widetilde{\Gamma'_{S, S, {}^*S}}
\Sigma'_{{}^*S, S, S\ot_R{}^*S}q^R_{{}^*S\ot_RS, S\ot_R{}^*S}(q^R_{{}^*S, S}\d \ot\Id_{S\ot_R{}^*S})\\
&&(\Id_R\ot q^R_{S, {}^*S})(\un{\eta}_R\ot \Id_{S\ot {}^*S})\\
&=&\Upsilon_{{}^*S}\widetilde{{\rm ev}'^R_S}
\widetilde{\widehat{\un{m}^R_S}}\widetilde{\Gamma'_{S, S, {}^*S}}q^R_{{}^*S, S\ot_R(S\ot_R{}^*S)}
(\Id_{{}^*S}\ot q^R_{S, S\ot_R{}^*S})(\delta\un{\eta}_R\ot q^R_{S, {}^*S})\\
&=&\Upsilon_{{}^*S}\widetilde{{\rm ev}'^R_S}
\widetilde{\widehat{\un{m}^R_S}}q^R_{{}^*S, (S\ot_RS)\ot_R{}^*S}(\Id_{{}^*S}\ot \widehat{q^R_{S, S}}
\theta'^{-1}_{S, S, {}^*S}(\Id_S\ot q^R_{S, {}^*S}))(\delta\un{\eta}_R\ot \Id_{S\ot {}^*S})\\
&=&\Upsilon_{{}^*S}q^R_{{}^*S, R}(\Id_{{}^*S}\ot {\rm ev}'^R_S\widehat{\un{m}^R_Sq^R_{S, S}}q^R_{S\ot S, {}^*S})
(\delta\un{\eta}_R\ot \Id_{S\ot {}^*S})\\
&=&\nu^R_{{}^*S}(\Id_{{}^*S}\ot \gamma(\un{m}_S\ot \Id_{{}^*S}))(\delta\un{\eta}_R\ot \Id_{S\ot {}^*S})\\
&=&{\footnotesize 
\gbeg{8}{17}
\gvac{6}\got{1}{S}\got{1}{{}^*S}\gnl
\gvac{5}\gnot{\hspace*{-4mm}\vspace*{2mm}\bullet}\gvac{-1}\gdb\gcl{1}\gcl{1}\gnl
\gvac{4}\gcn{1}{1}{1}{-7}\gcl{1}\gcl{1}\gcl{1}\gnl
\gcl{1}\gvac{2}\gsbox{2}\gnot{\hspace{5mm}e}\gvac{2}\gcl{1}\gcl{1}\gcl{1}\gnl
\gcl{1}\gvac{2}\gcn{1}{1}{1}{-3}\gmp{i}\gcl{1}\gcl{1}\gcl{1}\gnl
\gcl{1}\gcl{1}\gvac{2}\gmu\gcl{1}\gcl{1}\gnl
\gcl{1}\gcl{1}\gvac{2}\gcn{1}{1}{2}{3}\gvac{1}\gcl{1}\gcl{1}\gnl
\gcl{1}\gcl{1}\gvac{2}\gmpu{\alpha}\gcl{1}\gcl{1}\gcl{1}\gnl
\gcl{1}\gcl{1}\gvac{2}\gmp{i}\gcl{1}\gcl{1}\gcl{1}\gnl
\gcl{1}\gcl{1}\gvac{2}\gmu\gcl{1}\gcl{1}\gnl
\gcl{1}\gcl{1}\gvac{2}\gcn{1}{1}{2}{3}\gvac{1}\gcl{1}\gcl{1}\gnl
\gcl{1}\gcl{1}\gsbox{2}\gnot{\hspace{5mm}e}\gvac{3}\gmu\gcl{1}\gnl
\gcl{1}\gmu\gmp{i}\gcn{1}{1}{4}{1}\gvac{2}\gcl{1}\gnl
\gcl{1}\gcn{1}{1}{2}{3}\gvac{1}\gmu\gvac{2}\gcn{1}{2}{1}{-3}\gnl
\gcl{1}\gvac{1}\gmpcu{\vartheta}\gcn{1}{1}{2}{3}\gnl
\gcl{1}\gvac{4}\gnot{\hspace*{-4mm}\vspace*{-2mm}\bullet}\gvac{-1}\gev\gnl
\gob{1}{{}^*S}
\gend}
=
{\footnotesize
\gbeg{7}{17}
\gvac{5}\got{1}{S}\got{1}{{}^*S}\gnl
\gvac{2}\gnot{\hspace*{-4mm}\vspace*{2mm}\bullet}\gvac{-1}\gdb\gvac{2}\gcl{1}\gcl{1}\gnl
\gvac{1}\gcn{1}{1}{1}{-1}\gcn{1}{1}{1}{5}\gvac{2}\gcl{1}\gcl{1}\gnl
\gcl{1}\gsbox{2}\gnot{\hspace{5mm}e}\gvac{3}\gcl{1}\gcl{1}\gcl{1}\gnl
\gcl{1}\gcl{1}\gcn{1}{1}{1}{3}\gvac{1}\gcl{1}\gcl{1}\gcl{1}\gnl
\gcl{1}\gcl{1}\gmpu{\alpha}\gcl{1}\gcl{1}\gcl{1}\gcl{1}\gnl
\gcl{1}\gcl{1}\gmu\gcl{1}\gcl{1}\gcl{1}\gnl
\gcl{1}\gcl{1}\gcn{1}{1}{2}{1}\gvac{1}\gcl{1}\gcl{1}\gcl{1}\gnl
\gcl{1}\gmu\gvac{1}\gcn{1}{1}{1}{-1}\gcl{1}\gcl{1}\gnl
\gcl{1}\gcn{1}{1}{2}{3}\gvac{1}\gcl{1}\gvac{1}\gcl{1}\gcl{1}\gnl
\gcl{1}\gvac{1}\gmp{i}\gcl{1}\gvac{1}\gcl{1}\gcl{1}\gnl
\gcl{1}\gvac{1}\gmu\gvac{1}\gcl{1}\gcl{1}\gnl
\gcl{1}\gvac{1}\gcn{1}{1}{2}{3}\gvac{2}\gcn{1}{1}{1}{-1}\gcl{1}\gnl
\gcl{1}\gvac{2}\gmu\gvac{1}\gcl{1}\gnl
\gcl{1}\gvac{2}\gcn{1}{1}{2}{3}\gvac{2}\gcn{1}{1}{1}{-1}\gnl
\gcl{1}\gvac{4}\gnot{\hspace*{-4mm}\vspace*{-2mm}\bullet}\gvac{-1}\gev\gnl
\gob{1}{{}^*S}
\gend
}
=
{\footnotesize
\gbeg{4}{6}
\gvac{2}\got{1}{S}\got{1}{{}^*S}\gnl
\gvac{1}\gnot{\hspace*{-4mm}\vspace*{2mm}\bullet}\gvac{-1}\gdb\gcl{1}\gcl{1}\gnl
\gcl{1}\gmu\gcl{1}\gnl
\gcl{1}\gcn{1}{1}{2}{3}\gvac{1}\gcl{1}\gnl
\gcl{1}\gvac{2}\gnot{\hspace*{-4mm}\vspace*{-2mm}\bullet}\gvac{-1}\gev\gnl
\gob{1}{{}^*S}
\gend}
=\mu^S_{{}^*S}.
\end{eqnarray*}
 Using the above equality we have that $\Phi$ is left $S$-linear in ${}_R\Cc_R$ if and only if 
\[
\mu^{R, S}_{{}^*S}\widetilde{\Phi}=\Phi\un{m}_S^R~\Leftrightarrow~ 
\mu^{R, S}_{{}^*S}\widetilde{\Phi}q^R_{S, S}=\Phi\un{m}_S^Rq^R_{S, S}~\Leftrightarrow~ 
\mu^{R, S}_{{}^*S}q^R_{S, {}^*S}(\Id_S\ot \Phi)=\Phi\un{m}_S~\Leftrightarrow~ 
\mu^S_{{}^*S}(\Id_S\ot \Phi)=\Phi\un{m}_S,
\]
that is, if and only if $\Phi$ is left $S$-linear in $\Cc$. So our proof is finished. 
\end{proof}

We look now more carefully at the relation \equref{CharFrobExt}. In the case when $\Cc$ is a category of vector 
spaces it says that the restriction at $R$ of the Nakayama automorphism of $S$ is equal to the 
Nakayama automorphism of $R$. Our next aim is to provide a similar interpretation for 
\equref{CharFrobExt} in a more general setting. Namely, the one provided by sovereign monoidal categories. 
Recall that a monoidal category is called sovereign if it is rigid and the left $()^*$ and 
right ${}^*()$ duality functors coincide. Note that braided monoidal categories are examples of 
sovereign monoidal categories, so the theory below applies to them. 
For more about braided categories and duality in monoidal categories we invite the reader 
to consult \cite{bulacu, k}.   

The Nakayama automorphism of a Frobenius algebra in a sovereign monoidal category $\Cc$ was introduced in 
\cite{fust}. More exactly, if $A$ is a Frobenius algebra then by \thref{charactFrobalg} 
we know that $A$ admits left and right dual objects. Since the category $\Cc$ is sovereign we have 
${}^*A=A^*:=\widehat{A}$, as objects in $\Cc$. Then the Nakayama automorphism of $A$ is defined as being 
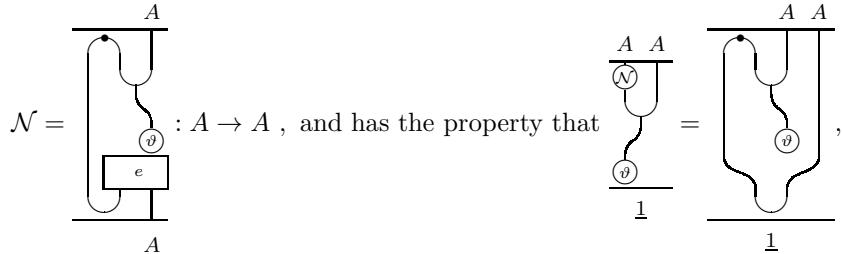
\[
\mathcal{N}={\footnotesize
\gbeg{3}{8}
\gvac{2}\got{1}{A}\gnl
\gvac{1}\gnot{\hspace*{-4mm}\vspace*{2mm}\bullet}\gvac{-1}\gdb\gcl{1}\gnl
\gcl{1}\gmu\gnl
\gcl{1}\gcn{1}{1}{2}{3}\gnl
\gcl{1}\gvac{1}\gmpcu{\vartheta}\gnl
\gcl{1}\gsbox{2}\gnot{\hspace{5mm}e}\gnl
\gev\gcl{1}\gnl
\gvac{2}\gob{1}{A}
\gend}: A\ra A~,~
\mbox{and has the property that}~
{\footnotesize
\gbeg{2}{6}
\got{1}{A}\got{1}{A}\gnl
\gmp{\mathcal{N}}\gcl{1}\gnl
\gmu\gnl
\gcn{1}{1}{2}{1}\gnl
\gmpcu{\vartheta}\gnl
\gob{2}{\un{1}}
\gend
}
=
{\footnotesize 
\gbeg{4}{8}
\gvac{2}\got{1}{A}\got{1}{A}\gnl
\gvac{1}\gnot{\hspace*{-4mm}\vspace*{2mm}\bullet}\gvac{-1}\gdb\gcl{1}\gcl{1}\gnl
\gcl{1}\gmu\gcl{1}\gnl
\gcl{1}\gcn{1}{1}{2}{3}\gvac{1}\gcl{1}\gnl
\gcl{1}\gvac{1}\gmpcu{\vartheta}\gcl{1}\gnl
\gcn{1}{1}{1}{3}\gvac{2}\gcn{1}{1}{1}{-1}\gnl
\gvac{1}\gev\gnl
\gvac{1}\gob{2}{\un{1}}
\gend},
\] 
where $(\vartheta, e)$ is a Frobenius pair for $A$. 
$\mathcal{N}$ is a unital algebra isomorphism in $\Cc$, cf. \cite[Proposition 18]{fust}. 

\begin{theorem}\thlabel{FrobExtBaseFSsov}
Let $\Cc$ be a sovereign monoidal category with coequalizers and such that any object of it is 
coflat and left robust. Let $R$ be a Frobenius separable algebra in $\Cc$ and $i: R\ra S$ an algebra 
extension. Then $i: R\ra S$ is Frobenius if and only if $S$ is Frobenius in $\Cc$ and 
$\widetilde{\mathcal{N}}\circ i=i\circ \mathcal{N}$, where 
$\widetilde{\mathcal{N}}$ and $\mathcal{N}$ are the Nakayama automorphisms of $S$ and $R$, 
respectively.  
\end{theorem}
\begin{proof}
In view of \thref{FrobExtBaseFS} it is enough to prove that \equref{CharFrobExt} is equivalent to 
$\widetilde{\mathcal{N}}\circ i=i\circ \mathcal{N}$. To this end denote by $(e, \vartheta)$ and 
$(\widetilde{e}, \widetilde{\vartheta})$ a Frobenius pair for $R$ and $S$, respectively. 
Since $\Cc$ is sovereign monoidal we have that ${}^*\widetilde{e}=\widetilde{e}^*$, and this amounts 
to 
$
{\footnotesize
\gbeg{4}{6}
\gvac{2}\got{1}{\widehat{S}}\got{1}{\widehat{S}}\gnl
\gvac{2}\gcl{1}\gcl{1}\gnl
\gsbox{2}\gnot{\hspace{5mm}\widetilde{e}}\gvac{2}\gcl{1}\gcl{1}\gnl
\gcn{1}{1}{1}{3}\gvac{1}\gnot{\hspace*{-4mm}\vspace*{-2mm}\bullet}\gvac{-1}\gev\gcn{1}{1}{1}{-1}\gnl
\gvac{2}\gnot{\hspace*{-4mm}\vspace*{-2mm}\bullet}\gvac{-1}\gev\gnl
\gvac{1}\gob{2}{\un{1}}
\gend}
=
{\footnotesize
\gbeg{4}{6}
\got{1}{\widehat{S}}\got{1}{\widehat{S}}\gnl
\gcl{1}\gcl{1}\gnl
\gcl{1}\gcl{1}\gsbox{2}\gnot{\hspace{5mm}\widetilde{e}}\gnl
\gcn{1}{1}{1}{3}\gev\gcn{1}{1}{1}{-1}\gnl
\gvac{1}\gev\gnl
\gvac{1}\gob{2}{\un{1}}
\gend}
$.   
Also, in the proof of implication $(i)\Rightarrow (iii)$ in \thref{charactFrobalg} 
the isomorphism $\Phi: S\ra {}^*S=\widehat{S}$ of left $S$-modules in $\Cc$ and its inverse 
$\phi$ are given by 
$
\Phi={\footnotesize
\gbeg{3}{6}
\gvac{2}\got{1}{S}\gnl
\gvac{1}\gnot{\hspace*{-4mm}\vspace*{2mm}\bullet}\gvac{-1}\gdb\gcl{1}\gnl 
\gcl{1}\gmu\gnl
\gcl{1}\gcn{1}{1}{2}{1}\gnl
\gcl{1}\gmpcu{\widetilde{\vartheta}}\gnl
\gob{1}{\widehat{S}}\gnl
\gend
}
$ 
and
$ 
\phi={\footnotesize
\gbeg{3}{5}
\gvac{2}\got{1}{\widehat{S}}\gnl
\gvac{2}\gcl{1}\gnl
\gsbox{2}\gnot{\hspace{5mm}\widetilde{e}}\gvac{2}\gcl{1}\gnl
\gcl{1}\gvac{1}\gnot{\hspace*{-4mm}\vspace*{-2mm}\bullet}\gvac{-1}\gev\gnl
\gob{1}{S}
\gend
}$, respectively. 

Thus we have that $\widetilde{\mathcal{N}}\circ i=i\circ \mathcal{N}$ is equivalent to 
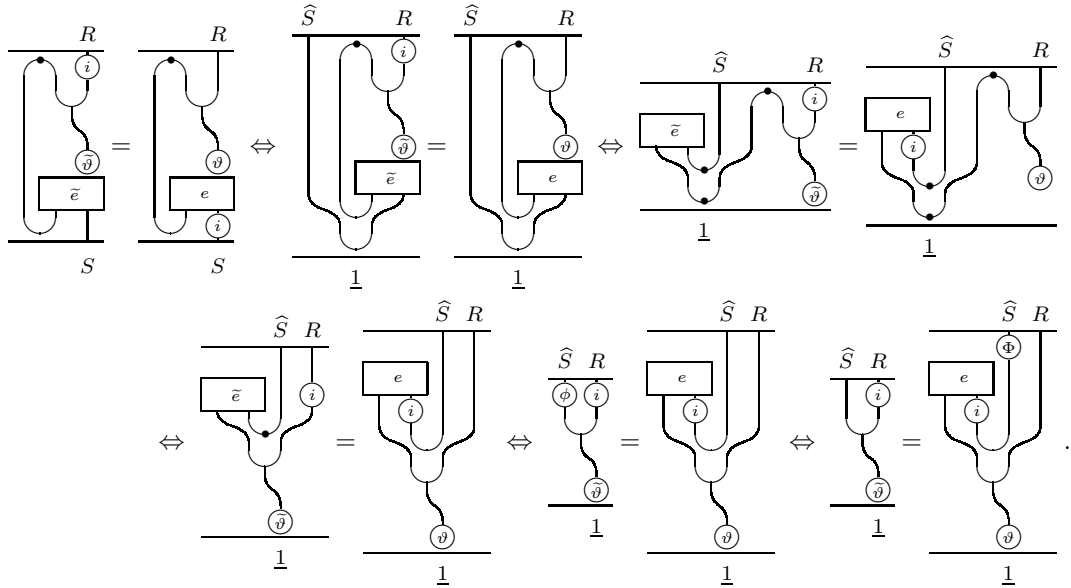
\begin{eqnarray*}
&&\hspace*{-1cm}{\footnotesize
\gbeg{3}{8}
\gvac{2}\got{1}{R}\gnl
\gvac{1}\gnot{\hspace*{-4mm}\vspace*{2mm}\bullet}\gvac{-1}\gdb\gmp{i}\gnl
\gcl{1}\gmu\gnl
\gcl{1}\gcn{1}{1}{2}{3}\gnl
\gcl{1}\gvac{1}\gmpcu{\widetilde{\vartheta}}\gnl
\gcl{1}\gsbox{2}\gnot{\hspace{5mm}\widetilde{e}}\gnl
\gev\gcl{1}\gnl
\gvac{2}\gob{1}{S}
\gend
}=
{\footnotesize
\gbeg{3}{8}
\gvac{2}\got{1}{R}\gnl
\gvac{1}\gnot{\hspace*{-4mm}\vspace*{2mm}\bullet}\gvac{-1}\gdb\gcl{1}\gnl
\gcl{1}\gmu\gnl
\gcl{1}\gcn{1}{1}{2}{3}\gnl
\gcl{1}\gvac{1}\gmpcu{\vartheta}\gnl
\gcl{1}\gsbox{2}\gnot{\hspace{5mm}e}\gnl
\gev\gmp{i}\gnl
\gvac{2}\gob{1}{S}
\gend}
~\Leftrightarrow~
{\footnotesize
\gbeg{4}{9}
\got{1}{\widehat{S}}\gvac{2}\got{1}{R}\gnl
\gcl{1}\gvac{1}\gnot{\hspace*{-4mm}\vspace*{2mm}\bullet}\gvac{-1}\gdb\gmp{i}\gnl
\gcl{1}\gcl{1}\gmu\gnl
\gcl{1}\gcl{1}\gcn{1}{1}{2}{3}\gnl
\gcl{1}\gcl{1}\gvac{1}\gmpcu{\widetilde{\vartheta}}\gnl
\gcl{1}\gcl{1}\gsbox{2}\gnot{\hspace{5mm}\widetilde{e}}\gnl
\gcn{1}{1}{1}{3}\gev\gcn{1}{1}{1}{-1}\gnl
\gvac{1}\gev\gnl
\gvac{1}\gob{2}{\un{1}}
\gend
}=
{\footnotesize
\gbeg{4}{9}
\got{1}{\widehat{S}}\gvac{2}\got{1}{R}\gnl
\gcl{1}\gvac{1}\gnot{\hspace*{-4mm}\vspace*{2mm}\bullet}\gvac{-1}\gdb\gcl{1}\gnl
\gcl{1}\gcl{1}\gmu\gnl
\gcl{1}\gcl{1}\gcn{1}{1}{2}{3}\gnl
\gcl{1}\gcl{1}\gvac{1}\gmpcu{\vartheta}\gnl
\gcl{1}\gcl{1}\gsbox{2}\gnot{\hspace{5mm}e}\gnl
\gcn{1}{1}{1}{3}\gev\gcn{1}{1}{1}{-1}\gnl
\gvac{1}\gev\gnl
\gvac{1}\gob{2}{\un{1}}
\gend}
~\Leftrightarrow~
{\footnotesize
\gbeg{6}{6}
\gvac{2}\got{1}{\widehat{S}}\gvac{2}\got{1}{R}\gnl
\gvac{2}\gcl{1}\gvac{1}\gnot{\hspace*{-4mm}\vspace*{2mm}\bullet}\gvac{-1}\gdb\gmp{i}\gnl
\gsbox{2}\gnot{\hspace{5mm}\widetilde{e}}\gvac{2}\gcl{1}\gcl{1}\gmu\gnl
\gcn{1}{1}{1}{3}\gvac{1}\gnot{\hspace*{-4mm}\vspace*{-2mm}\bullet}\gvac{-1}\gev\gcn{1}{1}{1}{-1}\gcn{1}{1}{2}{3}\gnl
\gvac{2}\gnot{\hspace*{-4mm}\vspace*{-2mm}\bullet}\gvac{-1}\gev\gvac{2}\gmpcu{\widetilde{\vartheta}}\gnl
\gvac{1}\gob{2}{\un{1}}
\gend
}=
{\footnotesize
\gbeg{6}{7}
\gvac{2}\got{1}{\widehat{S}}\gvac{2}\got{1}{R}\gnl
\gvac{2}\gcl{1}\gvac{1}\gnot{\hspace*{-4mm}\vspace*{2mm}\bullet}\gvac{-1}\gdb\gcl{1}\gnl
\gsbox{2}\gnot{\hspace{5mm}e}\gvac{2}\gcl{1}\gcl{1}\gmu\gnl
\gcl{1}\gmp{i}\gcl{1}\gcl{1}\gcn{1}{1}{2}{3}\gnl
\gcn{1}{1}{1}{3}\gvac{1}\gnot{\hspace*{-4mm}\vspace*{-2mm}\bullet}\gvac{-1}\gev\gcn{1}{1}{1}{-1}\gvac{1}\gmpcu{\vartheta}\gnl
\gvac{2}\gnot{\hspace*{-4mm}\vspace*{-2mm}\bullet}\gvac{-1}\gev\gnl
\gvac{1}\gob{2}{\un{1}}
\gend
}\\
&&\hspace{1cm}\Leftrightarrow~
{\footnotesize
\gbeg{4}{8}
\gvac{2}\got{1}{\widehat{S}}\got{1}{R}\gnl
\gvac{2}\gcl{1}\gcl{1}\gnl
\gsbox{2}\gnot{\hspace{5mm}\widetilde{e}}\gvac{2}\gcl{1}\gmp{i}\gnl
\gcn{1}{1}{1}{3}\gvac{1}\gnot{\hspace*{-4mm}\vspace*{-2mm}\bullet}\gvac{-1}\gev\gcn{1}{1}{1}{-1}\gnl
\gvac{1}\gmu\gnl
\gvac{1}\gcn{1}{1}{2}{3}\gnl
\gvac{2}\gmpcu{\widetilde{\vartheta}}\gnl
\gvac{2}\gob{1}{\un{1}}
\gend
}=
{\footnotesize
\gbeg{4}{9}
\gvac{2}\got{1}{\widehat{S}}\got{1}{R}\gnl
\gvac{2}\gcl{1}\gcl{1}\gnl
\gsbox{2}\gnot{\hspace{5mm}e}\gvac{2}\gcl{1}\gcl{1}\gnl
\gcl{1}\gmp{i}\gcl{1}\gcl{1}\gnl
\gcn{1}{1}{1}{3}\gev\gcn{1}{1}{1}{-1}\gnl
\gvac{1}\gmu\gnl
\gvac{1}\gcn{1}{1}{2}{3}\gnl
\gvac{2}\gmpcu{\vartheta}\gnl
\gvac{2}\gob{1}{\un{1}}
\gend
}~\Leftrightarrow~
\footnotesize{
\gbeg{2}{6}
\got{1}{\widehat{S}}\got{1}{R}\gnl
\gmp{\phi}\gmp{i}\gnl
\gmu\gnl
\gcn{1}{1}{2}{3}\gnl
\gvac{1}\gmpcu{\widetilde{\vartheta}}\gnl
\gvac{1}\gob{1}{\un{1}}
\gend}=
{\footnotesize
\gbeg{4}{9}
\gvac{2}\got{1}{\widehat{S}}\got{1}{R}\gnl
\gvac{2}\gcl{1}\gcl{1}\gnl
\gsbox{2}\gnot{\hspace{5mm}e}\gvac{2}\gcl{1}\gcl{1}\gnl
\gcl{1}\gmp{i}\gcl{1}\gcl{1}\gnl
\gcn{1}{1}{1}{3}\gev\gcn{1}{1}{1}{-1}\gnl
\gvac{1}\gmu\gnl
\gvac{1}\gcn{1}{1}{2}{3}\gnl
\gvac{2}\gmpcu{\vartheta}\gnl
\gvac{2}\gob{1}{\un{1}}
\gend
}~\Leftrightarrow~
\footnotesize{
\gbeg{2}{6}
\got{1}{\widehat{S}}\got{1}{R}\gnl
\gcl{1}\gmp{i}\gnl
\gmu\gnl
\gcn{1}{1}{2}{3}\gnl
\gvac{1}\gmpcu{\widetilde{\vartheta}}\gnl
\gvac{1}\gob{1}{\un{1}}
\gend}=
{\footnotesize
\gbeg{4}{9}
\gvac{2}\got{1}{\widehat{S}}\got{1}{R}\gnl
\gvac{2}\gmp{\Phi}\gcl{1}\gnl
\gsbox{2}\gnot{\hspace{5mm}e}\gvac{2}\gcl{1}\gcl{1}\gnl
\gcl{1}\gmp{i}\gcl{1}\gcl{1}\gnl
\gcn{1}{1}{1}{3}\gev\gcn{1}{1}{1}{-1}\gnl
\gvac{1}\gmu\gnl
\gvac{1}\gcn{1}{1}{2}{3}\gnl
\gvac{2}\gmpcu{\vartheta}\gnl
\gvac{2}\gob{1}{\un{1}}
\gend
}~.
\end{eqnarray*} 
Finally, we compute that 
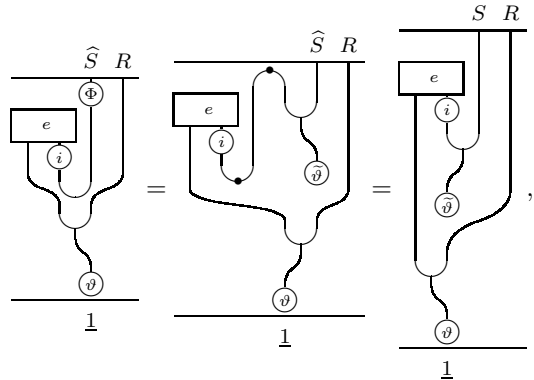
\[
{\footnotesize
\gbeg{4}{9}
\gvac{2}\got{1}{\widehat{S}}\got{1}{R}\gnl
\gvac{2}\gmp{\Phi}\gcl{1}\gnl
\gsbox{2}\gnot{\hspace{5mm}e}\gvac{2}\gcl{1}\gcl{1}\gnl
\gcl{1}\gmp{i}\gcl{1}\gcl{1}\gnl
\gcn{1}{1}{1}{3}\gev\gcn{1}{1}{1}{-1}\gnl
\gvac{1}\gmu\gnl
\gvac{1}\gcn{1}{1}{2}{3}\gnl
\gvac{2}\gmpcu{\vartheta}\gnl
\gvac{2}\gob{1}{\un{1}}
\gend
}=
{\footnotesize
\gbeg{6}{10}
\gvac{4}\got{1}{\widehat{S}}\got{1}{R}\gnl
\gvac{3}\gnot{\hspace*{-4mm}\vspace*{2mm}\bullet}\gvac{-1}\gdb\gcl{1}\gcl{1}\gnl
\gsbox{2}\gnot{\hspace{5mm}e}\gvac{2}\gcl{1}\gmu\gcl{1}\gnl
\gcl{1}\gmp{i}\gcl{1}\gcn{1}{1}{2}{3}\gvac{1}\gcl{1}\gnl
\gcl{1}\gvac{1}\gnot{\hspace*{-4mm}\vspace*{-2mm}\bullet}\gvac{-1}\gev\gvac{1}\gmpcu{\widetilde{\vartheta}}\gcl{1}\gnl
\gcn{1}{1}{1}{7}\gvac{4}\gcn{1}{1}{1}{-1}\gnl
\gvac{3}\gmu\gnl
\gvac{3}\gcn{1}{1}{2}{1}\gnl
\gvac{3}\gmpcu{\vartheta}\gnl
\gvac{3}\gob{1}{\un{1}}
\gend
}=
{\footnotesize
\gbeg{4}{12}
\gvac{2}\got{1}{S}\got{1}{R}\gnl
\gvac{2}\gcl{1}\gcl{1}\gnl
\gsbox{2}\gnot{\hspace{5mm}e}\gvac{2}\gcl{1}\gcl{1}\gnl
\gcl{1}\gmp{i}\gcl{1}\gcl{1}\gnl
\gcl{1}\gmu\gcl{1}\gnl
\gcl{1}\gcn{1}{1}{2}{1}\gvac{1}\gcl{1}\gnl
\gcl{1}\gmpcu{\widetilde{\vartheta}}\gvac{1}\gcn{1}{2}{1}{-3}\gnl
\gcl{1}\gnl
\gmu\gnl
\gcn{1}{1}{2}{3}\gnl
\gvac{1}\gmpcu{\vartheta}\gnl
\gvac{1}\gob{1}{\un{1}}
\gend},
\]
finishing the proof of the Theorem.
\end{proof}
 
\section{Applications to wreaths}\selabel{ApplWreaths}
\setcounter{equation}{0}

As we already noticed, \thref{charactFrobalg} can be also specialized either 
for a monad in an arbitrary category or 
for a monad in a 2-category. Actually the former is a particular case of the latter 
if we regard the monoidal category of endofunctors as an one object 2-category. This is why we 
restrict ourselves in presenting some characterizations for (separable) Frobenius monads in 2-categories only. 
These will be used later on for the characterization of a (separable) Frobenius wreath extension 
in a monoidal category. 

Let $U$ be 0-cell in a $2$-category ${\cal K}$. Then ${\cal K}(U):={\cal K}(U, U)$ is a monoidal
category. The objects are 1-cells $U\to U$, morphisms are $2$-cells, and the tensor product
is given by horizontal composition of $2$-cells. The unit is $1_U$, the unit 1-cell on $U$. 
With this simple observation in mind it is immediate that a monad in a 2-category 
${\cal K}$ is nothing than an algebra in a monoidal category ${\cal K}(A)$, for a certain 
0-cell $A$ of ${\cal K}$. 

Recall that a monad $(A, t, \mu, \eta)$ in ${\cal K}$ is called 
Frobenius if there exist 2-cells $\vartheta: t\Rightarrow 1_A$ and $e: 1_A\Rightarrow tt$ 
such that the diagrams below are commutative 
\[
\xymatrix{
t\ar@2{->}[r]^{1_t\odot e}\ar@2{->}[d]_-{e\odot 1_t}&ttt\ar@2{->}[d]^-{\mu\odot 1_t}\\
ttt\ar@2{->}[r]^{1_t\odot \mu}&tt
}~~,~~
\xymatrix{
1_A\ar@2{->}[r]^{e}\ar@2{->}[rd]_-{\eta}&tt\ar@2{->}[d]^-{\vartheta\odot 1_t}\\
&t
}~~,~~
\xymatrix{
1_A\ar@2{->}[r]^{e}\ar@2{->}[rd]_-{\eta}&tt\ar@2{->}[d]^-{1_t\odot \vartheta}\\
&t
}~~.
\]
A simple inspection shows that $(A, t, \mu, \eta)$ is Frobenius if and only if $(t, \mu, \eta)$ 
is a Frobenius algebra in the monoidal category ${\cal K}(A)$, so \thref{charactFrobalg} applies. 
Note that the existence of a left (right) dual for an object $u: A\ra A$ in the monoidal category 
${\cal K}(A)$ reduces to the existence of a left (right) adjunction for $u$ in the 2-categorical sense. 
Namely, $u$ has a left dual if there is a 1-cell $v: A\ra A$ and 2-cells 
$\iota:1_A\Rightarrow uv$ and $j:vu\Rightarrow 1_A$ such that 
$(1_u\odot j)(\iota\odot 1_u)=1_u$ and $(j\odot 1_v)(1_v\odot \iota)=1_v$. In this case we say also that 
$u$ is a right adjoint to $v$ and denote this adjunction as before, $(\iota, j): v\dashv u$. 
 
In what follows the vertical composition in $\mathcal{K}$ will be denoted by juxtaposition. 

\begin{corollary}\colabel{Frob2cat}
Let ${\cal K}$ be a 2-category and $\mathbb{A}=(A, t, \mu, \eta)$ a monad in ${\cal K}$. Then the following 
assertions are equivalent:
\begin{itemize}
\item[(i)] $\mathbb{A}$ is a Frobenius monad;
\item[(ii)] $(t, \mu, \eta)$ is a Frobenius algebra in the monoidal category ${\cal K}(A)$; 
\item[(iii)] $t$ admits a coalgebra structure in the monoidal category ${\cal K}(A)$, say 
$(t, \delta:t\Rightarrow tt, \va: t\Rightarrow 1_A)$, such that 
\[
(1_t\odot \mu)(\d\odot 1_t)=\d\mu=(\mu\odot 1_t)(1_t\odot \d);
\] 
\item[(iv)] There exists an adjunction $(\rho, \l): t\dashv t$ such that $\l(\mu\odot 1_t)=\l(1_t\odot \mu)$;
\item[(v)] There exists an adjunction of the form $(\rho, \vartheta\mu): t\dashv t$, with 
$\vartheta: t\Rightarrow 1_A$ a 2-cell in ${\cal K}$.
\end{itemize}
\end{corollary}
\begin{proof}
The equivalence between (i) and (ii) follows from the comments made above. The statements in 
(ii), (iii), (iv) and respectively (v) are precisely the ones in (i), (iv), (vii) and respectively (viii) 
in \thref{charactFrobalg}, specialized for the case when the monoidal category is ${\cal K}(A)$. 
We leave the details to the reader. 
\end{proof}

Of course we could add to \coref{Frob2cat} another four equivalent statements, namely the one 
corresponding to (ii), (iii), (v) and (vi) in \thref{charactFrobalg} but for later use we prefer 
to keep only those that can be formulated in terms of the monad, and so to don't involve the existence 
of a dual object for the monad that is different from the monad itself. Also, remark that (ix) in 
\thref{charactFrobalg} cannot be specialized for monads in 2-categories; it can be applied only 
in the situation when $1_A$, the unit object of ${\cal K}(A)$, is a left $\ot$-generator 
for ${\cal K}(A)$. Nevertheless, a more general treatment 
in this direction can be found in \cite{lauda}. 

In the separable case we have a monoidal interpretation for the notion, too. First, a monad 
$(A, t, \mu, \eta)$ in a 2-category $\mathcal{K}$ is called separable if the multiplication 
$\mu$ splits as a $t$-bimodule, in the sense that there is a $2$-cell 
$\gamma: t\Rightarrow tt$ in $\mathcal{K}$ such that the diagrams below are commutative  
\[
\xymatrix{
tt\ar@2{->}[r]^{1_t\odot \gamma}\ar@2{->}[d]_-{\gamma\odot 1_t}&ttt\ar@2{->}[d]^-{\mu\odot 1_t}\\
ttt\ar@2{->}[r]^{1_t\odot \mu}&tt
}~~,~~
\xymatrix{
tt\ar@2{->}[r]^{\mu}\ar@2{->}[rd]_-{1_{tt}}&t\ar@2{->}[d]^-{\gamma}\\
&tt
}~.
\]

The proof of the next result is immediate, so we will omit it. 

\begin{proposition}
For a monad $\mathbb{A}=(A, t, \mu, \eta)$ in a $2$-category $\mathcal{K}$ the following assertions 
are equivalent:
\begin{itemize}
\item[(i)] $\mathbb{A}$ is separable;
\item[(ii)] $(t, \mu, \eta)$ is a monoidal separable algebra in $\mathcal{K}(A)$;
\item[(iii)] There exists a 2-cell $e: 1_A\Rightarrow tt$ such that the diagrams below are commutative 
\[
\xymatrix{
t\ar@2{->}[r]^{1_t\odot e}\ar@2{->}[d]_-{e\odot 1_t}&ttt\ar@2{->}[d]^-{\mu\odot 1_t}\\
ttt\ar@2{->}[r]^{1_t\odot \mu}&tt
}~~,~~
\xymatrix{
1_A\ar@2{->}[r]^{e}\ar@2{->}[rd]_-{\eta}&tt\ar@2{->}[d]^-{\mu}\\
&t
}~.
\]
\end{itemize}
\end{proposition}

We move now to the Eilenberg-Moore $2$-categories. Namely, 
to a 2-category ${\cal K}$ we can associate a new 2-category $EM({\cal K})$, called the Eilenberg-Moore 
category associated to ${\cal K}$, cf. \cite{LackRoss}. The $0$-cells in $EM({\cal K})$ are monads in ${\cal K}$, 
$1$-cells are the monad morphisms and $2$-cells 
$
\xymatrix{
(f, \psi)\ar@2{->}[r]^{\rho}&(g, \phi)
}
$ 
are $2$-cells 
$
\xymatrix{
f\ar@2{->}[r]^{\rho}&gt
}
$ 
in ${\cal K}$ obeying the equality
\[
(1_g\odot\mu_t)(\rho\odot 1_t)\psi=(1_g\odot \mu_t)(\phi\odot 1_t)(1_s\odot \rho).
\] 
The vertical composition of two $2$-cells 
$
\xymatrix{
(f, \psi)\ar@2{->}[r]^{\rho}&(g, \phi)\ar@2{->}[r]^{\rho'}&(h, \gamma)
}
$ 
is given by 
\[
\xymatrix{
(f, \psi)\ar@2{->}[r]^-{\rho'\circ \rho}&(h, \gamma)
}~,~ 
\rho'\circ \rho:=(1_h\odot \mu_t)(\rho'\odot 1_t)\rho,
\]
while the horizontal composition of two cells 
\[
\xymatrix{
\mathbb{A}\rtwocell^{(f, \psi)}_{(f', \psi')}{\rho}&\mathbb{B}
\rtwocell^{(g, \phi)}_{(g', \phi')}{\rho'}&\mathbb{C}
}
\]
is defined by $(g, \phi)(f, \psi)=(gf, (1_g\odot \psi)\circ (\phi\odot 1_f))$, etc. and 
$
\xymatrix{
gf\ar@2{->}[r]^{\rho'\oslash \rho}&g'f't
}
$ 
given by 
\[
\rho'\oslash \rho:=(1_{g'}\odot 1_{f'}\odot \mu_t)(1_{g'}\odot \rho\odot 1_t) 
(1_{g'}\odot \psi)(\rho'\odot 1_f).
\]
The identity morphism of the $1$-cell $(f, \psi)$ is $1_f\odot \eta_t$, 
and for any monad $\mathbb{A}=(A, t, \mu_t, \eta_t)$ in ${\cal K}$ we have  
$(1_\mathbb{A}, i_\mathbb{A})=((1_A, 1_t), \eta_t)$. 

Motivated by the theory of entwined modules in $\Cc$-categories \cite{bc4, bc5}, we are 
interested to study when the ``algebra" extension produced by a monad in $EM({\cal K})$ 
is Frobenius, respectively separable. Note that a (co)monad in $EM({\cal K})$ is called a 
(co)wreath, so the main goal of this section is to see when an algebra extension defined by a 
wreath is Frobenius, and respectively separable. 
As we will see we can reduce this problem to the study of an algebra extension in a 
suitable monoidal category.

According to \cite{LackRoss}, a wreath is a monad 
$\mathbb{A}=(A, t, \mu, \eta)$ in ${\cal K}$ together with a $1$-cell 
$
\xymatrix{
A\ar[r]^s&A
}
$  
and $2$-cells 
$
\xymatrix{
ts\ar@2{->}[r]^\psi&st
}, 
$
$
\xymatrix{
1_A\ar@2{->}[r]^\sigma &st
} 
$ 
and 
$
\xymatrix{
ss\ar@2[r]^\zeta&st
}
$ 
satisfying the following conditions:
\begin{eqnarray}
&&(1_s\odot \mu)(\psi\odot 1_t)(1_t\odot \psi)=\psi(\mu\odot 1_s)~,~
\psi(\eta\odot 1_s)=1_s\odot \eta~;\eqlabel{wr1}\\
&&(1_s\odot \mu)(\psi\odot 1_t)(1_t\odot \sigma)=(1_s\odot \mu)(\sigma\odot 1_t)~;\eqlabel{wr2}\\
&&(1_s\odot \mu)(\psi\odot 1_t)(1_t\odot \zeta)=(1_s\odot \mu)(\zeta\odot 1_t) 
(1_s\odot \psi)(\psi\odot 1_s)~;\eqlabel{wr3}\\
&&(1_s\odot \mu)(\zeta\odot 1_t)(1_s\odot\zeta)=(1_s\odot \mu)(\zeta\odot 1_t)
(1_s\odot \psi)(\zeta\odot 1_s)~;\eqlabel{wr4}\\
&&(1_s\odot \mu)(\zeta\odot 1_t)(1_s\odot \sigma)=1_s\odot \eta~;\eqlabel{wr5}\\
&&(1_s\odot \mu)(\zeta\odot 1_t)(1_s\odot \psi)(\sigma\odot 1_s)=1_s\odot \eta~.\eqlabel{wr6} 
\end{eqnarray}

\begin{lemma}
Let ${\cal K}$ be a 2-category. Then a wreath in ${\cal K}$ is nothing that an algebra in a monoidal category of 
the form $EM({\cal K})(\mathbb{A})$, where $\mathbb{A}$ is a 0-cell in $EM({\cal K})$, that is, a monad in ${\cal K}$.   
\end{lemma}
\begin{proof}
As we already have mentioned for several times, since $EM({\cal K})$ is a 2-category it follows that 
for any 0-cell $\mathbb{A}=(A, t, \mu, \eta)$ of $EM({\cal K})$ we have a monoidal 
structure on the category $EM({\cal K})(\mathbb{A})$. Furthermore, since a wreath is a monad in the 
2-category $EM({\cal K})$ it then follows that a wreath is an algebra in a monoidal category having the form mentioned 
in the statement. 

For later use and also for the sake of 
the reader we next describe these structures explicitly. Namely,   
\begin{itemize}
\item[$\bullet$] the objects of $EM({\cal K})(\mathbb{A})$ are the 1-cells of $EM({\cal K})$ 
from $\mathbb{A}$ to $\mathbb{A}$, that is, monad morphisms $(s, \psi)$ from $\mathbb{A}$ to $\mathbb{A}$;
\item[$\bullet$] if $(s, \psi)$ and $(s', \psi')$ are monad morphisms 
from $\mathbb{A}$ to $\mathbb{A}$ then a morphism between $(s, \psi)$ and $(s', \psi')$ 
in $EM({\cal K})(\mathbb{A})$ is a 2-cell $\rho: (s, \psi)\Rightarrow (s', \psi')$ in $EM({\cal K})$;
\item[$\bullet$] the composition of two morphisms in $EM({\cal K})(\mathbb{A})$ is defined by the 
vertical composition of 2-cells in $EM({\cal K})$ and the identity morphism corresponding 
to an object $(s, \psi)$ of $EM({\cal K})(\mathbb{A})$ 
is $1_s\odot \eta$, where $\odot$ stands for the horizontal composition of 2-cells in ${\cal K}$;
\item[$\bullet$] the monoidal structure is defined by the horizontal composition of 2-cells in $EM({\cal K})$. 
More precisely, if $(s, \psi), (s', \psi')$ are objects of $EM({\cal K})(\mathbb{A})$ we then define 
\[
(s, \psi)\ot (s', \psi')=(s', \psi')(s, \psi)=(s's, (\psi'\odot 1_s)(1_{s'}\odot \psi): s'st\Rightarrow ts's),
\] 
and if $\rho: (s, \psi)\Rightarrow (f, \gamma)$ and $\rho': (s', \psi')\Rightarrow (f', \gamma')$ are 
two morphisms in $EM({\cal K})(\mathbb{A})$ we then have $\rho\ot \rho'=\rho'\oslash\rho$ as a 
2-cell in $EM({\cal K})$,
\[
\rho\ot \rho':(s, \psi)\ot (s', \psi')=(s's, (\psi'\odot 1_s)(1_{s'}\odot \psi))\Rightarrow     
(f'f, (\gamma'\odot 1_f)(1_{f'}\odot \gamma))=(f, \gamma)\ot (f', \gamma');
\]
\item[$\bullet$] the unit object of $EM({\cal K})(\mathbb{A})$ is $(1_A, 1_t)$ and for 
any object $(s, \psi)$ of $EM({\cal K})(\mathbb{A})$ we have $\Id_{(s, \psi)}=1_s\odot\eta$. 
\end{itemize} 
Now it can be easily verified that $(s, \psi):\mathbb{A}\ra \mathbb{A}$ is an object 
of $EM({\cal K})(\mathbb{A})$, i.e. a monad morphism, if and only if \equref{wr1} holds. 
Then $(s, \psi)$ has an algebra structure in $EM({\cal K})(\mathbb{A})$ if and only if 
there exist $\zeta: (s, \psi)\ot (s, \psi)\ra (s, \psi)$ and $\sigma: (1_A, 1_t)\ra (s, \psi)$ 
morphisms in $EM({\cal K})(\mathbb{A})$ such that $\zeta$ is associative and 
$\sigma$ is a unit for $\zeta$. The latest assertion can be restated in terms of the 
structure of ${\cal K}$ as follows:
\begin{itemize}
\item[$\bullet$] $\zeta$ and $\sigma$ are 2-cells in ${\cal K}$, 
$\zeta: ss\Rightarrow st$ and $\sigma: 1_A\Rightarrow st$, such that 
\equref{wr3} and \equref{wr2} hold; 
\item[$\bullet$] $\zeta$ is associative if and only if \equref{wr4} is fulfilled;
\item[$\bullet$] $\sigma$ is a unit for $\zeta$ if and only if \equref{wr5} and \equref{wr6} 
are satisfied.   
\end{itemize}
Thus our proof is complete.
\end{proof}

\begin{remark}
Using arguments similar to the ones above we get that a cowreath in ${\cal K}$, that is a 
comonad in the 2-category $EM({\cal K})$, is nothing than a coalgebra in a 
monoidal category of the form $EM({\cal K})(\mathbb{A})$, where $\mathbb{A}$ is a 
suitable monad in ${\cal K}$. 
\end{remark}

From now on we denote a wreath in $EM({\cal K})$ by 
$(A, t, \mu, \eta, s, \psi, \zeta, \sigma)$ or, shortly, by $(\mathbb{A}, s, \psi, \zeta, \sigma)$ 
in the case when the structure of the monad $\mathbb{A}=(A, t, \mu, \eta)$ is fixed from the beginning. 
Following \cite{LackRoss} to such a wreath we can associate the so-called wreath product. That is 
the monad in ${\cal K}$, 
\[
\left(A, st, 
\xymatrix{\mu_{st}: 
stst~~\ar@2{->}[r]^{1_s\odot \psi\odot 1_t}~~&sstt~~\ar@2{->}[r]^{1_s\odot 1_s\odot \mu}~~&
~~sst\ar@2{->}[r]^{\zeta\odot 1_t}~~&~~stt
\ar@2{->}[r]^{1_s\odot \mu}~~&~~st
},
\xymatrix{
\sigma: 
1_A\ar@2{->}[r]&st
}
\right).
\]  
Otherwise stated, the wreath product is a monoidal algebra in ${\cal K}(A)$. The same is 
$(t, \mu, \eta)$ and we have $\iota:=(1_s\odot \mu)(\sigma\odot 1_t): t\Rightarrow st$ 
an algebra morphism in ${\cal K}(A)$ since    
\begin{eqnarray*}
\mu_{st}(\iota\odot\iota)&=&(1_s\odot \mu)(\zeta\odot 1_t)(1_s\odot 1_s\odot \mu)(1_s\odot \mu\odot 1_t)
(1_s\odot 1_t\odot \iota)(\iota\odot 1_t)\\
&=&(1_s\odot \mu)(\zeta\odot 1_t)(1_s\odot (1_s\odot \mu(\mu\odot 1_t))(\psi\odot 1_t\odot 1_t)(1_s\odot \sigma\odot 1_t))
(\iota\odot 1_t)\\
&\equal{\equref{wr2}}&
(1_s\odot \mu)(\zeta\odot 1_t)(1_s\odot (1_s\odot\mu)((1_s\odot \mu)(\sigma\odot 1_t)\odot 1_t))(\iota\odot 1_t)\\
&=&(1_s\odot \mu)(\zeta\odot 1_t)(1_s\odot (1_s\odot\mu)(\sigma\odot 1_t)\mu)(\iota\odot 1_t)\\
&=&(1_s\odot \mu)(1_s\odot 1_t\odot \mu)(\zeta\odot 1_t\odot 1_t)(1_s\odot (\sigma\odot 1_t)\mu)(\iota\odot 1_t)\\
&=&(1_s\odot \mu)((1_s\odot \mu)(\zeta\odot 1_t)(1_s\odot \sigma)\odot 1_t)(1_s\odot \mu)(\iota\odot 1_t)\\
&\equal{\equref{wr5}}&(1_s\odot \mu)(\iota\odot 1_t)\\
&=&(1_s\odot \mu)(1_s\odot 1_t\odot \mu)(\sigma \odot 1_t\odot 1_t)=(1_s\odot \mu)(\sigma \odot 1_t)\mu=\iota\mu 
\end{eqnarray*}
and $\iota\eta=(1_s\odot \mu)(\sigma\odot 1_t)\eta=(1_s\odot \mu(1_t\odot \eta))\sigma=\sigma$. 

\begin{definition}
The canonical monad extension associated to a 
wreath $(\mathbb{A}, s, \psi, \zeta, \sigma)$ in ${\cal K}$ is the monad morphism 
$(1_A, \iota): \mathbb{A}\ra (A, st, \mu_{st}, \sigma)$ in ${\cal K}$. We call this canonical monad extension 
Frobenius, respectively separable, if ${\cal K}(A)$ admits coequalizers and any object of it is coflat, and, 
moreover, the associated algebra extension $\iota: t\Rightarrow st$ is Frobenius, respectively separable, 
in the monoidal category ${\cal K}(A)$.   
\end{definition}   

Due to the monoidal flavor of the above definition we have the following characterizations 
for the canonical monad extension associated to a wreath to be Frobenius.

\begin{theorem}\label{charactextFrob2categ}
Let $(A, t, \mu, \eta, s, \psi, \zeta, \sigma)$ be a wreath in ${\cal K}$. Then the following assertions 
are equivalent:
\begin{itemize}
\item[(i)] $(A, t, \mu, \eta, s, \psi, \zeta, \sigma)$ is a Frobenius monad in $EM({\cal K})$, that is a 
Frobenius wreath in ${\cal K}$;
\item[(ii)] $(s, \psi)$ is a Frobenius algebra in the monoidal category $EM({\cal K})(\mathbb{A})$;
\item[(iii)] $(s, \psi)$ admits a coalgebra structure in $EM({\cal K})(\mathbb{A})$ with the comultiplication 
structure morphism $(s, \psi)$-bilinear, that is there exists a cowreath structure in ${\cal K}$ of the form 
\[
(\mathbb{A}, s, \psi, 
\xymatrix{
s\ar@2{->}[r]^-{\d}&sst
}, 
\xymatrix{
s\ar@2{->}[r]^-{\va}&t
})
\] 
such that  
\begin{eqnarray*}
(1_s\odot 1_s\odot \mu)(\delta\odot 1_t)\zeta
&=&(1s\odot 1_s\odot \mu)(1_s\odot \zeta\odot 1_t)(1_s\odot 1_s\odot \psi)(\delta\odot 1_s)\\
&=&(1_s\odot 1_s\odot \mu)(1_s\odot\psi\odot 1_t)(\zeta\odot 1_s\odot 1_t)(1_s\odot \delta);
\end{eqnarray*}
\item[(iv)] There exists an adjunction $(\r, \l): (s, \psi)\dashv (s, \psi)$ in $EM({\cal K})(\mathbb{A})$ 
such that $\l$ is associative, this means, 
\[
\mu(\l\odot 1_t)(1_s\odot \psi)(\zeta\odot 1_s)=\mu(\l\odot 1_t)(1_s\odot \zeta);
\]
\item[(v)] There is an adjunction $(\r, \l): (s, \psi)\dashv (s, \psi)$ in $EM({\cal K})(\mathbb{A})$ 
with $\l$ having the form $\l=\mu(\vartheta\odot 1_t)\zeta: ss\Rightarrow t$, for some 2-cell 
$\vartheta: s\Rightarrow t$ in ${\cal K}$. 
\end{itemize}
If ${\cal K}(A)$ admits coequalizers and any object of it is coflat then (i)-(v) above are 
also equivalent to 
\begin{itemize}
\item[(vi)] The canonical monad extension associated to $(A, t, \mu, \eta, s, \psi, \zeta, \sigma)$ is Frobenius.
\end{itemize}  
\end{theorem}
\begin{proof}
The equivalences between (ii), (iii), (iv), (v) and (vi) follow from \coref{Frob2cat} applied to 
the monad $(A, t, \mu, \eta, s, \psi, \zeta, \sigma)$ in the 2-category $EM({\cal K})$. The difficult part is to 
prove the equivalence between (i) and (vi), providing that ${\cal K}(A)$ admits coequalizers and any object of it is 
coflat.

$(i)\Leftrightarrow (vi)$. We have that $\iota: t\Rightarrow st$ is a Frobenius extension if and only if there 
exist 2-cells $\vartheta: st\Rightarrow t$ and $e: 1_A\Rightarrow (st)\ot_t(st)$ in ${\cal K}$ such that 
$\vartheta$ is $t$-bilinear and the diagrams in \thref{FrobSepExt} (ii), specialized for the context provided by the 
category $\Cc={\cal K}(A)$, are commutative. 

It is obvious to check that giving a $t$-bilinear morphism $\vartheta: st\Rightarrow t$ in ${\cal K}(A)$ is equivalent 
to giving a morphism $\varsigma: s\Rightarrow t$ in ${\cal K}$ such that 
\begin{equation}\eqlabel{2tsharp1}
\mu(\varsigma\odot 1_t)\psi=\mu(1_t\odot \varsigma): ts\Rightarrow t.
\end{equation}
Indeed to $\vartheta$ corresponds $\varsigma=\vartheta(1_s\odot \eta)$, while $\vartheta$ can be recovered 
from $\varsigma$ as $\vartheta=\mu(\varsigma\odot 1_t)$. Note that \equref{2tsharp1} is imposed by the 
right $t$-linearity of $\vartheta$. 

Similarly, thinking in the monoidal sense, we have 
\[
(st)\ot_t(st)=(t\ot s)\ot_t(t\ot s)
\xymatrix{
\ar@2{->}[r]^-{\Upsilon_{t\ot s, s}}&t\ot s\ot s=sst ,
}
\]
and so to give $e$ is equivalent to give a 2-cell $\kappa: 1_A\Rightarrow sst$ in ${\cal K}$. 
Now, as in the previous proofs involving coequalizers, one can show that the pair $(\vartheta, e)$ 
is a Frobenius pair for the extension $\iota: t\Rightarrow st$ if and only if 
$\varsigma$ satisfies \equref{2tsharp1} and the following equalities hold,  
\begin{eqnarray}
&&\hspace*{-3cm}
(1_{ss}\odot \mu)(1_s\odot \psi\odot 1_t)(1_s\odot \mu\odot 1_{st})(\zeta\odot 1_{tst})(1_s\odot\psi\odot 1_{st})
(1_{st}\odot \kappa)\nonumber\\
&=&(1_{ss}\odot \mu)(1_s\odot \zeta\odot 1_t)(1_{sss}\odot \mu)(1_{ss}\odot \psi\odot 1_t)(\kappa\odot 1{st});\eqlabel{2tsharp2}\\
\sigma&=&(1_s\odot \mu)(1_s\odot\varsigma\odot 1_t)\kappa ~~
=~~(1_s\odot \mu)(\psi\odot 1_t))(\varsigma\odot 1_{st})\kappa .\eqlabel{2tsharp3}
\end{eqnarray}
We claim now that \equref{2tsharp2} is equivalent to the following two equalities:
\begin{eqnarray}
\hspace{5mm}
(1_{ss}\odot \mu)(\kappa\odot 1_t)&=&(1_{ss}\odot \mu)(1_s\odot \psi\odot 1_t)(\psi\odot 1_{st})(1_t\odot\kappa)~{\rm and}~
\eqlabel{2tsharp21}\\
\hspace{5mm}(1_{ss}\odot \mu)(1_s\odot \psi\odot 1_t)(\zeta\odot 1_{st})(1_s\odot \kappa)&=&
(1_{ss}\odot \mu)(1_s\odot \zeta\odot 1_t)(1_{ss}\odot \psi)(\kappa\odot 1_s).\eqlabel{2tsharp22}
\end{eqnarray}
This would have as a consequence the following equivalence: $(\vartheta, e)$ is a Frobenius pair 
for $\iota: t\Rightarrow st$ if and only of there exists a pair $(\varsigma: s\Rightarrow t, \kappa: 1_A\Rightarrow sst)$ 
such that \equref{2tsharp1}, \equref{2tsharp21}, \equref{2tsharp22} and \equref{2tsharp3} hold. 
But these relations have the following meaning: 
\begin{itemize}
\item[$\bullet$] \equref{2tsharp1} says that $\varsigma: (s, \psi)\Rightarrow (1_A, 1_t)$ is a morphism in $EM({\cal K})(\mathbb{A})$;
\item[$\bullet$] \equref{2tsharp21} says that $\kappa: (1_A, 1_t)\Rightarrow (s, \psi)\ot (s, \psi)=(ss, (\psi\odot 1_s)(1_s\odot \psi))$ 
is a morphism in $EM({\cal K})(\mathbb{A})$ as well;
\item[$\bullet$] \equref{2tsharp22} expresses the commutativity of the diagram 
\[
\xymatrix{
(s, \psi)\ar[r]^-{\Id_{(s, \psi)}\ot \kappa}\ar[d]_-{\kappa\ot \Id_{(s, \psi)}}&
(s, \psi)\ot (s, \psi)\ot (s, \psi)\ar[d]^-{\zeta\ot \Id_{(s, \psi)}}\\
(s, \psi)\ot (s, \psi)\ot (s, \psi)\ar[r]^-{\Id_{(s, \psi)}\ot \zeta}&(s, \psi)\ot (s, \psi)
}
\] 
in $EM({\cal K})(\mathbb{A})$, and
\item[$\bullet$] \equref{2tsharp3} can be restated in terms of the monoidal structure of $EM({\cal K})(\mathbb{A})$ as 
\[
\xymatrix{
(
(1_A, 1_t)\ar[r]^-{e}&(s, \psi)\ot (s, \psi)\ar[r]^-{\varsigma\ot \Id_{(s, \psi)}}&(s, \psi)
)
=\sigma=
(
(1_A, 1_t)\ar[r]^-{e}&(s, \psi)\ot (s, \psi)\ar[r]^-{\Id_{(s, \psi)}\ot \varsigma}&(s, \psi)
).
}
\]
Thus $(\vartheta, e)$ is a Frobenius pair for $i:t\Rightarrow st$ in ${\cal K}(A)$ if and only if 
$(\varsigma, \kappa)$ is a Frobenius pair for $(s, \psi)$ in $EM({\cal K})(\mathbb{A})$, as desired. 
\end{itemize}

So it remains to show that \equref{2tsharp2} is equivalent to \equref{2tsharp21} and \equref{2tsharp22}. 
Towards this end, we first show that by composing the both sides of \equref{2tsharp2} to the right with 
$(1_s\odot \mu)(\sigma\odot 1_t)$ we get \equref{2tsharp21}. Indeed, on one hand we have 
\begin{eqnarray*}
&&\hspace*{-1.5cm}
(1_{ss}\odot \mu)(1_s\odot \psi\odot 1_t)(1_s\odot \mu\odot 1_{st})(\zeta\odot 1_{tst})(1_s\odot\psi\odot 1_{st})
(1_{st}\odot \kappa)(1_s\odot \mu)(\sigma\odot 1_t)\\
&=&(1_{ss}\odot \mu)(1_s\odot \psi\odot 1_t)(1_s\odot \mu\odot 1_{st})(\zeta\odot 1_{tst})
(1_s\odot \psi(\mu\odot 1_s)\odot 1_{st})(1_{stt}\odot \kappa)(\sigma\odot 1_t)\\
&\equal{\equref{wr1}}&
(1_{ss}\odot \mu)(1_s\odot\psi\odot 1_t)(1_s\odot \mu(1_t\odot \mu)\odot 1_{st})\\
&&\hspace*{5mm}
(\zeta\odot 1_{ttst})
(1_s\odot (\psi\odot 1_t)(1_t\odot \psi)\odot 1_{st})(\sigma\odot 1_{stst})(1_t\odot \kappa)\\
&=&(1_{ss}\odot \mu)(1_s\odot \psi\odot 1_t)(1_s\odot \mu\odot 1_{st})\\
&&\hspace*{5mm}
((1_s\odot \mu)(\zeta\odot 1_t)(1_s\odot \psi)(\sigma\odot 1_s)\odot 1_{tst})
(\psi\odot 1_{st})(1_t\odot \kappa)\\
&\equal{\equref{wr6}}&
(1_{ss}\odot \mu)(1_s\odot \psi\odot 1_t)(\psi\odot 1_{st})(1_t\odot \kappa).
\end{eqnarray*}
On the other hand, using \equref{wr2} and \equref{wr5} we deduce that    
\[
(1_{ss}\odot \mu)(1_s\odot \zeta\odot 1_t)(1_{sss}\odot \mu)(1_{ss}\odot \psi\odot 1_t)(\kappa\odot 1{st})
(1_s\odot \mu)(\sigma\odot 1_t)=(1_{ss}\odot \mu)(\kappa\odot 1_t),
\]   
as required. In a similar manner, if we compose the both sides of \equref{2tsharp2} to the 
right with $1_s\odot \eta$ we obtain \equref{2tsharp22}, and this is essentially due to the 
second equality in \equref{wr1}.

For the converse, \equref{2tsharp21} and \equref{2tsharp22} imply \equref{2tsharp2} since 
\begin{eqnarray*}
&&\hspace*{-1.5cm}
(1_{ss}\odot \mu)(1_s\odot \zeta\odot 1_t)(1_{sss}\odot \mu)(1_{ss}\odot \psi\odot 1_t)(\kappa\odot 1_{st})\\
&=&(1_{ss}\odot \mu)((1_{ss}\odot \mu)(1_s\odot \zeta\odot 1_t)(1_{ss}\odot\psi)(\kappa\odot 1_s)\odot 1_t)\\
&\equal{\equref{2tsharp22}}&
(1_{ss}\odot \mu)(1_{sst}\odot \mu)
(1_s\odot\psi\odot 1_{tt})(\zeta\odot 1_{stt})(1_s\odot\kappa\odot 1_t)\\
&=&(1_{ss}\odot \mu)(1_s\odot\psi\odot 1_t)(\zeta\odot 1_{st})
(1_s\odot (1_{ss}\odot \mu)(\kappa\odot 1_t))\\
&\equal{\equref{2tsharp21}}&
(1_{ss}\odot \mu)(1_{ss}\odot \mu\odot 1_t)
(1_s\odot \psi\odot 1_{tt})(\zeta\odot 1_{stt})(1_{ss}\odot \psi\odot 1_t)(1_s\odot \psi\odot 1_{st})
(1_{st}\odot \kappa)\\
&=&(1_{ss}\odot \mu)(1_s\odot (1_s\odot \mu)(\psi\odot 1_t)(1_t\odot \psi)\odot 1_t)(\zeta\odot 1_{tst})
(1_s\odot\psi\odot 1_{st})(1_{st}\odot \kappa)\\
&\equal{\equref{wr1}}&
(1_{ss}\odot \mu)(1_s\odot \psi\odot 1_t)(1_s\odot \mu\odot 1_{st})(\zeta\odot 1_{tst})(1_s\odot\psi\odot 1_{st})
(1_{st}\odot \kappa),
\end{eqnarray*} 
as desired. This finishes the proof of the theorem. 
\end{proof}

We focus now on the separability case. Similar to Theorem \ref{charactextFrob2categ} we have the 
following characterizations for the canonical monad extension associated to a wreath to be separable.

\begin{theorem}\label{charactextSep2categ}
Let $(A, t, \mu, \eta, s, \psi, \zeta, \sigma)$ be a wreath in ${\cal K}$. Then the following assertions 
are equivalent:
\begin{itemize}
\item[(i)] $(A, t, \mu, \eta, s, \psi, \zeta, \sigma)$ is a separable monad in $EM({\cal K})$, that is a 
separable wreath in ${\cal K}$;
\item[(ii)] $(s, \psi)$ is a separable algebra in the monoidal category $EM({\cal K})(\mathbb{A})$.
\end{itemize}
If ${\cal K}(A)$ admits coequalizers and any object of it is coflat then (i)-(ii) above are 
also equivalent to 
\begin{itemize}
\item[(iii)] The canonical monad extension associated to $(A, t, \mu, \eta, s, \psi, \zeta, \sigma)$ is separable.
\end{itemize}
\end{theorem}
\begin{proof}
The wreath $(A, t, \mu, \eta, s, \psi, \zeta, \sigma)$ is actually a monad $(s, \psi)$ in 
$\mathrm{EM}(\mathcal{K})((A, t, \mu, \eta))$ and moreover, it is separable if and only if 
$(s, \psi)$ is so within the monoidal category $\mathrm{EM}(\mathcal{K})((A, t, \mu, \eta))$. This shows 
the equivalences between (i) and (ii). Furthermore, one can see easily that both are equivalent to 
the fact that there exists a $2$-cell  
$
\xymatrix{
(1_A, 1_t)\ar@2{->}[r]^-{e}&(ss, (1_s\odot \psi)(\psi\odot 1_s))
}
$ 
in $\mathrm{EM}(\mathcal{K})(\mathbb{A})$, that is, a $2$-cell $e: 1_A\Rightarrow sst$ 
in $\mathcal{K}$ for which the diagram 
\begin{equation}\eqlabel{sepwreath1}
\xymatrix{
t\ar@2{->}[r]^-{e\odot 1_t}\ar@2{->}[d]_-{1_t\odot e}&sstt\ar@2{->}[rr]^-{1_{ss}\odot \mu}&&sst\\
tsst\ar@2{->}[r]^-{\psi\odot 1_{st}}&stst\ar@2{->}[rr]^-{1_s\odot \psi\odot 1_t}&&sstt\ar@2{->}[u]_-{1_{ss}\odot \mu}
}
\end{equation}
is commutative, such that the following equalities hold: 
\begin{eqnarray}
(1_s\odot(1_s\odot\mu)(\zeta\odot 1_t)(1_s\odot \psi))(e\odot 1_s)&=&
(1_s\odot (1_s\odot\mu)(\psi\odot 1_t))(\zeta\odot 1_{st})(1_s\odot e), 
\eqlabel{sepwreath2}\\
(1_s\odot \mu)(\zeta\odot 1_t)e&=&\sigma .\eqlabel{sepwrath3}
\end{eqnarray}

We prove now the equivalence between (i) and (iii), so we are in the hypothesis that 
$\mathrm{EM}(\mathcal{K})(\mathbb{A})$ admits coequalizers and every object of it is coflat. 
Then the extension $\iota: t\Rightarrow st$ is separable in $\mathcal{K}(A)$ if and only if 
there exists a $2$-cell $e: 1_A:\Rightarrow (st)\ot_t(st)$ in $\mathcal{K}$ that belongs to 
$\mathcal{W}$ and obeys $\zeta^te=\sigma$, where $\mathcal{W}$ is the set defined in 
\thref{SepExtMon}, specialized for the category $\mathcal{K}(A)$. As in the proof of 
Theorem \ref{charactextFrob2categ} one can easily verify that this is equivalent to 
the existence of a $2$-cell $e: 1_A\Rightarrow sst$ in $\mathcal{K}$ such that 
\equref{sepwrath3} is satisfied and 
\begin{eqnarray}
&&\hspace*{-1cm}
(1_s\odot (1_s\odot \mu)(\zeta\odot \mu)(1_s\odot \psi\odot 1_t))(e\odot 1_{st})\nonumber\\
&&\hspace{1cm}= 
(1_s\odot (1_s\odot \mu)(\psi\odot 1_t))(\zeta\odot 1_s\odot \mu)(1_s\odot (1_s\odot \psi)(\psi\odot 1_s)\odot 1_t)
(1_{st}\odot e).\eqlabel{sepwreath1and2}
\end{eqnarray}
Actually, by \equref{wr1} it follows that \equref{sepwreath1and2} is equivalent to 
\begin{eqnarray}
&&\hspace*{-1cm}
(1_s\odot (1_s\odot \mu)(\zeta\odot \mu)(1_s\odot \psi\odot 1_t))(e\odot 1_{st})\nonumber\\
&&\hspace{1cm}=
(1_s\odot (1_s\odot \mu)(\psi\odot 1_t)(\mu\odot 1_{st}))((\zeta\odot 1_t)(1_s\odot \psi)\odot 1_{st})(1_{st}\odot e).
\eqlabel{sepwreath1and2ech}
\end{eqnarray}

We state that \equref{sepwreath1and2ech} is equivalent to \equref{sepwreath1} and \equref{sepwreath2}, and this 
would end the proof. Indeed, if we compose the both sides of \equref{sepwreath1and2ech} to the right 
with $1_s\odot \eta$ we then get \equref{sepwreath2}. To get \equref{sepwreath1} we compose the both sides 
of \equref{sepwreath1and2ech} to the right with $\iota=(1_s\odot \mu)(\psi\odot 1_t)(1_t\odot \sigma): t\Rightarrow st$. 
On one hand we have 
\begin{eqnarray*}
&&\hspace{-1.3cm}
(1_s\odot (1_s\odot \mu)(\psi\odot 1_t)(\mu\odot 1_{st}))((\zeta\odot 1_t)(1_s\odot \psi)\odot 1_{st})(1_{st}\odot e)
(1_s\odot \mu)(\psi\odot 1_t)(1_t\odot \sigma)\\
&=&(1_s\odot (1_s\odot \mu)(\psi\odot 1_t)(\mu\odot 1_{st}))(\zeta\odot 1_{tst})
(1_s\odot \psi(\mu\odot1_s)\odot 1_{st})(\psi\odot _{tsst})(1_{tst}\odot e)(1_t\odot \sigma)\\
&\equal{\equref{wr1}}&
(1_s\odot (1_s\odot \mu)(\psi\odot 1_t)(\mu\odot 1_{st})(1_t\odot \mu\odot 1_{st}))\\
&&((\zeta\odot 1_{tt})(1_s\odot \psi\odot 1_t)(1_{st}\odot \psi)(\psi\odot 1_{ts})\odot 1_{st})
(1_{tst}\odot e)(1_t\odot \sigma)\\
&=&(1_s\odot (1_s\odot\mu)(\psi\odot 1_t)(\mu\odot 1_{st}))((1_s\odot \mu)(\zeta\odot 1_t)(1_s\odot \psi)(\psi\odot 1_s)\odot 1_{tst})\\
&&(1_{ts}\odot \psi\odot 1_{st})(1_{tst}\odot e)(1_t\odot \sigma)\\
&\equal{\equref{wr3}}&
(1_s\odot (1_s\odot \mu)(\psi\odot 1_t)(\mu\odot 1_{st})(1_t\odot\mu\odot 1_{st}))
((\psi\odot 1_t)(1_t\odot \zeta)\odot 1_{tst})(1_{ts}\odot \psi\odot 1_{st})\\
&&(1_t\odot \sigma\odot 1_{sst})(1_t\odot e)\\
&=&(1_s\odot (1_s\odot \mu)(\psi\odot 1_t)(\mu\odot 1_{st}))(\psi\odot 1_{tst})\\
&&(1_t\odot (1_s\odot \mu)(\zeta\odot 1_t)(1_s\odot \psi)(\sigma\odot 1_s)\odot 1_{st})(1_t\odot e)\\
&\equal{\equref{wr6}}&
(1_s\odot (1_s\odot \mu)(\psi\odot 1_t))(\psi\odot 1_{st})(1_t\odot e).
\end{eqnarray*}
On the other hand,
\begin{eqnarray*}
&&\hspace{-1.3cm}
(1_s\odot (1_s\odot \mu)(\zeta\odot \mu)(1_s\odot \psi\odot 1_t))(e\odot 1_{st})(1_s\odot \mu)(\psi\odot 1_t)(1_t\odot \sigma)\\
&=&(1_s\odot (1_s\odot \mu)(\zeta\odot \mu))(1_{ss}\odot (\psi\odot 1_t)(1_{ts}\odot \mu))
(1_{sst}\odot (\psi\odot 1_t)(1_t\odot \sigma))(e\odot 1_t)\\
&=&(1_s\odot (1_s\odot \mu)(\zeta\odot 1_t))(1_{sss}\odot \mu)
(1_{ss}\odot (1_{st}\odot \mu)(\psi\odot 1_{tt}))(1_{sst}\odot (\psi\odot 1_t)(1_{t}\odot \sigma))(e\odot 1_t)\\
&\equal{\equref{wr1}}&
(1_s\odot (1_s\odot \mu)(\zeta\odot 1_t))
(1_{ss}\odot (1_s\odot \mu)(\psi\odot 1_t)(1_t\odot \sigma))(1_{ss}\odot \mu)(e\odot 1_t)\\
&\equal{\equref{wr2}}&
(1_{ss}\odot \mu(1_t\odot \mu))(1_s\odot\zeta\odot 1_{tt})(1_{ss}\odot (\sigma\odot 1_t)\mu)(e\odot 1_t)\\
&=&(1_{ss}\odot \mu)(1_s\odot (1_s\odot \mu)(\zeta\odot 1_t)(1_s\odot \sigma)\odot 1_t)(1_{ss}\odot \mu)(e\odot 1_t)\\
&\equal{\equref{wr5}}&
(1_{ss}\odot \mu)(e\odot 1_t).
\end{eqnarray*}
Comparing the results of the above two computations we deduce \equref{sepwreath2}, as stated. 
The converse is also true since 
\begin{eqnarray*}
&&\hspace{-1.3cm}
(1_s\odot (1_s\odot \mu)(\zeta\odot \mu)(1_s\odot \psi\odot 1_t))(e\odot 1_{st})\\
&=&(1_{ss}\odot \mu)(1_s\odot (1_s\odot \mu)(\zeta\odot 1_t)(1_s\odot \psi)\odot 1_t)(e\odot 1_{st})\\
&\equal{\equref{sepwreath2}}&
(1_{ss}\odot \mu(\mu\odot 1_t))(1_s\odot\psi\odot 1_{tt})(\zeta\odot 1_{stt})(1_s\odot e\odot 1_t)\\
&=&(1_{ss}\odot\mu)(1_s\odot\psi\odot 1_t)(\zeta\odot 1_{st})(1_s\odot (1_{ss}\odot \mu)(e\odot 1_t))\\
&\equal{\equref{sepwreath1}}&
(1_{ss}\odot \mu)(1_s\odot \psi\odot 1_t)(1_{sts}\odot \mu)(\zeta\odot 1_{sts})(1_{ss}\odot \psi\odot 1_t)
(1_s\odot \psi\odot 1_{st})(1_{st}\odot e)\\
&=&(1_{ss}\odot \mu(1_t\odot \mu))(1_s\odot (\psi\odot 1_t)(1_t\odot \psi)\odot 1_t)(\zeta\odot 1_{tst})
(1_s\odot \psi\odot 1_{st})(1_{st}\odot e)\\
&\equal{\equref{wr1}}&
(1_s\odot (1_s\odot \mu)(\psi\odot 1_t)(\mu\odot 1_{st}))((\zeta\odot 1_t)(1_s\odot \psi)\odot 1_{st})(1_{st}\odot e),
\end{eqnarray*}
as required. So our proof is complete. 
\end{proof}

We end this section by specializing Theorems \ref{charactextFrob2categ} and \ref{charactextSep2categ} 
to the case when ${\cal K}=\Cc$, a monoidal category regarded as an one object 2-category. 
It was explained in \cite{bc4} that a wreath in $\Cc$ is a pair $(A, X)$ with 
$A$ an algebra in $\Cc$ and $(X, \psi, \zeta, \sigma)$ an algebra ${\cal T}_A^\#$. 
Here ${\cal T}_A^\#$ denotes the monoidal category $EM(\Cc)(A)$, and the notation is justified by 
the fact that ${\cal T}_A^\#$ is a generalization of the category of transfer morphisms 
through the algebra $A$ in $\Cc$, ${\cal T}_A$, previously introduced by Tambara in 
\cite{tambara}. Note that $\psi: X\ot A\ra A\ot X$, $\zeta: X\ot X\ra A\ot X$ and $\sigma: \un{1}\ra A\ot X$ 
are morphisms in $\Cc$ satisfying seven compatibility relations, namely the ones in \equref{wr1}-\equref{wr6} 
specialized for this particular situation.  

For a wreath $(A, X)$ in $\Cc$ we denote by $A\#_{\psi, \zeta, \sigma}X$ the associated wreath product.  
$A\#_{\psi, \zeta, \sigma}X$ is an algebra in $\Cc$ and $\sigma$ induces an algebra morphism 
$\iota=\un{m}_A(\Id_A\ot \sigma): A\ra A\#_{\psi, \zeta, \sigma}X$ in $\Cc$. Furthermore, 
the wreath algebra $A\#_{\psi, \zeta, \sigma}X$ is also an $A$-ring in $\Cc$, that is an 
algebra in the monoidal category ${}_A^!{\cal C}_A$, providing that $A, X$ are left coflat objects. In fact, 
one of the main results in \cite{bc4} asserts that we have an $A$-ring structure on $A\ot X$ with the 
left $A$-module structure given by $\un{m}_A$ if and only if $(A, X)$ is a wreath in $\Cc$. 
Anyway, when $A\#_{\psi, \zeta, \sigma}X$ is considered as an $A$-ring we will denote it by $A\ot X$. 
Remark that this double structure is possible in view of the comments made at the beginning of the proof of 
\prref{Frobalgextcharact}.

We have now the following characterizations for Frobenius/separable wreaths in monoidal categories. 

\begin{corollary}\colabel{CharFrobWreatsMon}
Let $\Cc$ be a monoidal category and $(A, X, \psi, \zeta, \sigma)$ a wreath in $\Cc$. 
Then the following assertions are equivalent:
\begin{itemize}
\item[(i)] $(A, X, \psi, \zeta, \sigma)$ is a Frobenius wreath in $\Cc$; 
\item[(ii)] $(X, \psi)\in {\cal T}_A^\#$ is a Frobenius monoidal algebra;
\item[(iii)] $(X, \psi)$ admits a coalgebra structure $(X, \psi, \delta, f)$ in ${\cal T}_A^\#$ 
such that $\d$ is $X$-bilinear, that is there is a cowreath structure in $\Cc$ of the form 
$(A, X, \psi, \d, f)$ such that 
\[{\footnotesize
\gbeg{4}{7}
\got{1}{X}\gvac{1}\got{1}{X}\gnl
\gcl{1}\gvac{1}\gcl{1}\gnl
\gcl{1}\gsbox{3}\gnl
\gbrc\gcl{1}\gcl{2}\gnl
\gcl{1}\gsbox{2}\gnl
\gmu\gcl{1}\gcl{1}\gnl
\gob{2}{A}\gob{1}{X}\gob{1}{X}
\gend}
={\footnotesize
\gbeg{4}{7}
\got{1}{X}\got{1}{X}\gnl
\gcl{1}\gcl{1}\gnl
\gsbox{2}\gnl
\gcl{1}\gcn{1}{1}{1}{3}\gnl
\gcl{1}\gsbox{3}\gnl
\gmu\gcl{1}\gcl{1}\gnl
\gob{2}{A}\gob{1}{X}\gob{1}{X}
\gend}
={\footnotesize
\gbeg{4}{8}
\gvac{1}\got{1}{X}\gvac{1}\got{1}{X}\gnl
\gvac{1}\gcl{1}\gvac{1}\gcl{1}\gnl
\gsbox{3}\gvac{3}\gcl{1}\gnl
\gcl{1}\gcl{1}\gcl{1}\gcl{1}\gnl
\gcl{1}\gcl{1}\gsbox{2}\gnl
\gcl{1}\gbrc\gcl{1}\gnl
\gmu\gcl{1}\gcl{1}\gnl
\gob{2}{A}\gob{1}{X}\gob{1}{X}
\gend}~,~~
\mbox{where}~~
\d={\footnotesize
\gbeg{3}{5}
\got{3}{X}\gnl
\gvac{1}\gcl{1}\gnl
\gsbox{3}\gnl
\gcl{1}\gcl{1}\gcl{1}\gnl
\gob{1}{A}\gob{1}{X}\gob{1}{X}
\gend}
~~\mbox{and}~~
\zeta={\footnotesize
\gbeg{2}{5}
\got{1}{X}\got{1}{X}\gnl
\gcl{1}\gcl{1}\gnl
\gsbox{2}\gnl
\gcl{1}\gcl{1}\gnl
\gob{1}{A}\gob{1}{X}
\gend}~.
\]
\item[(iv)] There is an adjunction $(\r, \l): (X, \psi)\dashv (X, \psi)$ in ${\cal T}_A^\#$ 
with $\l$ associative or, otherwise stated, there exist morphisms 
$\rho: \un{1}\ra A\ot X\ot X$ and $\l: X\ot X\ra A$ in $\Cc$ such that the following 
relations are satisfied:
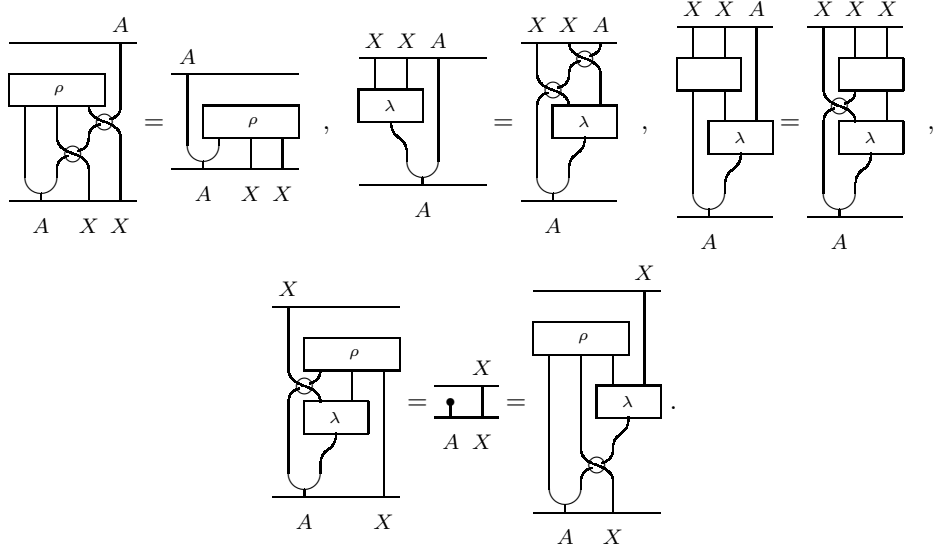
\begin{eqnarray*}
&&{\footnotesize
\gbeg{4}{7}
\gvac{3}\got{1}{A}\gnl
\gvac{2}\gvac{1}\gcl{1}\gnl
\gnot{\hspace{9mm}\rho}\gsbox{3}\gvac{3}\gcl{1}\gnl
\gcl{1}\gcl{1}\gbrc\gnl
\gcl{1}\gbrc\gcl{1}\gnl
\gmu\gcl{1}\gcl{1}\gnl
\gob{2}{A}\gob{1}{X}\gob{1}{X}
\gend}
={\footnotesize
\gbeg{4}{5}
\got{1}{A}\gnl
\gcl{1}\gnl
\gcl{1}\gnot{\hspace{9mm}\rho}\gsbox{3}\gnl
\gmu\gcl{1}\gcl{1}\gnl
\gob{2}{A}\gob{1}{X}\gob{1}{X}
\gend}
~~\hspace{2mm},~~\hspace{2mm}
{\footnotesize
\gbeg{4}{6}
\got{1}{X}\got{1}{X}\got{1}{A}\gnl
\gcl{1}\gcl{1}\gcl{3}\gnl
\gnot{\hspace{4mm}\l}\gsbox{2}\gnl
\gcn{1}{1}{2}{3}\gnl
\gvac{1}\gmu\gnl
\gvac{1}\gob{2}{A}
\gend}
={\footnotesize
\gbeg{3}{7}
\got{1}{X}\got{1}{X}\got{1}{A}\gnl
\gcl{1}\gbrc\gnl
\gbrc\gcl{1}\gnl
\gcl{1}\gnot{\hspace{4mm}\l}\gsbox{2}\gnl
\gcl{1}\gcn{1}{1}{2}{1}\gnl
\gmu\gnl
\gob{2}{A}
\gend}
~~\hspace{2mm},~~\hspace{2mm}
{\footnotesize
\gbeg{3}{8}
\got{1}{X}\got{1}{X}\got{1}{A}\gnl
\gcl{1}\gcl{1}\gcl{3}\gnl
\gsbox{2}\gnl
\gcl{1}\gcl{1}\gnl
\gcl{1}\gnot{\hspace{4mm}\l}\gsbox{2}\gnl
\gcl{1}\gcn{1}{1}{2}{1}\gnl
\gmu\gnl
\gob{2}{A}
\gend}
={\footnotesize
\gbeg{3}{8}
\got{1}{X}\got{1}{X}\got{1}{X}\gnl
\gcl{1}\gcl{1}\gcl{1}\gnl
\gcl{1}\gsbox{2}\gnl
\gbrc\gcl{1}\gnl
\gcl{1}\gnot{\hspace{4mm}\l}\gsbox{2}\gnl
\gcl{1}\gcn{1}{1}{2}{1}\gnl
\gmu\gnl
\gob{2}{A}
\gend}
~~\hspace{2mm},\\
&&\hspace{3.5cm}
{\footnotesize
\gbeg{4}{8}
\got{1}{X}\gnl
\gcl{1}\gnl
\gcl{1}\gnot{\hspace{9mm}\rho}\gsbox{3}\gnl
\gbrc\gcl{1}\gcl{4}\gnl
\gcl{1}\gnot{\hspace{4mm}\l}\gsbox{2}\gnl
\gcl{1}\gcn{1}{1}{2}{1}\gnl
\gmu\gnl
\gob{2}{A}\gvac{1}\gob{1}{X}
\gend}
={\footnotesize
\gbeg{2}{3}
\gvac{1}\got{1}{X}\gnl
\gu{1}\gcl{1}\gnl
\gob{1}{A}\gob{1}{X}\gnl
\gend}
={\footnotesize
\gbeg{4}{9}
\gvac{3}\got{1}{X}\gnl
\gvac{2}\gvac{1}\gcl{2}\gnl
\gnot{\hspace{9mm}\rho}\gsbox{3}\gnl
\gcl{1}\gcl{1}\gcl{1}\gcl{1}\gnl
\gcl{1}\gcl{1}\gnot{\hspace{4mm}\l}\gsbox{2}\gnl
\gcl{1}\gcl{1}\gcn{1}{1}{2}{1}\gnl
\gcl{1}\gbrc\gnl
\gmu\gcl{1}\gnl
\gob{2}{A}\gob{1}{X}
\gend}~.
\end{eqnarray*}
\item[(v)] There is an adjunction $(\r, \l): (X, \psi)\dashv (X, \psi)$ in ${\cal T}_A^\#$ 
with $\l$ of the form $\l=\un{m}_A(\Id_A\ot \varsigma)\zeta$, for some morphism 
$\varsigma: X\ra A$ in $\Cc$, or, in other words, there exist morphisms 
$\r: \un{1}\ra A\ot X\ot X$ and $\varsigma: X\ra A$ in $\Cc$ such that the conditions 
in (iv) above are fulfilled if we keep $\rho$ and replace $\l$ with 
$\un{m}_A(\Id_A\ot \varsigma)\zeta$.  
\end{itemize}
If $\Cc$ admits coequalizers and any object of it is coflat then (i)-(v) are equivalent to 
\begin{itemize}
\item[(vi)] $\iota: A\ra A\#_{\psi, \zeta, \sigma}X$ is an algebra Frobenius extension in $\Cc$;
\item[(vii)] $A\ot X$ is a Frobenius algebra in ${}^!_A{\cal C}_A$, i.e., a Frobenius $A$-ring, 
\end{itemize}
and if, moreover, $\un{1}$ is a left $\ot$-generator for $\Cc$ then all the statements from 
(i) to (vii) are also equivalent to  
\begin{itemize}
\item[(viii)] The functor restriction of scalars $F: \Cc_{A\#_{\psi, \zeta, \sigma}X}\ra \Cc_A$ is a 
Frobenius functor. 
\end{itemize} 
\end{corollary}     
\begin{proof}
The equivalences between (i)-(v) follow from Theorem \ref{charactextFrob2categ}, as well as 
their equivalences with (vi), providing that $\Cc$ admits coequalizers and any object of it 
is coflat (since, in the notations of Theorem \ref{charactextFrob2categ}, we have ${\cal K}(A)=\Cc$). 

From \prref{Frobalgextcharact} we get the desired equivalences with (vii), and 
in the extra hypothesis that $\un{1}$ is a left $\ot$-generator for $\Cc$ we can apply 
\thref{FrobSepExt} to get the ones with (viii), respectively. Note that, since $A\ot X$ is always 
left robust in $\Cc$, for the equivalences with (vii) we need only $A, X$ to be coflat objects.  
\end{proof}

\begin{corollary}\colabel{CharSepWreatsMon}
In the hypothesis and notations of \coref{CharFrobWreatsMon} the following statements are equivalent:
\begin{itemize}
\item[(i)] $(A, X, \psi, \zeta, \sigma)$ is a separable wreath in $\Cc$;
\item[(ii)] The algebra $(X, \psi, \sigma)$ in $\mathcal{T}_A^\#$ is separable;
\item[(iii)] There exists a morphism $e: \un{1}\ra A\ot X\ot X$ 
such that 
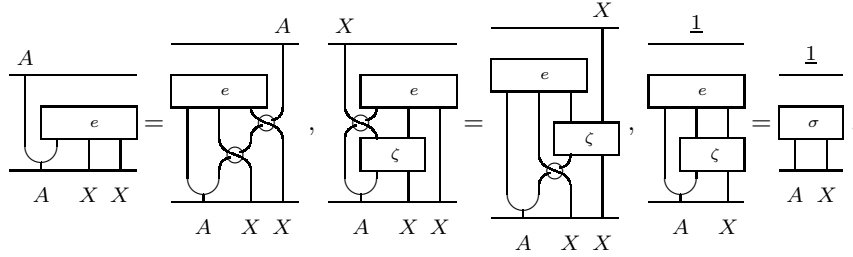
\[
{\footnotesize
\gbeg{4}{5}
\got{1}{A}\gnl
\gcl{1}\gnl
\gcl{1}\gsbox{3}\gnot{\hspace{1cm}e}\gnl
\gmu\gcl{1}\gcl{1}\gnl
\gob{2}{A}\gob{1}{X}\gob{1}{X}
\gend}
=
{\footnotesize
\gbeg{4}{7}
\gvac{3}\got{1}{A}\gnl
\gvac{3}\gcl{1}\gnl
\gsbox{3}\gnot{\hspace{1cm}e}\gvac{3}\gcl{1}\gnl
\gcl{1}\gcl{1}\gbrc\gnl
\gcl{1}\gbrc\gcl{1}\gnl
\gmu\gcl{1}\gcl{1}\gnl
\gob{2}{A}\gob{1}{X}\gob{1}{X}
\gend
}~,~
{\footnotesize
\gbeg{4}{7}
\got{1}{X}\gnl
\gcl{1}\gnl
\gcl{1}\gsbox{3}\gnot{\hspace{1cm}e}\gnl
\gbrc\gcl{1}\gcl{1}\gnl
\gcl{1}\gsbox{2}\gnot{\hspace{5mm}\zeta}\gvac{2}\gcl{1}\gnl
\gmu\gcl{1}\gcl{1}\gnl
\gob{2}{A}\gob{1}{X}\gob{1}{X}
\gend
}
=
{\footnotesize
\gbeg{4}{8}
\gvac{3}\got{1}{X}\gnl
\gvac{3}\gcl{1}\gnl
\gsbox{3}\gnot{\hspace{1cm}e}\gvac{3}\gcl{1}\gnl
\gcl{1}\gcl{1}\gcl{1}\gcl{1}\gnl
\gcl{1}\gcl{1}\gsbox{2}\gnot{\hspace{5mm}\zeta}\gnl
\gcl{1}\gbrc\gcl{1}\gnl
\gmu\gcl{1}\gcl{1}\gnl
\gob{2}{A}\gob{1}{X}\gob{1}{X}
\gend
}~,~
{\footnotesize
\gbeg{3}{7}
\gvac{1}\got{1}{\un{1}}\gnl
\gnl
\gsbox{3}\gnot{\hspace{1cm}e}\gnl
\gcl{1}\gcl{1}\gcl{1}\gnl
\gcl{1}\gsbox{2}\gnot{\hspace{5mm}\zeta}\gnl
\gmu\gcl{1}\gnl
\gob{2}{A}\gob{1}{X}
\gend}
=
\footnotesize{
\gbeg{2}{5}
\got{2}{\un{1}}\gnl
\gnl
\gsbox{2}\gnot{\hspace{5mm}\sigma}\gnl
\gcl{1}\gcl{1}\gnl
\gob{1}{A}\gob{1}{X}
\gend
}~.
\]
\end{itemize}
If $\Cc$ admits coequalizers and any object of it is coflat then (i)-(iii) are 
equivalent to 
\begin{itemize}
\item[(iv)] $\iota: A\ra A\#_{\psi, \zeta, \sigma}X$ is a separable algebra extension in $\Cc$; 
\item[(v)] $A\ot X$ is a separable algebra in ${}^!_A\Cc_A$, that is, a separable $A$-ring,
\end{itemize}
and if, moreover, $\un{1}$ is a left $\ot$-generator for $\Cc$ then (i)-(v) are also equivalent to 
\begin{itemize}
\item[(vi)] The restriction of scalars functor $F: \Cc_{A\#_{\psi, \zeta, \sigma}X}\ra \Cc_A$ 
is separable. 
\end{itemize}
\end{corollary}
\begin{proof}
The equivalences between (i), (ii) and (iv) follow by specializing Theorem \ref{charactextSep2categ} 
for the case when $\mathcal{K}$ is $\Cc$, a monoidal category. The statement (iii) is  
an explicit description of the separability morphism of a separable algebra in 
$\mathcal{T}_A^\#$. Finally, the 
equivalences of (i)-(iv) with (v) and (vi) follow because of \prref{Frobalgextcharact} and 
\thref{FrobSepExt}, respectively. Notice that, as in the Frobenius case,  
since $A\ot X$ is always left robust in $\Cc$, for the equivalences with (v) we need only $A, X$ to be 
coflat objects.  
\end{proof}

In \cite{bc5} we will apply the results in the last two corollaries to the wreath extensions 
produced by generalized entwining structures, previously introduced in \cite{bc4}. As a consequence 
we will obtain a set of Frobenius properties  
and Maschke-type theorems for generalized entwined module categories. Specializing them for 
the contexts provided by Hopf algebras and their generalizations we get at the end most of the 
Frobenius properties and Maschke-type theorems known for different sorts of entwined modules. We also  
get new ones, specially in the case when we consider contexts coming from quasi-Hopf algebra theory.     

\end{document}